\documentclass[a4paper,11pt,reqno]{amsart}
\usepackage{amsmath}
\usepackage{amsthm,enumerate}

\usepackage{mathrsfs}
\usepackage{hyperref}

\usepackage{amssymb}

\usepackage{bigints}

\usepackage{leftidx}

\usepackage{mathtools}
\usepackage[T1]{fontenc}
\usepackage[utf8x]{inputenc}
\usepackage[italian,english]{babel}

\usepackage{color}

\usepackage{fancyhdr}

\usepackage{calc}
\usepackage{url}

\usepackage[text={6.5in,9.0in},centering]{geometry}

\newtheorem{theorem}{Theorem}[section] 

\newtheorem{lemma}[theorem]{Lemma}
\newtheorem{proposition}[theorem]{Proposition}
\newtheorem{definition}[theorem]{Definition}
\newtheorem{remark}[theorem]{Remark}
\numberwithin{equation}{section}
\newtheorem{hypothesis}[theorem]{Hypotheses}

\renewcommand{\d}{{\mathrm d}}
\def\R{\mathbb{R}}
\def\N{\mathbb{N}}
\def\M{\mathbb{M}}
\def\P{\mathscr{P}}
\def\H{\mathbb{H}}
\def\V{\mathbb{V}}
\def\D{\mathbb{D}}
\def\Mdue{\mathscr{M}_2^M(\mathbb{M}^n)}
\def\cW{\mathcal{W}_2}

\begin{document}
\title[Wasserstein stability of porous medium-type equations]{Wasserstein stability of porous medium-type equations \\ on manifolds with Ricci curvature bounded below}
\author{Nicol\`o De Ponti, Matteo Muratori, Carlo Orrieri}

\address{Nicol\`o De Ponti: Scuola Internazionale Superiore di Studi Avanzati (SISSA), Via Bonomea 265, 34136 Trieste (Italy)}

\email{ndeponti@sissa.it}

\address{Matteo Muratori: Dipartimento di Matematica, Politecnico di Milano, Piazza Leonardo Da Vinci 32, 20133 Milano (Italy)}

\email{matteo.muratori@polimi.it}

\address{Carlo Orrieri: Dipartimento di Matematica, Universit\`a degli Studi di Pavia, Via Ferrata 5, 27100 Pavia (Italy)}

\email{carlo.orrieri@unipv.it}

\begin{abstract}
Given a complete, connected Riemannian manifold $ \mathbb{M}^n $ with Ricci curvature bounded from below, we discuss the stability of the solutions of a porous medium-type equation with respect to the 2-Wasserstein distance.
We produce (sharp) stability estimates under negative curvature bounds, which to some extent generalize well-known results by Sturm \cite{Stu} and Otto-Westdickenberg \cite{OW}.
The strategy of the proof mainly relies on a quantitative $L^1$--$L^\infty$ smoothing property of the equation considered, combined with the {Hamiltonian} approach developed by Ambrosio, Mondino and Savar\'e in a {metric-measure} setting \cite{AMS}.  
\end{abstract}

\maketitle

\tableofcontents

\section{Introduction}

In this paper we investigate the Cauchy problem for the following \textit{porous medium-type} equation:
\begin{equation}\label{pme_intro}
\begin{cases}
 \partial_t\rho = \Delta P(\rho) & \text{in } \mathbb{M}^n \times \mathbb{R}^+ \, , \\
\rho(\cdot,0) = \mu_0 \ge 0 & \text{in } \mathbb{M}^n \times \{ 0 \} \, ,
\end{cases}
\end{equation} 
where $ \mu_0 $ {is a suitable finite, nonnegative Borel measure} and $ P $ is a nonlinearity whose model case corresponds to $ P(\rho) = \rho^m $ with $ m>1 $, namely the \emph{porous medium equation} (PME for short). 
Here $\mathbb{M}^n$ is a smooth, complete, connected, $n$-dimensional ($ n \ge 2 $) Riemannian manifold endowed with the standard Riemannian distance $ \mathsf{d} $ and the Riemannian volume measure $\mathcal{V}$. 
We denote by $ \Delta $ the Laplace-Beltrami operator on $ \M^n $, which hereafter for simplicity will mostly be referred to as the ``Laplacian''. {The initial datum $ \mu_0 $ is assumed to belong to $\mathscr{M}_2^M(\mathbb{M}^n) $}, namely the space of finite, nonnegative Borel measures on $ \mathbb{M}^n $ having mass $ M $ and finite second moment, that is
$$
\mu_0(\mathbb{M}^n) = M \qquad \text{and} \qquad  \int_{\mathbb{M}^n} \mathsf{d}(x,o)^2 \, \d\mu_0(x) < \infty  
$$
for some (hence all) $ o \in \mathbb{M}^n $. 
As is well known, one can make $ \mathscr{M}_2^M(\mathbb{M}^n) $ a complete metric space by endowing it with the $2$-Wasserstein distance, 
which we will denote by  $\mathcal{W}_2$ (see Subsection \ref{W-S} for more details).

\smallskip

Our main focus is on a stability property of the evolution \eqref{pme_intro} with respect to $\mathcal{W}_2$, when $\mathbb{M}^n$ is possibly noncompact (with infinite volume) and its Ricci curvature is merely bounded from below.
This is strongly motivated by the results obtained by Sturm \cite{Stu} and Otto-Westdickenberg \cite{OW} under the nonnegativity assumption of the Ricci curvature, which we recall below. {We point out that by ``stability'' we mean the possibility to control the $ \mathcal{W}_2 $-distance between two solutions of \eqref{pme_intro}, along the flow, in terms of the $ \mathcal{W}_2 $-distance of the corresponding initial data. We will refer to this property as ``contraction'' when the $ \mathcal{W}_2 $-distance of the initial data cannot be increased by the flow.}
  
To attack the problem we have at our disposal at least two different points of view.
On the one hand, one can profit from the recent developments in the theory of nonlinear diffusion equations in non-Euclidean setting, where the connection with the geometry of the underlying structure is taken into account. 
On the other hand, the theory of optimal transportation can be employed to lift the problem to the space of measures endowed with the Wasserstein distance. The results obtained herein actually take advantage of the combination of techniques borrowed from both the two approaches.

\smallskip

For what concerns the analysis of nonlinear diffusion equations on Riemannian manifolds, we mention the following recent contributions.  
In \cite{BGV} the authors consider well-posedness and finite-time extinction phenomena for the fast-diffusion equation (i.e.\ \eqref{pme_intro} with $ P(\rho)=\rho^m $ for $ m \in (0,1)  $) on Cartan-Hadamard manifolds, namely simply connected, complete Riemannian manifolds with nonpositive sectional curvature, for sufficiently integrable initial data.
In the same geometric setting, in \cite{GMPrm} the porous medium equation is investigated when initial data are finite Borel measures, by means of potential techniques. 
Still in the Cartan-Hadamard setting and for porous medium equation, in \cite{GMPpures} the authors study well-posedness and blow-up phenomena for initial data possibly growing at infinity.  
The asymptotic behaviour for large times is addressed in \cite{GMV1, GMV2}, complementing some results previously obtained in \cite{VazHyp} in the hyperbolic space $ \H^n $. {Surprisingly, not much is known on the asymptotics of the \emph{heat equation} in $ \H^n $: we refer to \cite{VazHeat} for an account of the state of the art along with some further progress}. 

\smallskip

With regards to the theory of optimal transport, after the seminal work of Otto et al.~\cite{JKO,O} a lot of interest has been drawn in the description of certain PDEs as gradient flows in the space of probability measures endowed with the quadratic Wasserstein distance.
In fact, when associated with a convex structure, such a formulation turns out to be extremely useful to obtain existence and stability results for a large class of PDEs.
To that purpose, a very general theory of gradient flows of geodesically-convex functionals in metric spaces has rigorously been developed by Ambrosio, Gigli and Savar\'e: we refer to the monograph \cite{AGS} for a comprehensive treatment of this topic. 

Let us first briefly comment on the analysis of the heat equation (at first in $ \R^n $), for which the picture is by now quite clear.
By setting 
$$
\mathsf{E}(\mu) : =
\begin{cases} 
\int_{\R^n} \rho \log \rho \, \d x & \text{if } \d \mu = \rho(x) \d x \, , \\
+ \infty & \text{elsewhere} \, ,
\end{cases}
$$ 
that is the so-called \emph{relative entropy}, and denoting by $\mathscr{P}_2(\R^n)$ the space of probability measures with finite second moment, the following holds: for every initial datum $\mu_0 \in \mathscr{P}_2(\R^n)$ there exists a unique gradient flow of $ \mathsf{E}$ in $ \left( \mathscr{P}_2(\R^n) , \mathcal{W}_2 \right) $ in the sense of Evolution Variational Inequalities (EVI), whose trajectories coincide with the corresponding solution of the heat equation. 
The Wasserstein contraction property of the solutions is then a consequence of the \emph{displacement convexity} of $\mathsf{E}$ in $ \left( \mathscr{P}_2(\R^n) , \mathcal{W}_2 \right) $.
This result was further extended to the Riemannian setting in \cite{VRS}, see also \cite{Vil, E}, upon taking into account the Ricci curvature of the manifold $\mathbb{M}^n$: it is shown that the bound $\mathrm{Ric} \ge \lambda$ ($ \lambda \in \R $) is equivalent to both the $\lambda$-convexity of the relative entropy and the following stability property of the generated gradient flow:
\begin{equation*}
\cW(\rho(t), \hat \rho(t)) \leq e^{-\lambda t} \, \cW(\mu_0, \hat \mu_0) \qquad \forall t \geq 0 \, , 
\end{equation*}  
where the densities $ \rho $ and $ \hat{\rho} $ represent the solutions of the heat equation on $ \M^n $ starting from $ \mu_0 $ and $ \hat{\mu}_0 $, respectively. An equivalence of this form is still missing in the context of nonlinear diffusion, where only partial results can be found in the literature.

As concerns the classical porous medium equation, a gradient flow interpretation was firstly treated in \cite{O}.
Then, numerous results have subsequently been obtained in the Euclidean setting even for more general PDEs. For instance, in \cite{CMV}  the authors quantify the Wasserstein contraction for diffusion equations that may also exhibit a nonlocal structure.  
In the one-dimensional case, contraction estimates for granular-media models are obtained in \cite{LT}, by exploiting the explicit formulation of the Wasserstein distance.  
Regularizing effects and decay estimates for porous medium evolutions (with a nonlocal pressure) can be obtained by means of the minimizing movement approximation scheme in $ \left( \mathscr{P}_2(\R^n) , \mathcal{W}_2 \right) $, as is shown in \cite{LiMS}.
Finally, we refer to \cite{BC} for a simple proof of the equivalence between the contraction of the flow and the convexity condition, in which the gradient-flow structure of the problem is in fact not exploited. 
A related argument (coming from the probabilistic coupling method) can also be found in the recent manuscript \cite{FP}.

As already mentioned above, for nonlinear diffusions the passage from the Euclidean to the Riemannian setting is not straightforward.
The first contribution in this direction was given by Sturm in \cite{Stu}, where the equivalence between the geodesic convexity of the free energy and the curvature-dimension conditions is shown.
In this setting, stability estimates for the PME on \emph{nonnegatively} curved manifolds are still a consequence of the geodesic convexity of the free energy, thus complementing, when $\mathrm{Ric} \geq 0$, the results of \cite{VRS} in the linear case.
More precisely, the gradient-flow structure of the PME on $\mathscr{P}_2(\M^n)$ can be derived by introducing the {free energy}
\begin{equation}\label{free}
\mathsf{\tilde E}(\mu) : =
\begin{cases} 
\int_{\M^n} U(\rho) \, \d \mathcal{V} & \text{if } \d \mu = \rho \, \d \mathcal{V} \, , \\
+ \infty & \text{elsewhere} \, ,
\end{cases}
\end{equation}
where $U$ is linked to the nonlinearity of the equation through the relation $P(\rho) = \rho U'(\rho)  - U(\rho)$.
When $\M^n$ satisfies $\mathrm{Ric}_x \geq 0$ for every $x \in \M^n$, it is shown that under the additional assumption $ \rho U'(\rho) \geq \left(1 - 1/n\right) U(\rho) $, the following contraction property holds along the flow:
\begin{equation*}
\frac{\d}{\d t} \cW(\rho(t), \hat \rho(t)) \le 0 \qquad \forall t \ge 0 \, .
\end{equation*}  
Furthermore, the conditions on $U$ and Ricci turn out to be also necessary for the contraction to hold, and they are equivalent to the displacement convexity of the functional $\mathsf{\tilde E}$:
\begin{equation*}
\mathsf{\tilde E}(\mu^s) \leq (1-s)\mathsf{\tilde E}\!\left(\mu^0\right) + s \mathsf{\tilde E}\!\left(\mu^1\right)
\end{equation*} 
for every $2$-Wasserstein geodesic $\lbrace \mu^s \rbrace_{0 \leq s \le 1} \subset \left( \mathscr{P}_2(\M^n) , \cW \right) $.

Let us recall that the above argument was subsequently revisited in the compact setting by Otto and Westdickenberg in \cite{OW} through the so-called \emph{Eulerian} calculus. Recent developments have also been obtained in \cite{OT} in the context of weighted Riemannian and Finsler manifolds.

\medskip

Our main goal is to obtain stability estimates for the porous medium-type evolution \eqref{pme_intro} without imposing the nonnegativity of the Ricci curvature. To that purpose, we need to introduce some key hypotheses both on the manifold and on the form of the nonlinearity we consider.

First of all, we assume that $\mathbb{M}^{n}$ ($ n \ge 3 $) supports the following \emph{Sobolev-type} inequality:
\begin{equation}\label{Sob-intro}
\left\| f \right\|_{L^{2^\star}\!(\mathbb{M}^n)} \le C_S \left( \left\| \nabla f \right\|_{L^{2}(\mathbb{M}^n)} + \left\| f \right\|_{L^{2}(\mathbb{M}^n)} \right) \quad \forall f \in W^{1,2}(\mathbb{M}^n) \, , \qquad 2^\star := \frac{2n}{n-2} \, ,
\end{equation}
and has {Ricci curvature} bounded from below, that is
\begin{equation}\label{K-Heb}
\mathrm{Ric}_x \ge -K \qquad \forall x \in \mathbb{M}^n 
\end{equation}
for some constant $K \ge 0 $, in the sense of quadratic forms. 
Note that 
\eqref{Sob-intro} is guaranteed on any complete, $ n $-dimensional ($ n \ge 3 $) Riemannian manifold satisfying \eqref{K-Heb} along with the \emph{noncollapse} condition, see Section \ref{sol-gen}. The $2$-dimensional case can also be dealt with by means of minor modifications: we refer to Remark \ref{2d}. For what concerns the nonlinearity, we assume $ P $ to be a $C^1([0,+\infty))$, strictly increasing function satisfying $P(0) = 0$ and the two-sided bound
\begin{equation}\label{power-assumption}
c_0 \, m \, \rho^{m-1} \le P^\prime (\rho) \le c_1 \, m \, \rho^{m-1}  \qquad \forall \rho \ge 0 \,  ,
\end{equation}
for some $ c_1 \ge c_0 >0 $ and $ m>1 $. In fact the requirement $m >1$ corresponds to the so-called \emph{slow diffusion} regime.
Furthermore, it will also be crucial to ask that $ P $ complies with the \emph{McCann} condition
\begin{equation}\label{fund-ineq-intro}  
\rho \, P^\prime\!(\rho) - \left(1 - \tfrac{1}{n}\right)P(\rho) \ge 0  \qquad \forall \rho \ge 0 \, .
\end{equation} 
Let us observe that the \textit{pure} porous medium nonlinearity, namely $ P(\rho)=\rho^m $, obviously complies with \eqref{power-assumption} and \eqref{fund-ineq-intro}.

In our main result, that is Theorem \ref{main-result}, we show that under the above conditions problem \eqref{pme_intro} admits a unique solution in an appropriate weak sense (see again Section \ref{sol-gen} for more details). Moreover, for any pair of initial data $\mu_0,\hat{\mu}_0 \in \mathscr{M}_2^M(\M^n) $, the corresponding solutions $\mu(t)=\rho(t)\mathcal{V} $ and $ \hat{\mu}(t)=\hat{\rho}(t)\mathcal{V}$ have a (bounded) density for every $t>0$ and satisfy the following stability estimate with respect to the 2-Wasserstein distance:
\begin{equation}\label{contr_intro}
\mathcal{W}_2 \!\left(\rho(t),\hat{\rho}(t) \right) \leq \exp\!\left\{ K \, c_1 \, \mathfrak{C}_{m} \left[ \left(tM^{m-1}\right)^{\frac{2}{2+n(m-1)}}  \vee \left(tM^{m-1}\right) \right] \right\} \mathcal{W}_2\! \left( \mu_0 ,\hat{\mu}_0 \right) \qquad \forall t > 0 \, ,
\end{equation}
where a semi-explicit form of the constant $\mathfrak{C}_{m}>0$ is also given. Estimate \eqref{contr_intro} seems to be new in the context of diffusion equations on manifolds, due to the presence of a nonlinear time power in the exponent. Moreover, in Theorem \ref{optimal} we exhibit a nontrivial example that shows that our estimate is indeed optimal (for small times).
Precisely, in the $n$-dimensional hyperbolic space $\mathbb{H}^n_K$ of constant sectional curvature $-K$ (thus of Ricci curvature $ -(n-1)K $), given two close enough points $x,y \in \mathbb{H}^n_K$  and the associated Dirac measures $\mu_0= M \delta_x $, $ \hat{\mu}_0=M\delta_y$, there holds
\begin{equation}\label{opt_intro}
\mathcal{W}_2 \!\left(\rho(t),\hat{\rho}(t) \right) \ge \left[ 1 + K \, \kappa \left( t M^{m-1} \right)^{\frac{2}{2+n(m-1)}}  \right] \mathcal{W}_2\! \left( \mu_0 ,\hat{\mu}_0 \right)  \qquad \forall t \in (0,\overline{t}) \, ,
\end{equation}
for a suitable constant $\kappa = \kappa(n,m)>0$ and a sufficiently small time $\overline{t} >0$. As a consequence, we can deduce that the PME \emph{is not a gradient flow} with respect to $\cW$ on $\H_K^n$, or more generally on negatively-curved manifolds, {in the sense of \emph{Evolution Variational Inequalities}}.
We refer to Remark \ref{nonlin-long} for further details.

\subsection{Strategy}
The strategy we adopt has its roots in the so called \textit{Eulerian} approach employed in \cite{OW, DS} and subsequently in \cite{AMS}. 
Instead of relying on existence and smoothness of the optimal transport map, the main insight of the Eulerian approach is to directly work in the subspace of smooth densities and to take advantage of the Benamou-Brenier formulation of the Wasserstein distance.   
The basic idea is to link the contraction property of the Wasserstein distance to the monotonicity of the associated Lagrangian. Moreover, as is discussed in greater detail in \cite{AMS}, the contraction of the distance under the action of the flow is also equivalent to the monotonicity of the associated Hamiltonian functional (in the sense of Fenchel duality). 
Such equivalence turns out to be more convenient in the context of porous medium flows; we give here a flavor of the strategy, referring to Section \ref{sec:str} for a more complete discussion. Let us start by writing the $2$-Wasserstein distance as an action functional of the following form:
\begin{equation*}
\frac{1}{2}\cW(\rho_0, \hat\rho_0)  = \inf \left\lbrace \int_0^1 \mathcal{L}\!\left(\rho^s, \tfrac{\d}{\d s} \rho^s\right)  \d s: \ s \mapsto \rho^s \, \text{ with } \rho^0 = \rho_0 \, , \ \rho^1 = \hat \rho_0  \right\rbrace ,
\end{equation*} 
where 
\begin{equation*}
\mathcal{L}(\rho, w) = \frac{1}{2} \int_{\M^n} |\nabla \phi|^2 \rho\, \d \mathcal{V} \, , \qquad  -\mathrm{div}(\rho \nabla \phi) = w \quad \text{in } \mathbb{M}^n \, . 
\end{equation*}  
Rather than looking directly at the Lagrangian $\mathcal{L}$, we consider the Hamiltonian functional
\begin{equation*}
\mathcal{E}_\rho[\phi] := \frac{1}{2} \int_{\M^n} |\nabla \phi	|^2 \rho \, \d \mathcal{V} \, .
\end{equation*}
If $ \rho \equiv \rho(t) $ is a solution of \eqref{pme_intro} and $\phi \equiv \phi(t) $ is the solution of the corresponding \emph{linearized} backward flow given by $\frac{\d }{\d t} \phi = - P'(\rho) \Delta \phi$,
it is not difficult to check that, at least formally, there holds (see \cite[Example~2.4]{AMS}) 
\begin{equation*}
\frac{\d }{\d t} \mathcal{E}_{\rho(t)}[\phi(t)] =  \int_{\M^n} P(\rho(t)) \, \Gamma_2(\phi(t)) \, \d \mathcal{V}  + \int_{\M^n} \left[ \rho(t) P'(\rho(t)) - P(\rho(t)) \right] \left( \Delta \phi(t) \right)^2 \d \mathcal{V} \, , 
\end{equation*}
where $\Gamma_2$ is the iterated \emph{carr\'e du champ} operator, whose definition is provided in Subsection \ref{B-E}. 
By exploiting \eqref{fund-ineq-intro} and the Bakry-\'Emery formulation $\mathrm{BE}(0,n)$ of the curvature bound $\mathrm{Ric} \geq 0$ (we refer again to Subsection \ref{B-E}), one can deduce the monotonicity of the Hamiltonian along the flow, namely $\frac{\d }{\d t} \mathcal{E}_{\rho(t)}[\phi(t)] \geq 0$, which is a key step in order to prove the $2$-Wasserstein contraction property (see \cite[Proposition 2.1]{AMS} in a simplified framework).

\smallskip

However, in the present setting we are dealing with the more general case in which the Ricci curvature is merely bounded from below. 
As a consequence, by employing the Bakry-\'Emery formulation $\mathrm{BE}(-K,n)$, a priori we only have
\begin{equation*}
\frac12 \frac{\d}{\d t} \mathcal{E}_{\rho(t)} \! \left[ \phi(t) \right] \ge -K \, \int_{\mathbb{M}^n} |\nabla\phi(t)|^2 \, P\!\left( \rho(t) \right) \d\mathcal{V} \, .
\end{equation*}
In order to compare $\rho(t)$ with $P(\rho(t))$, and therefore to close the above differential inequality, the crucial idea is now to exploit a \emph{quantitative} $L^1(\M^n)$--$L^\infty(\M^n)$ smoothing estimate for weak energy solutions of \eqref{pme_intro}, see Proposition \ref{smooth-approx}. To that purpose, it is necessary to first understand problem \eqref{pme_intro} for more regular initial data, namely
\begin{equation}\label{pme-reg-intro}
\begin{cases}
 \partial_t\rho = \Delta P(\rho) & \text{in } \mathbb{M}^n \times \mathbb{R}^+ \, , \\
\rho(\cdot,0) = \rho_0 \ge 0 & \text{on } \mathbb{M}^n \times \{ 0 \} \, ,
\end{cases}
\end{equation}
where $ \rho_0 \in L^1(\mathbb{M}^n) \cap L^\infty(\mathbb{M}^n) $; in fact, it will also be essential to deal with a nondegenerate regularization of the equation, which will be addressed in detail in Sections \ref{basics-pme} and \ref{sec:str}. We point out that smoothing effects are a very important and well-established tool in the theory of a large class of nonlinear diffusion equations: we refer the reader e.g.~to the monograph \cite{V06}.
This way we are able to integrate the differential inequality to get the estimate
\begin{equation}\label{ham:exp_intro}
\mathcal{E}_{\rho(t)}[\phi(t)] \geq \exp\{-K\,C(t,m,n)\} \, \mathcal{E}_{\rho_0}[\phi(0)] \, ,
\end{equation}
where an explicit computation of $C(t,m,n)>0$ is available (see Lemma \ref{deriv-ham-chiusa}).
The final step consists of exploiting the dual formulation of the Wasserstein distance for suitable regular curves, and we refer to Subsection \ref{subsec: out} for a precise description of the strategy that allows one to pass from \eqref{ham:exp_intro} to the stability estimate \eqref{contr_intro}.

As for the optimality, we choose $\M^n$ as the hyperbolic space $\H_K^n$ of constant sectional curvature $-K$. 
The key ingredient to derive \eqref{opt_intro} is a delicate estimate on the Wasserstein distance between suitable radially-symmetric densities centered about two different (sufficiently close) points.
To that purpose, we take advantage of a result originally proved by Ollivier \cite{Ol} in the simpler case of uniform densities, combined with the behaviour for small times of Barenblatt solutions of the PME in $\H_K^n$, obtained in \cite{VazHyp}.
All the rigorous computations are carried out in Subsection \ref{opt-small}. 

\smallskip

Let us point out that the extension of the present results to a metric-measure setting appears not to be straightforward, mainly due to the PDE techniques we employ in Section \ref{basics-pme}. Indeed, the proof of the $ L^1$--$L^\infty$ smoothing estimate, which is a crucial ingredient of our strategy, is not directly applicable. The point is that we take advantage of a uniformly \emph{parabolic} regularization of problem \eqref{pme_intro} in smooth domains, whose analogue in the metric-measure framework is in principle not available (see Remark \ref{R}). Another key tool we use, in order to show that solutions starting from data in $ \mathscr{M}_2^M(\M^n) $ stay in $ \mathscr{M}_2^M(\M^n) $, is the so-called \emph{compact-support} property, that we establish again by pure PDE methods (see Proposition \ref{compact-support}). The counterpart of this result in metric-measure spaces should be investigated by a different approach.

\subsection{Notations}\label{sect:not}

Throughout, we will deal with a complete and connected Riemannian manifold $(\mathbb{M}^n, \mathfrak{g})$. In the sequel, for simplicity, we will omit the explicit dependence of the geometric quantities on the metric $\mathfrak{g}$. We denote by $\mathsf{d}$ the associated Riemannian distance and by $ \mathcal{V} $ the Riemannian volume measure. The former, with some abuse of notation, will also be used to denote distance between sets. The symbol $T_x\mathbb{M}^n$ will stand for the tangent space at $ x \in \mathbb{M}^n $, endowed with a scalar product $ \langle \cdot , \cdot \rangle $ that induces the norm $ |\cdot| $.

We define $\mathscr{M}(\mathbb{M}^n)$ as the space of finite, nonnegative  Borel measures over $(\mathbb{M}^n,\mathsf{d})$ and $\mathscr{M}^M(\mathbb{M}^n)$ as the space of measures $\mu\in \mathscr{M}(\mathbb{M}^n)$ such that $\mu(\mathbb{M}^n)=M>0$. If $\mu$ also has a finite second moment we write $\mu \in \mathscr{M}^M_2(\mathbb{M}^n)$, and we denote by $ \cW(\mu,\nu) $ the $ 2 $-Wasserstein distance between any two elements $ \mu,\nu \in  \mathscr{M}^M_2(\mathbb{M}^n)  $. If the measures have densities w.r.t.~$ \mathcal{V} $, say $ \rho_\mu $ and $ \rho_\nu $, we will often write $ \cW(\rho_\mu,\rho_\nu) $ in place of $ \cW(\mu,\nu) $. 

For simplicity's sake, in the following we use the notations $\H$, $\V$ and $\D$ for the Hilbert spaces 
\begin{equation}\label{spazi}
\H:= L^2(\mathbb{M}^n) \, , \qquad  \V:= W^{1,2}(\mathbb{M}^n) \, , \qquad  \mathbb{D}:=\{f\in \mathbb{V}: \, \Delta f \in \mathbb{H}\} \, , 
\end{equation}
with associated norms $ \| f \|_{\mathbb{V}}^2 := \| f \|_{\mathbb{H}}^2 + \| \nabla f \|_{\mathbb{H}}^2$ and $\|f\|^2_{\mathbb{D}}:=\|f\|^2_{\mathbb{V}}+\|\Delta f\|^2_{\mathbb{H}}$. It is useful to recall that, by an elementary cut-off argument (in case $ \mathbb{M}^n $ is noncompact), it can be shown that $\mathbb{V} $ coincides with $ W^{1,2}_0(\mathbb{M}^n) $, where the latter symbol denotes the closure of $ C_c^\infty(\mathbb{M}^n) $ with respect to $ \| \cdot \|_{\mathbb{V}} $. 

Given $T> 0$ and two Hilbert spaces $X$ and $Y$ continuously embedded in a Banach space $U$, we introduce the space of time-dependent functions
\begin{equation*}\label{w12}
W^{1,2}((0,T); X,Y):= \left\lbrace u \in W^{1,2}((0,T); U): \, u \in L^2((0,T);X) \, , \ \tfrac{ \d u}{\d t}  \in L^2((0,T);Y) \right\rbrace ,
\end{equation*}
with associated norm 
\[
\| u\|^2_{W^{1,2}((0,T); X,Y)} := \| u\|^2_{L^2((0,T); X)} + \left\| \tfrac{ \d u}{\d t} \right\|^2_{L^2((0,T);Y)} .  
\]
Let $T>0$. For any function $F \in C^1(\R)$ with $F(0) = 0$, such that $0 < \lambda \le F'(r) \leq \lambda^{-1}$ for every $r \in \R$, for some $ \lambda>0 $, in agreement with \cite[Section 3.3]{AMS} we introduce the set
\begin{equation*}\label{ND}
\mathcal{ND}_F(0,T):= \left\lbrace u \in W^{1,2}((0,T);\H) \cap C^1([0,T];\V'): \, F(u) \in L^2((0,T);\D)  \right\rbrace .
\end{equation*}

As a general rule, we will use superscripts to denote the parameter of curves that are related to geodesics in the Wasserstein space over $(\mathbb{M}^n,\mathsf{d})$ and subscripts to denote the index or parameter of an approximation. Since subscripts are also typically used to refer to initial data of a Cauchy problem as in \eqref{pme_intro} or \eqref{pme-reg-intro}, we will try to avoid ambiguity as much as possible.

Finally, when referring to a function $ \rho: D \subseteq \mathbb{M}^n \times \mathbb{R}^+ \to \mathbb{R} $ (or to a measure) evaluated at some time $t$ \emph{as a whole}, we will adopt the notation $ \rho(t) $ (or $ \mu(t) $). As for its time derivative, we will write $ \tfrac{\partial \rho}{\partial t} $ whenever it can be understood as a classical partial derivative; we will write $ \tfrac{\d \rho}{\d t} $ instead if it must be interpreted as the time derivative of $ \rho $ as a curve in a suitable Banach space. The notation $ \dot{\rho} $ will mostly be used for \emph{metric derivatives}.

\bigskip
\bigskip
\noindent \textbf{List of main notations}
\medskip

\halign{$#$\hfil\ &#\hfil\cr
\M^n & complete, connected, $n$-dimensional Riemannian manifold\cr
\H_K^n & $n$-dimensional hyperbolic space with sectional curvature $-K$ \cr
\mathcal{V} & Riemannian volume measure on $ \M^n $  \cr
\mathrm{Ric}_x& Ricci curvature at $x \in \M^n$ \cr 
T_x \, \M^n & tangent space at $ x \in \M^n $ \cr
\exp_x v & exponential map at $x\in \M^n$ along $ v \in  T_x \, \M^n$ \cr
\P(\M^n)&Borel probability measures on $\M^n$\cr
\P_2(\M^n)&Borel probability measures with finite quadratic moment\cr
\mathscr{M}_2^M(\mathbb{M}^n)& nonnegative Borel measures with mass $M$ and finite quadratic moment\cr
\cW(\mu_0,\mu_1)& Kantorovich-Rubinstein-Wasserstein distance\cr
C([0,T];X)&continuous curves from $ [0,T] $ with values in the metric space $X$ \cr
\mathrm{Lip}\!\left([0,1];X\right)& Lipschitz curves from $[0,1]$ with values in the metric space $X$ \cr
W^{k,p}(\mathbb{M}^n)&standard Sobolev spaces in $\M^n$ \cr
\H:= L^2(\mathbb{M}^n) & space of square integrable functions \cr
\V:= W^{1,2}(\mathbb{M}^n) & standard Sobolev space of order $1$ \cr 
 \mathbb{D}:= \left\{f\in \mathbb{V}: \, \Delta f \in \mathbb{H} \right\}  & space of Sobolev functions with square integrable Laplacian \cr 
{L}_{c}^\infty(\mathbb{M}^n)&bounded real functions with compact support in $\M^n$ \cr
C_b(\mathbb{M}^n) & bounded and continuous real functions in $\M^n$\cr
\mathrm{Lip}_c(\mathbb{M}^n) & Lipschitz real functions with compact support in $\M^n$ \cr
W^{1,2}_c((0,T);L^2(\mathbb{M}^n)) & Sobolev space of $L^2(\mathbb{M}^n)$-valued functions with compact support in time \cr
\Gamma, \Gamma_2, \boldsymbol{\Gamma}_2 & classical, iterated  and nonlocal \emph{carr\'e du champ} operator, respectively, see Section \ref{B-E} \cr
\mathcal{E}_\rho[f] & weighted Dirichlet energy (Hamiltonian functional), see \eqref{eq:weighted Dirichlet} \cr
\mathcal{E}^*_\rho[\ell] & dual of the Hamiltonian functional, see \eqref{eq:dual weighted Dirichlet} \cr 
Q_s\varphi & Hopf-Lax semigroup starting from $\varphi$, see \eqref{def:hopf-lax} \cr
}
\medskip

\section{Statement of the main results} \label{sol-gen}
We consider the following nonlinear diffusion equation:
\begin{equation}\label{pme}
\begin{cases}
 \partial_t\rho = \Delta P(\rho) & \text{in } \mathbb{M}^n \times \mathbb{R}^+ \, , \\
\rho(\cdot,0) = \mu_0 \ge 0 & \text{in } \mathbb{M}^n \times \{ 0 \} \, ,
\end{cases}
\end{equation}
where $ \mu_0 \in \mathscr{M}_2^M(\mathbb{M}^n)$ and $ \rho \mapsto P(\rho) $ is a suitable $ C^1([0,+\infty)) $ function of \emph{porous medium type}. We require that $ \M^n $ and $ P $ satisfy a precise set of hypotheses.

\begin{hypothesis}[Manifold]\label{h_geometric}
We assume throughout that $ \mathbb{M}^n $ ($n\geq 3$) is a smooth, complete and connected Riemannian manifold. Moreover, it will comply with either one or more of the following conditions:
\begin{itemize}
\item The Ricci curvature is uniformly bounded from below, i.e.~there exists $K \ge 0$ such that 
\begin{equation}\label{ricci-K}\tag{H1}
\mathrm{Ric}_x(v,v) \ge -K|v|^2 \qquad \forall x \in \mathbb{M}^n \text{ and } v \in T_x\mathbb{M}^n \, ;
\end{equation}
\item For some $C_S>0$ there holds the Sobolev-type inequality
\begin{equation}\label{Sob}\tag{H2}
\left\| f \right\|_{L^{2^\star}\!(\mathbb{M}^n)} \le C_S \left( \left\| \nabla f \right\|_{L^{2}(\mathbb{M}^n)} + \left\| f \right\|_{L^{2}(\mathbb{M}^n)} \right) \quad \forall f \in W^{1,2}(\mathbb{M}^n) \, , \qquad \text{with } \, 2^\star := \frac{2n}{n-2} \, .
\end{equation} 
\end{itemize}
\end{hypothesis}

A result originally due to Varopoulos \cite{Var} asserts that \eqref{Sob} does hold on any complete, $ n $-dimensional ($ n \ge 3 $) Riemannian manifold satisfying \eqref{ricci-K} along with the \emph{noncollapse} condition 
\begin{equation}\label{non-coll}
\inf_{x \in \mathbb{M}^n} \mathcal{V}\!\left(B_1(x)\right) > 0 \, ,
\end{equation}
where $ B_1(x) := \{ y \in \mathbb{M}^n : \, \mathsf{d}(x,y)<1 \} $. We refer in particular to \cite[Theorem 3.2]{Heb} (in fact $ B_1 $ could be replaced by $ B_r $ for any $ r>0 $). Condition \eqref{non-coll} is also necessary for \eqref{Sob} to hold, see \cite[Lemma 2.2]{Heb}. Note that \eqref{ricci-K} and \eqref{non-coll} are for free on any \emph{compact} Riemannian manifold, a simple subcase of the frameworks we will work within. On the other hand, if $ \mathbb{M}^n $ is noncompact and has finite volume, or more in general has an end with finite volume, then \eqref{non-coll} (and therefore \eqref{Sob}) necessarily fails. 

\smallskip

As concerns the nonlinearity $ P $ appearing in \eqref{pme}, we introduce the following set of hypotheses.  We write them separately in order to be able to single out the specific assumption(s) needed for each result we will prove.

\begin{hypothesis}[Nonlinearity]\label{h_nonlin}
We assume throughout that $ P \in C^1([0,+\infty)) $. Moreover, it will comply with either one or more of the following conditions:
\begin{equation}\label{increas}\tag{H3}
\text{$P(0)=0$ and the map $\rho \mapsto P(\rho)$ is strictly increasing} \, ;
\end{equation}
there exist $ c_1 \ge c_0 >0$ and $m >1$ such that 
\begin{equation}\label{below-above-prime}\tag{H4}
c_0 \, m \, \rho^{m-1}  \le P^\prime (\rho) \le c_1 \, m \, \rho^{m-1} \qquad \forall \rho \ge 0 \,  ,
\end{equation}
\begin{equation}\label{fund-ineq}\tag{H5}
\rho \, P^\prime\!(\rho) - \left(1 - \tfrac{1}{n}\right)P(\rho) \ge 0  \qquad  \forall \rho \ge 0 \, .
\end{equation}
\end{hypothesis}
It is straightforward to check that \eqref{fund-ineq} is implied by \eqref{below-above-prime} provided $ c_1 \le c_0 \, m \, \frac{n}{n-1}$. 

\medskip

Let us firstly notice that the choice $ P(\rho)=\rho^m $ for some $m>1$ (corresponding to the PME) obviously implies \eqref{increas}, \eqref{below-above-prime} and \eqref{fund-ineq}. We point out that condition \eqref{below-above-prime} is essential to establish the smoothing effect (see \eqref{smoothing}) and the compact-support property (see Proposition \ref{compact-support}), while \eqref{fund-ineq} is a key tool to develop the Hamiltonian approach in its abstract formulation (we refer to Lemma \ref{deriv-ham-2}). 

\medskip

We start by providing a good notion of weak solution of \eqref{pme} for initial data in $ \mathscr{M}_2^M(\mathbb{M}^n) $ and for a general nonlinearity $ P $, which is inspired by the (wide) existing literature, see Section \ref{basics-pme}.

\begin{definition}[Weak Wasserstein solutions]\label{defsol-w}
Let $ P $ comply with assumption \eqref{increas}. Given $ \mu_0 \in \mathscr{M}_2^M(\mathbb{M}^n) $, we say that a nonnegative measurable function $ \rho $ is a \emph{Wasserstein solution} of \eqref{pme} if, for every $ T > \tau>0 $, there hold
\begin{equation}\label{sol-w1}
\rho , P(\rho) \in L^2(\mathbb{M}^n \times (\tau,T)) \, , \quad \nabla P(\rho) \in L^2(\mathbb{M}^n \times (\tau,T)) \, ,
\end{equation}
\begin{equation}\label{sol-w2}
\int_0^T \int_{\mathbb{M}^n} \rho \,  \partial_t\eta \, \d\mathcal{V} \d t = \int_0^T \int_{\mathbb{M}^n} \left\langle \nabla P(\rho) \, , \nabla\eta \right\rangle \d\mathcal{V} \d t
\end{equation}
for every $ \eta \in W^{1,2}_c((0,T);L^2(\mathbb{M}^n)) $ with $ \nabla{\eta} \in L^2((0,T);L^2(\mathbb{M}^n)) $, and
\begin{equation*}\label{w-cont}
\mu \in C\!\left([0,T);\left(\mathscr{M}_2^M(\mathbb{M}^n) , \cW \right)\right) \ \text{with } \mu(0)=\mu_0 \, ,
\end{equation*}
where $ \mu(t) = \rho(t) \mathcal{V} $ for $ t>0 $.
\end{definition}

We are now in position to state our main results, which will be proved in Section \ref{sec:str}.
\begin{theorem}[Wasserstein stability]\label{main-result}
Let $\mathbb{M}^n$ ($ n \ge 3 $) comply with assumptions \eqref{ricci-K} and \eqref{Sob}.
Let moreover $ P $ comply with assumptions \eqref{increas}, \eqref{below-above-prime} and \eqref{fund-ineq}. Let $ \mu_0 \in \mathscr{M}_2^M(\mathbb{M}^n) $. Then there exists a \emph{unique} Wasserstein solution $ \rho $ of \eqref{pme}, which satisfies the smoothing estimate 
\begin{equation}\label{smoothing}
\left\| \rho(t) \right\|_{L^\infty\left(\mathbb{M}^n\right)} \le C \left( t^{-\frac{n}{2+n(m-1)}} M^{\frac{2}{2+n(m-1)}} +M \right) \qquad \forall t>0 \, ,
\end{equation}
where $ C\geq 1 $ is a constant depending only on $C_S$, $ n $, $ c_0 $ and independent of $ m $ ranging in a bounded subset of $ (1,+\infty) $. Furthermore, if $ \hat\rho $ is the Wasserstein solution of \eqref{pme} corresponding to another initial datum $ \hat{\mu}_0 \in \mathscr{M}_2^M(\mathbb{M}^n) $, the \emph{stability estimate}
\begin{equation}\label{wass-contr}
\mathcal{W}_2 \!\left(\rho(t),\hat{\rho}(t) \right) \leq \exp\!\left\{ K \, c_1 \, \mathfrak{C}_{m} \left[ \left(tM^{m-1}\right)^{\frac{2}{2+n(m-1)}}  \vee \left(tM^{m-1}\right) \right] \right\} \mathcal{W}_2\! \left( \mu_0 ,\hat{\mu}_0 \right) \qquad \forall t > 0
\end{equation}
holds, where $ \mathfrak{C}_m := C^{m-1} \, 2^{m-2} \left[ 2+n(m-1) \right] $.
\end{theorem}

In fact \eqref{wass-contr} is sharp, as $ t \downarrow 0 $, in the hyperbolic space $ \mathbb{H}^n_K $ of sectional curvature $-K$, i.e.~of Ricci curvature $ -(n-1) K $. 

\begin{theorem}[Optimality]\label{optimal}
Estimate \eqref{wass-contr} is \emph{optimal} in $ \mathbb{M}^n = \mathbb{H}^n_K $, for $ P(\rho)=\rho^m $, with the choices $ \mu_0 = M \delta_x $ and $ \hat{\mu}_0 = M \delta_y $, provided the points $ x , y \in \mathbb{H}^n_K $ are close enough. More precisely, upon setting $ \delta :=\mathsf{d}(x,y)>0 $, there exist constants $ \kappa=\kappa(n,m)>0 $, $ \overline{\delta}=\overline{\delta}(n,K,m)>0 $ and $ \overline{t}=\overline{t}(\delta,n,K,m,M)>0 $ such that if $ \delta \in (0,\overline{\delta}) $ then 
\begin{equation}\label{optimal-delta}
\mathcal{W}_2 \!\left(\rho(t),\hat{\rho}(t) \right) \ge \left[ 1 +K \, \kappa \left( t M^{m-1} \right)^{\frac{2}{2+n(m-1)}} \right] \mathcal{W}_2\! \left( \mu_0 ,\hat{\mu}_0 \right)  \qquad \forall t \in (0,\overline{t}) \, .
\end{equation}
\end{theorem}

The proof of Theorem \ref{optimal} will be provided in Subsection \ref{opt-small}. Some comments regarding both Theorem \ref{main-result} and Theorem \ref{optimal} are now in order. 

\begin{remark}[The PME, the heat equation and gradient flows] \label{nonlin-long}\rm
As mentioned above, the explicit choice $ P(\rho)=\rho^m $ corresponds to the well-known \emph{porous medium equation} (PME). In this case estimate \eqref{wass-contr} holds with $ c_1 = 1 $. In particular, if we let $ m \downarrow 1 $, thanks to the fact that $ \mathfrak{C}_m \to 1 $ we recover exactly the following stability estimate for the heat equation:
\begin{equation}\label{contr-heat}
\mathcal{W}_2 \!\left(\rho(t),\hat{\rho}(t) \right) \leq e^{K \, t}  \,\mathcal{W}_2 \!\left(\mu_0,\hat \mu_0 \right) \qquad \forall t > 0 \, .
\end{equation}
We recall that the Ricci bound \eqref{ricci-K} is \emph{equivalent} to the $(-K)$-gradient flow formulation of the heat equation with respect to the relative entropy in $ (\mathscr{P}_2(\mathbb{M}^n),\mathcal{W}_2) $, from which \eqref{contr-heat} follows: we refer to \cite[Theorem 1.1 and Corollary 1.4]{VRS} for more details. We stress that, as a byproduct of Theorem \ref{optimal}, we can deduce that in general on negatively-curved manifolds the porous medium equation \emph{cannot} be seen as the gradient flow of some $ \lambda$-convex functional with respect to the 2-Wasserstein distance, at least in the sense of Evolution Variational Inequalities (see \cite{AGS}). 
Indeed, if it was, then the estimate 
\begin{equation*}\label{contr-flow}
\mathcal{W}_2 \!\left(\rho(t),\hat{\rho}(t) \right) \leq e^{\lambda\, t}  \,\mathcal{W}_2 \!\left(\mu_0,\hat \mu_0 \right) \qquad \forall t > 0 
\end{equation*}
would hold for some $\lambda  \in \R$, thus contradicting \eqref{optimal-delta}. On the other hand, it is known that the PME \emph{can} indeed be seen as the gradient flow of the free energy \eqref{free} in the case where the Ricci curvature is nonnegative (we refer to \cite{Stu} and \cite{O,OW}), so that \eqref{wass-contr} holds with $ K=0 $.  
\end{remark}

\begin{remark}[The Cartan-Hadamard case]\rm
If, in place of \eqref{Sob}, the manifold $ \mathbb{M}^n $ supports a \emph{Euclidean} Sobolev inequality, namely 
\begin{equation}\label{Sob-euc}
\left\| f \right\|_{L^{2^\star}\!(\mathbb{M}^n)} \le C_S \left\| \nabla f \right\|_{L^{2}(\mathbb{M}^n)} \qquad \forall f \in C^1_c(\mathbb{M}^n) \, ,
\end{equation} 
then it is not difficult to deduce that \eqref{wass-contr} turns into the following estimate:
\begin{equation}\label{wass-contr-ch}
\mathcal{W}_2 \!\left(\rho(t),\hat{\rho}(t) \right) \leq \exp\!\left\{ K \, c_1 \, \mathfrak{C}_{m}  \left(tM^{m-1}\right)^{\frac{2}{2+n(m-1)}}  \right\} \mathcal{W}_2\! \left( \mu_0 ,\hat{\mu}_0 \right) \qquad \forall t > 0 \, .
\end{equation}
This is a simple consequence of our method of proof, since in that case the smoothing effect \eqref{smoothing} holds with no additional $M$ term on the right-hand side, which causes the linear term to appear at the exponent of \eqref{wass-contr}. 
Note that when $K>0$ estimate \eqref{wass-contr-ch} improves \eqref{wass-contr} for long times (the just mentioned leading linear term disappears) and remains qualitatively the same for small times.
We recall that \eqref{Sob-euc} does hold, for instance, on any \emph{Cartan-Hadamard} manifold, that is a complete, simply connected Riemannian manifold with everywhere nonpositive sectional curvature (see \cite{GMPrm} and references therein).
\end{remark}

\begin{remark}[The $2$-dimensional case]\label{2d}\rm
The results of Theorem \ref{main-result} can also be extended to the dimension $n = 2$. In that case, the Sobolev inequality should be replaced by the Gagliardo-Nirenberg inequality
\begin{equation}\label{gn-2}
\left\| f \right\|_{L^r(\mathbb{M}^2)} \le {C}_{GN} \left( \left\| \nabla{f} \right\|_{L^2(\mathbb{M}^2)} + \left\| f \right\|_{L^2(\mathbb{M}^2)}  \right)^{\frac{r-s}{r}} \left\| f \right\|_{L^s(\mathbb{M}^2)} ^{\frac{s}{r}} \qquad \forall f \in W^{1,2}(\mathbb{M}^2) \cap L^s(\mathbb{M}^2) \, ,
\end{equation}
for some $ r>s>0 $ and $ C_{GN}>0 $. We recall that, by \cite[Theorem 3.3]{BCLS}, the validity of \eqref{gn-2} for \emph{some} $r>s>0$ yields the validity of the same inequality for \emph{every} $r>s>0$. In particular, this allows us to reproduce the proof of Proposition \ref{smooth-approx} also for $n=2$, starting from \eqref{gn-2} in place of \eqref{NGN}. The rest of the results we need in order to establish Theorem \ref{main-result} also hold for $ n=2 $. Note that, again, inequality \eqref{gn-2} is satisfied (e.g.~with $ r>2 $ and $s=r-2$) on any $2$-dimensional Riemannian manifold complying with \eqref{ricci-K} and \eqref{non-coll}: this is a simple consequence of \cite[Lemma 2.1 and Theorem 3.2]{Heb}. As concerns the optimality result contained in Theorem \ref{optimal}, we just observe that its proof follows with no modifications in the case $n=2$ as well (see Subsection \ref{opt-small}). 
\end{remark}

\section{Geometric and functional preliminaries} \label{aux-tools}

In this section we recall some basic results concerning the $\Gamma$-calculus, curvature conditions, the Wasserstein distance(s) and the Hopf-Lax semigroup. We also resume a crucial density result for Wasserstein geodesics, which will be needed in the sequel. 

\subsection{The Bakry-\'Emery curvature condition}\label{B-E}

Let $(\M^n,\mathcal{B},\mathcal{V})$ be the measure space given by the 
Riemannian manifold $\mathbb{M}^n$, the $\sigma$-algebra of Borel sets $\mathcal{B}$ and the volume measure $\mathcal{V}$ associated with the metric. Given a diffusion operator $L$ on $(\M^n,\mathcal{B},\mathcal{V})$ and a suitable algebra of functions $\mathcal{A}$, it is by now standard to define the \emph{carr\'e du champ} operator
\begin{equation*}
\Gamma(f,g):= \frac{1}{2}\left(L(fg) - fLg - gLf \right), \qquad f,g \in \mathcal{A} \, ,
\end{equation*}
along with the \emph{iterated carr\'e du champ} operator
\begin{equation}\label{gamma-2-def}
\Gamma_2(f,g):= \frac{1}{2}\left(L\!\left(\Gamma(f,g)\right) - \Gamma(f,Lg) - \Gamma(g,Lf) \right), \qquad f,g \in \mathcal{A} \, .
\end{equation}
The introduction of these tools is motivated by the fact that they carry the geometric information on the measure space $(\M^n,\mathcal{B},\mathcal{V})$, being at the same time very suitable for computations. For more details, we refer the reader to the monograph \cite{BGL} and to the original paper by Bakry and \'Emery \cite{BE}.

In the present setting we fix once for all $L$ as the unique self-adjoint extension in $L^2(\M^n)$ of the Laplace-Beltrami operator $L= \Delta$. In this case, it is apparent that  
\begin{equation}\label{gamma}
\Gamma(f,g) = \left\langle \nabla f , \nabla g \right\rangle , \quad  \Gamma(f):=\Gamma(f,f)=|\nabla f|^2 \, , \quad\Gamma_2(f) := \Gamma_2(f,f) = \frac{1}{2} \, \Delta\! \left|\nabla f \right|^2 - \left\langle \nabla f , \nabla \Delta f \right\rangle .  
\end{equation} 
We recall that, thanks to \cite[Theorem 2.4]{Str}, the operator $ (-\Delta) $ defined in $ C_c^\infty(\mathbb{M}^n) $ is essentially self-adjoint on \emph{any} complete Riemannian manifold, i.e.~$\mathbb{D}$ coincides with the closure of $ C_c^\infty(\mathbb{M}^n) $ with respect to the norm $ \| \cdot \|_{\mathbb{D}} $. 

When the Ricci curvature of $\M^n$ is uniformly bounded from below by a constant $\lambda \in \R$, by applying the Bochner-Lichnerowicz formula it follows that (for all sufficiently regular function $f$)
\begin{equation}\label{BEkn}
\Gamma_2(f) \geq \lambda \Gamma(f) + \frac{1}{n}(\Delta f)^2 \, , 
\end{equation} 
which goes under the name of \emph{Bakry-\'Emery curvature-dimension condition} $\mathrm{BE}(\lambda,n)$. It is possible to show that in fact the converse implication is also true: if a Riemannian manifold $\mathbb{M}^n$ satisfies the condition $\mathrm{BE}(\lambda,N)$, then $n\leq N$ and $\mathrm{Ric} \ge \lambda$, see \cite[Subsection~1.16 and Sections~C.5,~C.6]{BGL} for further details.
\smallskip

Let us now introduce the (local) Dirichlet form $\mathcal{E}: \mathbb{H} \rightarrow [0,+\infty]$ by setting
\begin{equation}\label{standard-form}
\mathcal{E}(f):=\int_{\mathbb{M}^n} \Gamma(f) \, \mathrm{d}\mathcal{V} = \int_{\M^n} |\nabla f|^2 \, \d \mathcal{V} \, ,
\end{equation}
with proper domain $\mathbb{V}$. In addition, according to \cite{AGS2}, it is convenient to define a suitable ``integral'' version of the $\Gamma_2$ operator, in the following form:
\begin{equation}\label{gamma2}
\begin{aligned}
\boldsymbol{\Gamma}_2[f;\rho] &:= \int_{\mathbb{M}^n} \left[ \frac{1}{2} \, \Gamma(f) \, \Delta \rho - \Gamma(f, \Delta f)\,  \rho \right] \mathrm{d}\mathcal{V} \\
&= \int_{\mathbb{M}^n} \left[ \frac{1}{2} \, \Gamma(f) \, \Delta \rho + \Gamma(f,\rho) \, \Delta f + \left( \Delta f \right)^2 \rho  \right] \mathrm{d}\mathcal{V} \qquad \forall (f,\rho)\in \mathbb{D}_{\infty} \, ,
\end{aligned}
\end{equation}
where $\mathbb{D}_{\infty}$ stands for the algebra of functions defined as $
\mathbb{D}_{\infty}:=\mathbb{D}\cap L^{\infty}(\mathbb{M}^n)$. Note that, formally, \eqref{gamma2} is obtained upon choosing $ g=f $ in \eqref{gamma-2-def}, multiplying by $ \rho $ and integrating by parts.
The introduction of the multilinear form $\boldsymbol{\Gamma}_2$ provides a weak version of the Bakry-\'Emery condition: for every $(f,\rho)\in \mathbb{D}_{\infty}$ with $\rho\geq 0$ there holds
\begin{equation}\label{eq:BE-c}
\boldsymbol{\Gamma}_2[f;\rho] \ge \lambda \, \int_{\mathbb{M}^n} \Gamma(f) \, \rho \, \mathrm{d}\mathcal{V} + \frac{1}{n} \int_{\M^n} \left(\Delta f\right)^2 \rho \, \d \mathcal{V} \, . 
\end{equation} 
On a Riemannian manifold, the two formulations \eqref{BEkn} and $\eqref{eq:BE-c}$ turn out to be equivalent, and we  will refer to both of them as $\mathrm{BE}(\lambda,n)$. For a proof of such equivalence see e.g.~\cite[Subsection 2.2]{AGS2}.
\smallskip

In order to deal with ``variational'' solutions of \eqref{pme}, we will also consider a weighted version of the Dirichlet energy \eqref{standard-form}. More precisely, given $\rho\in L^{\infty}(\mathbb{M}^n)$ with $ \rho \ge 0 $, we set $\mathcal{E}_{\rho}:\mathbb{V}\rightarrow [0,+\infty)$ as
\begin{equation}\label{eq:weighted Dirichlet}
\mathcal{E}_{\rho}[f]:=\int_{\mathbb{M}^n}\Gamma(f) \, \rho \, \mathrm{d}\mathcal{V} \, .
\end{equation}
The associated dual weighted Dirichlet energy $\mathcal{E}^*_{\rho}:\mathbb{V}'\rightarrow [0,+\infty]$ is defined as
\begin{equation}\label{eq:dual weighted Dirichlet}
\frac 12 \mathcal{E}^\ast_{\rho}[\ell]:= \sup_{f\in\mathbb{V}}  {}_{\mathbb{V}'} \langle \ell,f \rangle_{\mathbb{V}} - \frac 12 \mathcal{E}_{\rho}[f] \, ,
\end{equation}
where we denoted by $\mathbb{V}'$ the dual space of $\mathbb{V}$.

\subsection{The Wasserstein space}\label{W-S} 

Let $(X,\mathsf{d})$ be a complete metric space. We say that a curve  $\gamma:[0,1]\rightarrow (X,\mathsf{d})$ belongs to $\mathrm{AC}^2([0,1];(X,\mathsf{d}))$ if there exists a function $w\in L^2((0,1))$ such that
\begin{equation}\label{def: curva AC2}
\mathsf{d}(\gamma(s),\gamma(t))\leq \int_s^t w(r)\,\d r \qquad \text{for every} \ 0\leq s\leq t\leq 1 \, . 
\end{equation}
When $\gamma \in \mathrm{AC}^2([0,1];(X,\mathsf{d}))$ its \emph{metric velocity}, defined as
\begin{equation*}
|\dot{\gamma}|(r):=\lim_{h\to 0}\frac{\mathsf{d}(\gamma(r+h),\gamma(r))}{|h|} \, ,
\end{equation*}
exists for a.e.~$r\in (0,1)$. Moreover, $|\dot\gamma| $ belongs to $ L^2((0,1))$ and provides the \emph{minimal} function $w$, up to negligible sets, such that \eqref{def: curva AC2} holds (see \cite[Theorem 1.1.2]{AGS}).

A (constant-speed) \emph{geodesic} is a curve $\gamma$ satisfying
\begin{equation*}
\mathsf{d}(\gamma(0),\gamma(1))=\int_0^1 |\dot{\gamma}|(r) \, \d r \, ,
\end{equation*}
or equivalently 
\begin{equation*}
\mathsf{d}(\gamma(s),\gamma(t))=\mathsf{d}(\gamma(0),\gamma(1))(t-s) \qquad\text{for every } 0\leq s\leq t\leq 1 \, ;
\end{equation*}
in particular, a geodesic is a Lipschitz curve. 

We say that a measure $\mu\in\mathscr{M}^M(\mathbb{M}^n)$ has finite $p$-moment, $p\geq 1$, and we write $\mu\in\mathscr{M}_p^M(\mathbb{M}^n)$, if there exists a point $o\in \mathbb{M}^n$ such that
\begin{equation*}\label{def: second moment}
\int_{\mathbb{M}^n}\mathsf{d}(x,o)^p \, \d\mu(x)<\infty \, .
\end{equation*}
We define the \emph{$p$-Wasserstein} cost between two measures $\mu^0,\mu^1\in \mathscr{M}(\mathbb{M}^n)$ as 
\begin{equation*}\label{def: p-Wasserstein}
\mathcal{W}^p_p(\mu^0,\mu^1) := \inf_\pi \int_{\mathbb{M}^n\times \mathbb{M}^n} \mathsf{d}(x,y)^p \, \d\pi(x,y) \, ,
\end{equation*}
where the infimum is taken among all the transport plans $ \pi $ between $\mu^0$ and $\mu^1$. The latter are measures $\pi\in \mathscr{M}(\mathbb{M}^n\times \mathbb{M}^n)$ such that $\pi(A\times \mathbb{M}^n)=\mu^0(A)$ and $\pi(\mathbb{M}^n\times B)=\mu^1(B)$ for every Borel sets $A,B\subset \mathbb{M}^n$. We observe that $\mathcal{W}_p(\mu^0,\mu^1)=\infty$ whenever $\mu^0(\mathbb{M}^n)\neq \mu^1(\mathbb{M}^n)$.
Another elementary fact is that 
\begin{equation}\label{elem}
\mu^0\in \mathscr{M}_p^M(\mathbb{M}^n) \quad \text{and} \quad \mathcal{W}_p(\mu^0,\mu^1)<\infty \quad \text{implies} \quad \mu^1\in \mathscr{M}_p^M(\mathbb{M}^n) \, .
\end{equation}
We are mainly interested in the cases $p=2$ and $p=1$. As regards the $1$-Wasserstein distance, we will only use these two well-known properties (see ~\cite[Chapter 6]{Vil}):
\begin{equation}\label{order-wass}
\mathcal{W}_1(\mu^0,\mu^1)\leq \mathcal{W}_2(\mu^0,\mu^1) \qquad \text{for every} \ \mu_1,\mu_2 \in \P_2(\mathbb{M}^n) \, ,
\end{equation}
\begin{equation}\label{dualita W1}
\mathcal{W}_1(\mu^0,\mu^1)=\sup\left\{\int_{\mathbb{M}^n}f \, \d \mu^1-\int_{\mathbb{M}^n} f \, \d \mu^0: \quad f:\mathbb{M}^n\rightarrow \mathbb{R} \,  , \ f \ \text{is }  \text{$1$-Lipschitz}\right\}.
\end{equation}

When $p=2$, it can be shown that for every $M\in (0,+\infty)$ the space $(\mathscr{M}_2^M(\mathbb{M}^n),\mathcal{W}_2)$ is a metric space, called the \emph{2-Wasserstein} (or simply Wasserstein) \emph{space} of mass $M$ over $\mathbb{M}^n$, which inherits many geometric properties of the ambient space $\mathbb{M}^n$. In particular, it is complete, separable and geodesic (for a proof of these facts we refer again to \cite[Chapter 6]{Vil}).

Here we will mostly work with the \emph{dual characterization} of the Wasserstein distance due to Kantorovich, which asserts that (see e.g.~\cite[Theorem 5.10(i)]{Vil}) for any $ \mu^0,\mu^1 \in \mathscr{M}^M(\mathbb{M}^n) $ there holds
\begin{equation}\label{dual-kantorovich}
\frac{1}{2} \, \mathcal{W}^2_2(\mu^0,\mu^1) = \sup_{ \substack{ \varphi,\psi \in C_b(\mathbb{M}^n): \\ \psi(x) \le \varphi(y) + \frac12 \mathrm{d}(x,y)^2 \ \forall x,y \in \mathbb{M}^n }  } \left\{ \int_{\mathbb{M}^n} \psi \, \d\mu^1 - \int_{\mathbb{M}^n} \varphi \, \d\mu^0 \right\} .
\end{equation}
From \eqref{dual-kantorovich} it is clear that, for any fixed $ \varphi $, the best possible choice of $ \psi $ is provided by 
\begin{equation}\label{Q1}
\psi(x) = Q_1\varphi(x) := \inf_{y\in \mathbb{M}^n}\varphi(y)+\frac12\,{\mathsf{d}(x,y)^2} \qquad \forall x \in \mathbb{M}^n \, .
\end{equation}
Thanks to \eqref{Q1}, by means of a cut-off argument it is not difficult to show that the supremum in \eqref{dual-kantorovich} can actually be taken over the space $ C_c(\M^n) $ of continuous and {compactly-supported} functions. A local regularization procedure then ensures that one can replace $ C_c(\M^n) $ with the space of \emph{Lipschitz} and \emph{compactly-supported} functions $ \mathrm{Lip}_c(\mathbb{M}^n) $.

In our framework it is convenient to see $ Q_1 \varphi $ as an endpoint of the \emph{Hopf-Lax} evolution semigroup starting from $ \varphi $. We recall that the latter is given by the family of maps $Q_s: \mathrm{Lip}_c(\mathbb{M}^n) \rightarrow \mathrm{Lip}_c(\mathbb{M}^n)$, $s\geq 0$, defined as
\begin{equation}\label{def:hopf-lax}
Q_s\varphi(x) := \inf_{y\in \mathbb{M}^n} \varphi(y)+\frac{\mathsf{d}(x,y)^2}{2s} \quad \forall s>0 \, ,  \qquad Q_0\varphi(x):=\varphi(x) \qquad \forall x \in \mathbb{M}^n \, .
\end{equation} 
It is readily seen that $Q_s\varphi$ satisfies
\begin{equation*}\label{eq:hopf-lax L^inf bound}
\inf_{\mathbb{M}^n}\varphi\leq Q_s\varphi(x)\leq \varphi(x) \qquad \forall s \ge 0 \, , \ \forall x \in \mathbb{M}^n \, .
\end{equation*}
More importantly, since $\M^n$ is a geodesic space, it can be shown (see \cite[Theorem 3.6]{AGSheat}) that $ (s,x) \mapsto Q_s\varphi(x)$ is the Lipschitz solution of the Hopf-Lax (or Hamilton-Jacobi) problem
\begin{equation}\label{eq:hopf-lax}
\begin{cases}
\frac{\partial}{\partial s} Q_s \varphi(x) = - \frac{1}{2} \, |\nabla Q_s \varphi |^2(x) \quad \text{for a.e. } (x,s) \in \mathbb{M}^n \times \R^+ \, , \\
Q_0 \varphi = \varphi \, .
\end{cases}
\end{equation} 

We can subsume the above discussion in the following.
\begin{proposition}\label{prop: hopf-lax duality}
Let $\mu^0,\mu^1\in\mathscr{M}^M(\mathbb{M}^n)$. Then 
\begin{equation*}\label{ew:wass-kant-pre}
\frac{1}{2} \, \mathcal{W}_2^2 \!\left(\mu^0,\mu^1 \right) = \sup_{\varphi \in \mathrm{Lip}_c(\mathbb{M}^n)} \left\{ \int_{\mathbb{M}^n} Q_1 \varphi \, \d\mu^1 - \int_{\mathbb{M}^n} \varphi \, \d\mu^0 \right\} .
\end{equation*}
\end{proposition}

We now recall a useful characterization of the convergence in the $2$-Wasserstein space, whose proof can be found in \cite[Proposition 7.1.5]{AGS}.

\begin{proposition}\label{wass-conv}
Let $ \mu \in \mathscr{M}^M_2(\mathbb{M}^n) $ and $ \{ \mu_j \}_{j \in \mathbb{N}} \subset \mathscr{M}^M_2(\mathbb{M}^n) $. Then 
$$
\lim_{j \to \infty} \mathcal{W}_2(\mu_j,\mu) = 0
$$
if and only if $ \mu_j \rightharpoonup \mu $ narrowly (i.e.~tested against any function of $ C_b(\mathbb{M}^n) $) and $ \{ \mu_j \}_{j \in \mathbb{N}} $ has equi-integrable second moments, namely there exists a point $o\in \M^n$ such that 
$$\lim_{k\to \infty} \limsup_{j\to \infty} \int_{\M^n\setminus B_k(o)}\mathsf{d}(x,o)^2 \, \d \mu_j(x)=0 \, .$$
\end{proposition}

In the proof of Theorem \ref{main-result} it will be crucial to connect any two given measures $\mu^0,\mu^1 \in \mathscr{M}_2^M(\mathbb{M}^n) $ through a curve in the $2$-Wasserstein space that satisfies some additional regularity properties, according to the following definition.

\begin{definition}\label{regcurve}
Let $\mu \equiv \{\mu^s\}_{s\in [0,1]} $ be a curve with values in $\mathscr{M}_2^M(\mathbb{M}^n)$. We say that $\mu$ is a \emph{regular curve} if $ \mu^s = \rho^s \mathcal{V} $ and the following hold:
\begin{itemize}
\item[(i)] There exists a constant $ R>0 $ such that $ \left\| \rho^s \right\|_{L^\infty(\M^n)} \leq R $ for every $s\in[0,1]$;
\item[(ii)] $\mu\in \mathrm{Lip}\!\left([0,1]; \left( \mathscr{M}_2^M(\mathbb{M}^n) , \cW \right) \right)$;
\item[(iii)] $\sqrt{\rho^s}\in \mathbb{V}$ and there exists a constant $E$ such that 
$$
\int_{\M^n} \Gamma\!\left(\sqrt{\rho^s}\right) \d \mathcal{V} \leq E \qquad \forall s\in[0,1] \, .
$$
\end{itemize}
\end{definition}

\begin{remark}\label{rr}\rm
If $\mu = \rho \mathcal{V} $ is a regular curve, in particular $\rho^s\in \mathbb{V}$ for every $ s \in [0,1] $. Moreover, thanks to \cite[Lemma 8.1]{AMS}, condition (ii) ensures that $\rho \in \mathrm{Lip}([0,1];\mathbb{V}')$.
\end{remark}

The following density result, whose proof is contained in \cite[Lemma 12.2]{AMS}, allows one to approximate Wasserstein geodesics by means of regular curves.

\begin{lemma}\label{lemma12-2ams} 
Let $\mathbb{M}^n$ satisfy \eqref{ricci-K} and $ \mu^0,\mu^1\in \mathscr{M}_2^M(\mathbb{M}^n)$. Then there exist a geodesic $ \{ \mu^s \}_{s \in [0,1]}  $ connecting $\mu^0$ and $\mu^1$ and a sequence of regular curves $ \{ \mu^s_j \}_{j \in \mathbb{N},s \in [0,1]} \subset \mathscr{M}_2^M(\mathbb{M}^n)$ such that
\begin{equation}\label{eq:est-lemma-app-A}
\lim_{j \to \infty} \mathcal{W}_2\left( \mu^s_j , \mu^s \right) = 0 \qquad \forall s \in [0,1] 
\end{equation}
and
\begin{equation}\label{eq:est-lemma-app-B}
\limsup_{j\to\infty}  \int_0^1 \left| \dot{\mu}^s_j \right|^2 \d s  \le  \mathcal{W}_2^2 \! \left( \mu^0 , \mu^1 \right) .
\end{equation}
Furthermore, if $ \mu^0 = \rho^0 \mathcal{V} $ and $ \mu^1 = \rho^1 \mathcal{V} $ with $ \rho^0 , \rho^1 \in {L}_{c}^\infty(\mathbb{M}^n) $, then $ \mu^s = \rho^s \mathcal{V} $ with $ \rho^s $ uniformly (w.r.t.~$s$) bounded and compactly supported, and in addition to \eqref{eq:est-lemma-app-A}--\eqref{eq:est-lemma-app-B} also the following hold:  
\begin{equation}\label{eq:improve1}
\lim_{j \to \infty} \left\| \rho^s_j - \rho^s \right\|_{{L}^p(\mathbb{M}^n)} = 0 \qquad \forall p \in [1,\infty) \, , \quad \forall s \in [0,1] \, ,
\end{equation}
\begin{equation}\label{eq:improve2}
\limsup_{j \to \infty} \sup_{s \in [0,1]} \left\| \rho^s_j \right\|_{{L}^\infty(\mathbb{M}^n)} < \infty \, .
\end{equation}
\end{lemma}

To conclude, given a regular curve $ \mu^s = \rho^s \d \mathcal V $ in the sense of Definition \ref{regcurve} (not necessarily a geodesic),  by combining \cite[Theorem 6.6, formula (6.11)]{AMS} and \cite[Theorem 8.2, formula (8.7)]{AMS} we can deduce that the following key identity holds:  
\begin{equation}\label{eq:key-id-E}
\int_0^1 \left| \dot{\mu}^s \right|^2 \d s = \int_0^1 \mathcal{E}^\ast_{\rho^s} \! \left[ \tfrac{\d}{\d s} \rho^s \right] \d s \, ,
\end{equation} 
where $\mathcal{E}^\ast_\rho$ is the dual weighted Dirichlet energy introduced in \eqref{eq:dual weighted Dirichlet}. Note that the r.h.s.~of \eqref{eq:key-id-E} does make sense, in view of Remark \ref{rr}.

\section{Fundamental properties of porous medium-type equations on manifolds}\label{basics-pme}

This section is devoted to the study of \eqref{pme} for more regular initial data, that is the problem
\begin{equation}\label{pme-reg}
\begin{cases}
 \partial_t\rho = \Delta P(\rho) & \text{in } \mathbb{M}^n \times \mathbb{R}^+ \, , \\
\rho(\cdot,0) = \rho_0 \ge 0 & \text{on } \mathbb{M}^n \times \{ 0 \} \, ,
\end{cases}
\end{equation}
with $ \rho_0 \in L^1(\mathbb{M}^n) \cap L^\infty(\mathbb{M}^n) $. To begin with, we will introduce the notion of \emph{weak energy} solution and then discuss some crucial related properties.
In particular, we will focus on the smoothing effect and on a bound on the support of such solutions (when the initial data are compactly supported).  
Inspired by \cite{AMS}, for a restricted class of nonlinearities we will also give an alternative (variational) notion of solution and consequently prove the equivalence with the weak-energy one.
Finally, with regards to the Hamiltonian strategy mentioned in the Introduction, we will discuss well-posedness results for the \emph{forward linearized} equation associated with \eqref{pme-reg} and for the related \emph{backward adjoint} equation.      

For convenience, in the following we make the additional (implicit) assumption that $\mathbb{M}^n$ is \emph{noncompact} and with \emph{infinite volume}, as well as in Subsection \ref{noncomp}. Note that for our purposes there is no point in considering noncompact manifolds with \emph{finite} volume, since most of our results require the validity of the Sobolev inequality \eqref{Sob}, which does not hold on such manifolds.

{The simple modifications required to deal with \emph{compact} manifolds will be shortly addressed in Subsection \ref{compact}}.
\subsection{Weak energy solutions}\label{weak-sol}

The concept of \emph{weak energy} solution of \eqref{pme-reg} has been proved to be well suited for porous medium-type equations: see e.g.~\cite[Subsections 5.3.2 and 11.2.1]{V07}, \cite[Section 3]{FM}, \cite[Subsections 3.1 and 3.2]{GMP13} or \cite[Section 2]{GMPrm}. Here we mostly take inspiration from \cite[Section 3]{FM}: there the framework is purely Euclidean, but the basic definitions and properties are straightforwardly adaptable to the Riemannian setting.

Even if in Subsection \ref{sect:not} we introduced the more synthetic notations \eqref{spazi}, here we keep the standard notations typically used in the PDE framework.

\begin{definition}[Weak energy solutions]\label{defsol}
Let $ P $ comply with assumption \eqref{increas}. Given a nonnegative $ \rho_0 \in L^1(\mathbb{M}^n) \cap L^\infty(\mathbb{M}^n) $, we say that a nonnegative measurable function $ \rho $ is a \emph{weak energy solution} of \eqref{pme-reg} if, for every $ T>0 $, there hold
\begin{equation*}\label{sol-p1}
\rho , P(\rho) \in L^2(\mathbb{M}^n \times (0,T))  \, , \qquad \nabla P(\rho) \in L^2(\mathbb{M}^n \times (0,T)) 
\end{equation*}
and 
\begin{equation}\label{sol-p2}
\int_0^T \int_{\mathbb{M}^n} \rho \,  \partial_t\eta \, \d\mathcal{V} \d t = -\int_{\mathbb{M}^n} \rho_0(x) \, \eta(x,0) \, \d\mathcal{V}(x) + \int_0^T \int_{\mathbb{M}^n} \left\langle \nabla P(\rho) \, , \nabla\eta \right\rangle \d\mathcal{V} \d t
\end{equation}
for every $ \eta \in W^{1,2}((0,T);L^2(\mathbb{M}^n)) $ with $ \nabla{\eta} \in L^2((0,T);L^2(\mathbb{M}^n)) $ such that $ \eta(T) = 0 $.
\end{definition}

Existence and uniqueness of weak energy solutions, at least for the class of data $ L^1(\mathbb{M}^n) \cap L^\infty(\mathbb{M}^n) $, is by now a well-established issue (see e.g.~the references quoted above). Nevertheless, since it will be very useful to our later purposes, we recall here the approximation procedure that allows one to construct such solutions: the essential idea is to approximate the possibly degenerate nonlinearity $ P \in C^1([0,+\infty))$ by means of suitable \emph{nondegenerate} nonlinearities. More precisely, for every $ \varepsilon>0 $ we define a function $ P_\varepsilon $ by
\begin{equation}\label{P-app}
\left( P_\varepsilon \right)^\prime\!(\rho) := 
\begin{cases}
P^\prime(\rho)+\varepsilon & \text{if} \ \rho \in \left[ 0 , \frac{1}{\varepsilon} \right] , \\
 \left[ P^\prime\!\left( \frac{1}{\varepsilon} \right) \wedge P^\prime\!\left(\rho\right) \right] + \varepsilon & \text{if} \ \rho > \frac{1}{\varepsilon} \, ,
\end{cases}
\qquad P_\varepsilon(0)=0 \, .
\end{equation}
In the following simple lemma, we collect some crucial properties enjoyed by $P_{\varepsilon}$. 
\begin{lemma}
Let $ P $ comply with \eqref{increas}. We have that $ P_\varepsilon\in C^1([0,+\infty))$ with the estimates
\begin{equation}\label{est-peps-1} 
P_\varepsilon(\rho) \le P(\rho) + \varepsilon \rho \qquad \forall \rho \ge 0 \,  ,
\end{equation}
\begin{equation}\label{est-peps-3} 
\left(P_\varepsilon \right)^\prime\!(\rho) \ge P^\prime(\rho)  \qquad \forall \rho \in \left[ 0 , \tfrac{1}{\varepsilon} \right] , 
\end{equation}
and
\begin{equation}\label{est-peps-2} 
\varepsilon \le \left(P_\varepsilon \right)^\prime\!(\rho) \le \max_{\rho\in \left[ 0 , 1 / \varepsilon \right]} P^\prime\!\left( \rho \right) + \varepsilon \qquad \forall \rho \ge 0 \, .
\end{equation}
In particular, $ P_\varepsilon $ is also strictly increasing. Moreover, if \eqref{fund-ineq} holds then 
\begin{equation}\label{est-peps-4}  
\rho \left(P_\varepsilon \right)^\prime\!(\rho) -\left(1 - \tfrac{1}{n}\right) P_\varepsilon(\rho) \ge 0  \qquad \forall \rho \ge 0 \, . 
\end{equation}
\end{lemma}
\begin{proof}
The fact that $ P_\varepsilon $ is  $ C^1([0,+\infty)) $ easily follows from \eqref{P-app} and the continuity of the minimum operator, since $ P'  $ is continuous. By integration, we obtain 
\begin{equation}\label{proof est-peps}
 P_\varepsilon(\rho)=\begin{cases}
P(\rho)+\varepsilon\rho & \text{if} \ \rho \in \left[ 0 , \frac{1}{\varepsilon} \right] , \\
P\!\left( \frac{1}{\varepsilon} \right) + \varepsilon\rho + \displaystyle \int_{\frac{1}{\varepsilon}}^{\rho} P^\prime\!\left( \tfrac{1}{\varepsilon} \right) \wedge P^\prime\!\left(s\right) \d s & \text{if} \ \rho > \frac{1}{\varepsilon} \, .
\end{cases}
\end{equation}
Inequality \eqref{est-peps-1} is obvious (being an identity) in the interval $\left[ 0 , \frac{1}{\varepsilon} \right]$, while for $\rho>\frac{1}{\varepsilon}$ it is a consequence of the trivial inequality $ P^\prime\!\left( \frac{1}{\varepsilon} \right) \wedge P^\prime\!\left(s\right) \le P^\prime\!\left(s\right)$. The bounds \eqref{est-peps-3} and \eqref{est-peps-2} are a direct consequence of the definition of $\left( P_\varepsilon \right)^\prime$ together with the properties $ P^\prime \ge 0$ and 
$$
\left[ P^\prime\!\left( \tfrac{1}{\varepsilon} \right) \wedge P^\prime\!\left(\rho\right) \right] + \varepsilon \le P^\prime\!\left( \tfrac{1}{\varepsilon} \right)+ \varepsilon\le \max_{\rho\in \left[ 0 , 1 / \varepsilon \right]} P^\prime\!\left( \rho \right) + \varepsilon\, .
$$
It remains to prove \eqref{est-peps-4} under \eqref{fund-ineq}. Its validity is clear in the interval $\left[ 0 , \frac{1}{\varepsilon} \right]$, so let us assume without loss of generality that $\rho > \frac{1}{\varepsilon}$. Using the explicit expression \eqref{proof est-peps}, it holds:  
\begin{equation}\label{proof est-peps2}
\begin{aligned}
 & \, \rho \left(P_\varepsilon \right)^\prime\!(\rho) -\left(1 - \tfrac{1}{n}\right) P_\varepsilon(\rho) \\
= & \, \rho \left[ P^\prime\!\left( \tfrac{1}{\varepsilon} \right) \wedge P^\prime\!\left(\rho\right) \right] + \varepsilon \rho - \left(1 - \tfrac{1}{n}\right) \left[ P\!\left( \tfrac{1}{\varepsilon} \right) + \varepsilon\rho + \int_{\frac{1}{\varepsilon}}^{\rho} P^\prime\!\left( \tfrac{1}{\varepsilon} \right) \wedge P^\prime\!\left(s\right) \d s \right] \\
\ge  & \, \rho \left[ P^\prime\!\left( \tfrac{1}{\varepsilon} \right) \wedge P^\prime\!\left(\rho\right) \right]  - \left(1 - \tfrac{1}{n}\right) \left[ P\!\left( \tfrac{1}{\varepsilon} \right)  + \int_{\frac{1}{\varepsilon}}^{\rho} P^\prime\!\left( \tfrac{1}{\varepsilon} \right) \wedge P^\prime\!\left(s\right) \d s \right] . 
\end{aligned}
\end{equation}
Now, if  $P^\prime\!\left( \frac{1}{\varepsilon}\right) \ge P^\prime\!\left( \rho \right)$ the result follows by noticing that $ P^\prime\!\left( \frac{1}{\varepsilon} \right) \wedge P^\prime\!\left(s\right) \le P^\prime\!\left( s \right)$, integrating in the right-most side of \eqref{proof est-peps2} and taking advantage of assumption \eqref{fund-ineq}. If instead $P^\prime\!\left( \frac{1}{\varepsilon}\right) < P^\prime\!\left( \rho \right)$, we exploit the property $P^\prime\!\left( \frac{1}{\varepsilon} \right) \wedge P^\prime\!\left(s\right) \le P^\prime\!\left( \frac{1}{\varepsilon} \right) $ and observe that in this case the right-most side of \eqref{proof est-peps2} is bounded from below by
$$
\begin{aligned}
 & \, \rho P^\prime\!\left( \tfrac{1}{\varepsilon} \right) - \left(1 - \tfrac{1}{n}\right) P\left( \tfrac{1}{\varepsilon} \right) - \left(1 - \tfrac{1}{n}\right)\left(\rho - \tfrac{1}{\varepsilon}\right) P^\prime\!\left( \tfrac{1}{\varepsilon} \right) \\
= &  \,  \tfrac{1}{\varepsilon} P^\prime\!\left( \tfrac{1}{\varepsilon} \right) -\left(1 - \tfrac{1}{n}\right)P\left( \tfrac{1}{\varepsilon} \right)+\tfrac{1}{n}\left(\rho - \tfrac{1}{\varepsilon}\right)P^\prime\!\left( \tfrac{1}{\varepsilon} \right) \ge 0 \, , 
\end{aligned}
$$
where we have used again assumption \eqref{fund-ineq} at $\rho=\frac{1}{\varepsilon}$ along with the fact that $P^\prime \ge 0$. 
\end{proof}

Note that if $ P $ complies with the left-hand bound in \eqref{below-above-prime} so does $ P^\prime_\varepsilon $ in the interval $ \left[0,{1}/{\varepsilon}\right] $, thanks to \eqref{est-peps-3}. This bound is crucial to establish the \emph{smoothing effect}, which is a key ingredient to our strategy (see Proposition \ref{smooth-approx} below). Accordingly, we thus address the following approximate version of \eqref{pme-reg}: 
\begin{equation}\label{pme-approx}
\begin{cases}
 \partial_t\rho_\varepsilon  = \Delta P_\varepsilon(\rho_\varepsilon) & \text{in } \mathbb{M}^n \times \mathbb{R}^+ \, , \\
\rho_{\varepsilon}(\cdot,0) = \rho_0 & \text{on } \mathbb{M}^n \times \{ 0 \} \, .
\end{cases}
\end{equation} 
We stress that the notation $ \rho_\varepsilon $ for the solution of \eqref{pme-approx} only makes sense for $ \varepsilon>0 $, and it should not be confused with the initial datum $ \rho_0 $. It may help to keep in mind that $ \rho_\varepsilon $  is supposed to be a ``regularization'' of $ \rho $, the latter being the solution of \eqref{pme-reg}.

Problem \eqref{pme-approx} can be interpreted both from the viewpoint of \emph{linear} and \emph{nonlinear} theory, in the sense that $ P_\varepsilon $ is a nonlinear function but it is ``uniformly elliptic'', hence one expects that the solutions of \eqref{pme-approx} enjoy, to some extent, properties similar to those satisfied by the solutions of the \emph{heat equation} (we refer to Propositions \ref{prop: sol AMS} and \ref{oleinik} below). We will mainly take advantage of the linear interpretation in Section \ref{sec:str}, in agreement with the approach of \cite{AMS}. The nonlinear interpretation is exploited in the present section. 

\begin{proposition}[Existence, uniqueness, properties of  weak energy solutions]\label{exunimain}
Let $ P $ comply with \eqref{increas}. Given a nonnegative $ \rho_0 \in L^1(\mathbb{M}^n) \cap L^\infty(\mathbb{M}^n) $, there exists a \emph{unique} weak energy solution $\rho$ of \eqref{pme-reg}, which enjoys the following additional properties:
\begin{itemize}
\item \emph{$L^1$-continuity:} $ \{ \rho(t) \}_{t \ge 0} $ is a continuous curve with values in $ L^1(\mathbb{M}^n) $;

\smallskip

\item \emph{Energy inequality:} $\rho$ satisfies
\begin{equation}\label{eest}
\int_0^T \int_{\mathbb{M}^n} \left| \nabla{P(\rho)} \right|^2 \d\mathcal{V} \d t + \int_{\mathbb{M}^n} \Psi(\rho(x,T)) \, \d\mathcal{V}(x) \le \int_{\mathbb{M}^n} \Psi(\rho_0) \, \d\mathcal{V} \qquad \forall T > 0 \, ,
\end{equation}
where $ \Psi(\rho) := \int_0^\rho P(r)\,\d r $;

\smallskip

\item \emph{Nonexpansivity of the $ L^p $ norms:} for every $ p \in [1,\infty] $ there holds
\begin{equation}\label{eq: non-exp}
\left\| \rho(t) \right\|_{L^p(\mathbb{M}^n)} \le \left\| \rho_0 \right\|_{L^p(\mathbb{M}^n)}  \qquad \forall t > 0 \, ;
\end{equation}

\smallskip

\item \emph{Mass conservation:} if in addition $ \mathbb{M}^n $ satisfies \eqref{ricci-K} then
\begin{equation}\label{cons-mass-th}
\int_{\mathbb{M}^n} \rho(x,t) \, \d\mathcal{V}(x) = \int_{\mathbb{M}^n} \rho_0 \, \d\mathcal{V} \qquad \forall t >0 \, ;
\end{equation}

\smallskip

\item \emph{Approximation:} if $ \varepsilon>0 $ and $ \rho_\varepsilon $ is the weak energy solution of \eqref{pme-approx}, where $ P_\varepsilon(\rho) $ is defined in \eqref{P-app}, then 
\begin{equation}\label{conv-epsilon}
\lim_{\varepsilon \downarrow 0} \left\| \rho_\varepsilon(t) - \rho(t) \right\|_{L^1_{\mathrm{loc}}(\mathbb{M}^n)} = 0 \qquad \forall t>0 \, ;
\end{equation}
if in addition \eqref{ricci-K} is satisfied, then
\begin{equation}\label{conv-epsilon-bis}
\lim_{\varepsilon \downarrow 0} \left\| \rho_\varepsilon(t) - \rho(t) \right\|_{L^1(\mathbb{M}^n)} = 0 \qquad \forall t>0 \,;  
\end{equation}

\smallskip

\item \emph{$L^1$-contraction:} if $ \hat{\rho} $ is the weak energy solution corresponding to another nonnegative initial datum $ \hat{\rho}_0 \in L^1(\mathbb{M}^n) \cap L^\infty(\mathbb{M}^n) $, then
\begin{equation}\label{L1-est}
\left\| \rho(t)-\hat{\rho}(t) \right\|_{L^1(\mathbb{M}^n)} \le \left\| \rho_0-\hat{\rho}_0 \right\|_{L^1(\mathbb{M}^n)} \qquad \forall t > 0 \, .
\end{equation}

\end{itemize}
\end{proposition}
\begin{proof}
We start by recalling that uniqueness of weak energy solutions follows from a standard trick due to Ole\u{\i}nik: given $T>0 $, one plugs the (admissible) test function 
$$
\eta(x,t) = \int_t^T \left[ P(\rho(x,s)) - P(\hat{\rho}(x,s)) \right] \d s \, , \qquad (x,t) \in \mathbb{M}^n \times [0,T] \, ,
$$
into the weak formulation satisfied by the difference between $ \rho $ and $ \hat{\rho} $ (the latter being two possibly different solutions corresponding to the same initial datum), thus obtaining
\begin{equation}\label{olei}
\begin{aligned}
& \int_0^T \int_{\mathbb{M}^n} \left( \rho-\hat{\rho} \right) \left( P(\rho) - P(\hat{\rho}) \right) \d\mathcal{V} \d t \\
= & \int_0^T \int_{\mathbb{M}^n} \left\langle \nabla \left[ P(\rho(x,t)) - P(\hat{\rho}(x,t)) \right] , \int_t^T \nabla \left[ P(\rho(x,s)) - P(\hat{\rho}(x,s)) \right] \d s \right\rangle \d\mathcal{V}(x) \d t \, .
\end{aligned}
\end{equation}
A simple time integration in \eqref{olei} yields 
$$
\int_0^T \int_{\mathbb{M}^n} \left( \rho-\hat{\rho} \right) \left( P(\rho) - P(\hat{\rho}) \right) \d\mathcal{V} \d t + \frac12 \int_{\mathbb{M}^n} \left| \int_0^T \nabla \left[ P(\hat{\rho}(x,s)) - P(\rho(x,s)) \right] \d s \right|^2 \d\mathcal{V}(x) = 0 \, ,
$$
which ensures that $ \rho = \hat{\rho} $ given the strict monotonicity of $ \rho \mapsto P(\rho) $ and the arbitrariness of $ T $. Note that here we have only used the validity of \eqref{sol-p2} for functions $\eta \in W^{1,2}((0,T);W^{1,2}(\mathbb{M}^n))$. Furthermore, the fact that $\rho_0 \in L^1(\mathbb{M}^n) \cap L^\infty(\mathbb{M}^n )$ is unimportant. These observations will be useful in the proof of Proposition \ref{prop: sol AMS} below.  

As concerns the construction of a weak energy solution, we will not provide a complete proof since the procedure is quite standard: see e.g.~\cite[Theorem 5.7 and Lemma 5.8]{V07} or \cite[Theorems 3.4 and 3.7]{GMP13} in Euclidean or weighted-Euclidean contexts. The basic idea consists first of solving problem \eqref{pme-reg} in a sequence $ D_k $ of bounded regular domains that form an exhaustion for $ \mathbb{M}^n $ (see the proof of Lemma \ref{compact-support} below for more details on such a sequence), with homogeneous Dirichlet boundary conditions on $ \partial D_k $. In order to do this, it is convenient to make a further approximation by replacing $ P $ with $ P_\varepsilon $: let us denote by $ \rho_{\varepsilon,k} $ the corresponding solutions, which are therefore regular enough (up to approximating also the initial datum $ \rho_0 $ and approximating further $ P_\varepsilon $ in case $ P' $ is merely continuous -- we skip this passages). A first key estimate is provided by the energy inequality itself, which is obtained upon multiplying the differential equation by $ P_\varepsilon(\rho_{\varepsilon,k}) $ and integrating by parts:
\begin{equation}\label{eest-2}
\int_0^T \int_{D_k} \left| \nabla{P_\varepsilon(\rho_{\varepsilon,k})} \right|^2 \d\mathcal{V} \d t + \int_{D_k} \Psi_\varepsilon(\rho_{\varepsilon,k}(x,T)) \, \d\mathcal{V}(x) = \int_{D_k} \Psi_\varepsilon(\rho_0) \, \d\mathcal{V} \qquad \forall T > 0 \, ,
\end{equation}
where $ \Psi_\varepsilon(\rho) := \int_0^\rho P_\varepsilon(r) \, \d r  $. Note that for the moment the energy inequality is in fact an identity. Another crucial estimate involves time derivatives and is obtained by multiplying the differential equation by $ \zeta \,  P^\prime_\varepsilon(\rho_{\varepsilon,k}) \, \partial_t \rho_{\varepsilon,k} $ and again integrating by parts, where $ \zeta \in C_c^\infty((0,+\infty)) $ is any cut-off function that depends only on time and satisfies $ 0 \le \zeta \le 1 $; this yields
\begin{equation}\label{eest-3}
\int_0^T \int_{D_k} \zeta \left| \partial_t \Upsilon_\varepsilon(\rho_{\varepsilon,k}) \right|^2 \d\mathcal{V} \d t = \frac{1}{2} \int_0^T \int_{D_k} \zeta^\prime  \left| \nabla{P_\varepsilon(\rho_{\varepsilon,k})} \right|^2 \d\mathcal{V} \d t \qquad \forall T > 0 \, ,
\end{equation}
where $ \Upsilon_\varepsilon(\rho) := \int_0^\rho \sqrt{P^\prime_\varepsilon(r)}\, \d r $. Finally, by using $ \rho_{\varepsilon,k} $ itself as a test function we obtain
\begin{equation}\label{eest-4}
\int_0^T \int_{D_k} \left| \nabla{\Upsilon_\varepsilon(\rho_{\varepsilon,k})} \right|^2 \d\mathcal{V} \d t + \frac{1}{2} \int_{D_k} \rho_{\varepsilon,k}(x,T)^2 \, \d\mathcal{V}(x) = \frac{1}{2} \int_{D_k} \rho_0^2 \, \d\mathcal{V} \qquad \forall T > 0 \, ;
\end{equation}
a similar computation ensures that in fact all $ L^p(D_k) $ norms do not increase: 
\begin{equation}\label{eq: non-exp-k}
\left\| \rho_{\varepsilon,k}(t) \right\|_{L^p(D_k)} \le \left\| \rho_0 \right\|_{L^p(D_k)}  \qquad \forall t > 0 \, , \quad \forall p \in [1,\infty] \, .
\end{equation}
If $ \hat{\rho}_{\varepsilon,k} $ is another (approximate) solution corresponding to a different nonnegative $ \hat{\rho}_0 \in L^1(\mathbb{M}^n) \cap L^\infty(\mathbb{M}^n) $, the $ L^1 $-contraction property simply follows upon multiplying the differential equation satisfied by $ (\rho_{\varepsilon,k} - \hat{\rho}_{\varepsilon,k}) $ formally by the test function $ \operatorname{sign}(\rho_{\varepsilon,k} - \hat{\rho}_{\varepsilon,k}) $ and integrating: this leads to
\begin{equation*}\label{L1-est-k}
\left\| \rho_{\varepsilon,k}(t)-\hat{\rho}_{\varepsilon,k}(t) \right\|_{L^1(D_k)} \le \left\| \rho_0-\hat{\rho}_0 \right\|_{L^1(D_k)} \qquad \forall t > 0 \, .
\end{equation*}
Actually, to be more rigorous, the sign function should further be approximated by regular nondecreasing functions, see \cite[Proposition 3.5]{V07}. We are now ready to pass to the limit into the weak formulation satisfied by each $ \rho_{\varepsilon,k} $, which reads  
\begin{equation}\label{sol-KE}
\int_0^T \int_{D_k} \rho_{\varepsilon,k} \,  \partial_t\eta \, \d\mathcal{V} \d t = -\int_{D_k} \rho_0(x) \, \eta(x,0) \, \d\mathcal{V}(x) + \int_0^T \int_{D_k} \left\langle \nabla P_\varepsilon(\rho_{\varepsilon,k}) \, , \nabla\eta \right\rangle \d\mathcal{V} \d t
\end{equation}
for every $ T>0 $ and every $ \eta \in W^{1,2}((0,T);L^2(D_k)) \cap L^2((0,T);W^{1,2}_0(D_k)) $ such that $ \eta(T) = 0 $. Indeed, the energy estimate \eqref{eest-2} ensures that $ \{ \nabla P_\varepsilon(\rho_{\varepsilon,k}) \}_{\varepsilon>0} $ weakly converges (up to subsequences) as $ \varepsilon \downarrow 0 $ to some vector field $ \vec{w} $ in $ L^2(D_k \times (0,T)) $, whereas \eqref{eq: non-exp-k} yields weak convergence of $ \{ \rho_{\varepsilon,k} \}_{\varepsilon>0} $ for instance in $ L^2(D_k \times (0,T)) $ to some limit function $ \rho_k $, still up to subsequences. 
On the other hand, estimates \eqref{eest-2}--\eqref{eest-4} guarantee that $ \{ \Upsilon_\varepsilon(\rho_{\varepsilon,k}) \}_{\varepsilon>0} $ is locally bounded in $ H^1(D_k \times (0,T)) $; in particular it admits a subsequence that converges pointwise almost everywhere. 
Since $ \Upsilon_\varepsilon $, $ \Upsilon_\varepsilon^{-1} $ and $ P_\varepsilon $ are continuous, monotone increasing functions converging pointwise (and therefore locally uniformly) as $ \varepsilon \downarrow 0 $ to their continuous limits $ \Upsilon(\rho) := \int_0^\rho \sqrt{P^\prime(r)}\, \d r $, $ \Upsilon^{-1} $ and $P$, respectively, we can assert that also $ \{ \rho_{\varepsilon,k} \}_{\varepsilon>0} $ and $ \{ P_\varepsilon(\rho_{\varepsilon,k}) \}_{\varepsilon>0} $ converge pointwise, up to subsequences. 
This is the key to guarantee the identification $ \vec{w} = \nabla{P(\rho_k)} $, so that by letting $ \varepsilon \downarrow 0 $ in \eqref{sol-KE} we end up with 
\begin{equation*}\label{sol-K}
\int_0^T \int_{D_k} \rho_{k} \,  \partial_t\eta \, \d\mathcal{V} \d t = -\int_{D_k} \rho_0(x) \, \eta(x,0) \, \d\mathcal{V}(x) + \int_0^T \int_{D_k} \left\langle \nabla P(\rho_{k}) \, , \nabla\eta \right\rangle \d\mathcal{V} \d t \, ,
\end{equation*}
which is valid for every $ T>0 $ and the same type of test functions $ \eta $ as in \eqref{sol-KE}. Note that all the above estimates pass to the limit as $ \varepsilon \downarrow 0 $ e.g.~by lower semicontinuity, yielding 
\begin{equation}\label{eest-2-k}
\int_0^T \int_{D_k} \left| \nabla{P(\rho_{k})} \right|^2 \d\mathcal{V} \d t + \int_{D_k} \Psi(\rho_{k}(x,T)) \, \d\mathcal{V}(x) \le \int_{D_k} \Psi(\rho_0) \, \d\mathcal{V} \qquad \forall T > 0 \, ,
\end{equation}
\begin{equation}\label{eest-3-k}
\int_0^T \int_{D_k} \zeta \left| \partial_t \Upsilon(\rho_{k}) \right|^2 \d\mathcal{V} \d t \le \frac{\max_{\mathbb{R}^+}|\zeta^\prime|}{2} \int_{D_k} \Psi(\rho_0) \, \d\mathcal{V} \qquad \forall T > 0 \, ,
\end{equation}
\begin{equation}\label{eest-4-k}
\int_0^T \int_{D_k} \left| \nabla{\Upsilon(\rho_{k})} \right|^2 \d\mathcal{V}  \d t + \frac{1}{2} \int_{D_k} \rho_{k}(x,T)^2 \, \d\mathcal{V}(x) \le \frac{1}{2} \int_{D_k} \rho_0^2 \, \d\mathcal{V} \qquad \forall T > 0 \, ,
\end{equation}
\begin{equation}\label{eq: non-exp-kk}
\left\| \rho_{k}(t) \right\|_{L^p(D_k)} \le \left\| \rho_0 \right\|_{L^p(D_k)}  \qquad \forall t > 0 \, , \quad \forall p \in [1,\infty] \, ,
\end{equation}
\begin{equation}\label{L1-est-kk}
\left\| \rho_{k}(t)-\hat{\rho}_{k}(t) \right\|_{L^1(D_k)} \le \left\| \rho_0-\hat{\rho}_0 \right\|_{L^1(D_k)} \qquad \forall t > 0 \, .
\end{equation}
At this point we are allowed to let $ k\to \infty $, so that $ D_k $ will eventually become the whole manifold $ \mathbb{M}^n $. By exploiting estimates \eqref{eest-2-k}--\eqref{L1-est-kk} and reasoning similarly to the previous step, we can easily deduce that $ \{ \rho_k \}_{k \in \mathbb{N}} $ (extended to zero in $ \mathbb{M}^n \setminus D_k $) suitably converges as $ k \to \infty $ to the energy solution $ \rho $ of \eqref{pme-reg}, which therefore satisfies \eqref{eest}, \eqref{eq: non-exp} and \eqref{L1-est} (upon repeating the same procedure starting from $ \hat{\rho}_0 $), along with
\begin{equation}\label{eest-3-k-last}
\int_0^T \int_{\mathbb{M}^n} \zeta \left| \partial_t \Upsilon(\rho) \right|^2 \d\mathcal{V} \d t \le \frac{\max_{\mathbb{R}^+}|\zeta^\prime|}{2} \int_{\mathbb{M}^n} \Psi(\rho_0) \, \d\mathcal{V} \qquad \forall T > 0 \, ,
\end{equation}
\begin{equation}\label{eest-4-k-last}
\int_0^T \int_{\mathbb{M}^n} \left| \nabla{\Upsilon(\rho)} \right|^2 \d\mathcal{V} \d t + \frac{1}{2} \int_{\mathbb{M}^n} \rho(x,T)^2 \, \d\mathcal{V}(x) \le \frac{1}{2} \int_{\mathbb{M}^n} \rho_0^2 \, \d\mathcal{V} \qquad \forall T > 0 \, .
\end{equation} 
Note that since $ P \in C^1([0,+\infty)) $ and $ \rho_0 \in L^1(\mathbb{M}^n)\cap L^{\infty}(\mathbb{M}^n) $ the r.h.s.~of \eqref{eest-3-k-last} is surely finite. We are thus left with proving $ L^1 $-continuity, mass conservation and \eqref{conv-epsilon}--\eqref{conv-epsilon-bis}. 

In order to establish the mass-conservation property, we take advantage of a recent result contained in \cite{BS}, which ensures that under \eqref{ricci-K} for every $ R \ge 1 $ there exist positive constants $ C,\gamma $ independent of $R$ and a nonnegative function $ \phi_R \in C^\infty_c(\mathbb{M}^n) $ such that $ \phi_R = 1 $ in $ B_R(o) $, $ \mathrm{supp} \, \phi_R \subset B_{\gamma R}(o) $ (let $ o \in \mathbb{M}^n $ be a fixed pole), $ \phi_R \le 1  $ and $ \left| \Delta \phi_R \right| \le {C}/{R} $. See in particular \cite[Corollary 2.3]{BS}. So let us plug into \eqref{sol-p2} the test function $ \eta(x,t) = \phi_R(x) \xi(t) $, where $ \xi \in C^\infty_c([0,T)) $ with $ \xi(0)=1 $; we obtain
\begin{equation}\label{sol-mass}
\int_0^T \int_{\mathbb{M}^n} \rho \, \phi_R \, \xi^\prime \, \d\mathcal{V} \d t = -\int_{\mathbb{M}^n} \rho_0 \, \phi_R \, \d\mathcal{V} + \int_0^T \int_{\mathbb{M}^n} \xi \left\langle \nabla P(\rho) \, , \nabla \phi_R \right\rangle \d\mathcal{V} \d t \, .
\end{equation}
If we suitably let $ \xi \to \chi_{[0,T]} $ and we integrate by parts the second term in the r.h.s.~of \eqref{sol-mass}, we end up with 
\begin{equation*}\label{sol-mass-2}
\int_{\mathbb{M}^n} \rho(x,T) \, \phi_R(x) \, \d\mathcal{V}(x) \d t = \int_{\mathbb{M}^n} \rho_0 \, \phi_R \, \d\mathcal{V} + \int_0^T \int_{\mathbb{M}^n} P(\rho) \, \Delta\phi_R \, \d\mathcal{V} \d t \, .
\end{equation*}
By letting $ R \to \infty $, exploiting the integrability properties of $ \rho $ (note that $ P(\rho) \in L^1(\M^n \times(0,T)) $) along with the above estimate on $ \Delta \phi_R $ and the arbitrariness of $T$, we deduce \eqref{cons-mass-th}.

As concerns $ L^1 $-continuity, as a first step we point out that it could be proved by means of an alternative construction of weak energy solutions that takes advantage of time-discretization and the Crandall-Liggett Theorem: see e.g.~\cite[Remark 3.7]{FM}. More comments on such a construction will be made in Remark \ref{R} at the end of this section. However, in the present framework it can be obtained in a more direct fashion, at least under \eqref{ricci-K}. Indeed, if we let $ \zeta \to \chi_{[0,T]} $ in \eqref{eest-3}, upon a passage to the limit as $ \varepsilon \downarrow 0 $ and $  k\to\infty $ we infer that
\begin{equation*}\label{cont-eest}
\int_0^T \int_{\mathbb{M}^n} \left| \partial_t \Upsilon(\rho) \right|^2 \d\mathcal{V} \d t + \frac{1}{2} \int_{\mathbb{M}^n} \left| \nabla{P(\rho(x,T))} \right|^2 \d\mathcal{V}(x) \le \frac 12 \int_{\mathbb{M}^n} \left| \nabla{P(\rho_0)} \right|^2 \d\mathcal{V} \qquad \forall T > 0 \, .
\end{equation*}
This in particular ensures that, at least for initial data $ \rho_0 \in C^1_c(\mathbb{M}^n) $, the curve $ t \mapsto \Upsilon(\rho(t)) $ is in $ W^{1,2}\!\left((0,T);L^2(\M^n) \right) $, which further guarantees that $ \rho(t) \to \rho_0 $ as $ t \downarrow 0 $ in $ L^1_{\mathrm{loc}}(\mathbb{M}^n) $ (recall the uniform boundedness of $\rho$); on the other hand, the just proved mass-conservation property implies $ \| \rho(t) \|_{L^1(\mathbb{M}^n)} = \| \rho_0 \|_{L^1(\mathbb{M}^n)} $ for all $ t>0 $, so that the convergence does occur in $ L^1(\mathbb{M}^n) $. By virtue of the contraction estimate \eqref{L1-est}, the $ L^1 $-continuity of $ t \mapsto \rho(t) $ at $ t=0 $ yields the $ L^1 $-continuity at any other time, so that in fact $ \rho \in C([0,+\infty);L^1(\mathbb{M}^n)) $. This holds provided $ \rho_0 \in C^1_c(\mathbb{M}^n) $: for a general initial datum $ \rho_0 \in L^1(\mathbb{M}^n) \cap L^\infty(\mathbb{M}^n) $, if we take a sequence $ \{ \rho_{j,0} \}_{j \in \N} \subset C^1_c(\mathbb{M}^n)  $ such that $ \rho_{j,0} \to \rho_0 $ in $ L^1(\mathbb{M}^n) $, with $ \rho_{j,0} \ge 0 $, still the contraction estimate \eqref{L1-est} ensures that the corresponding sequence of energy solutions $ \{ \rho_j \}_{j \in \mathbb{N}} $ converges to $ \rho $ in $ L^\infty(\mathbb{R}^+;L^1(\mathbb{M}^n)) $, hence also $ t \mapsto \rho(t) $ belongs to $ C([0,+\infty);L^1(\mathbb{M}^n)) $ (so that a posteriori we have the right to write all the above estimates for \emph{every} rather than \emph{almost every} $t$ or $T$). 

Let us finally establish the approximation properties \eqref{conv-epsilon}--\eqref{conv-epsilon-bis}. Given $ \varepsilon>0 $, if $ \rho_\varepsilon $ is the weak energy solution of \eqref{pme-approx} then it satisfies \eqref{eest} (with $ P\equiv P_\varepsilon $ and $ \Psi \equiv\Psi_\varepsilon $), \eqref{eq: non-exp}, \eqref{cons-mass-th} and \eqref{eest-3-k-last}--\eqref{eest-4-k-last} (with $ \Upsilon \equiv \Upsilon_\varepsilon $ and $ \Psi \equiv \Psi_\varepsilon $): by proceeding as in the first part of the proof, one can easily infer that $ \{ \rho_\varepsilon \}_{\varepsilon>0} $ converges pointwise almost everywhere in $ \mathbb{M}^n \times \mathbb{R}^+ $ as $ \varepsilon \downarrow 0 $ to $ \rho $, up to subsequences. This implies convergence in $ L^1_{\mathrm{loc}}(\mathbb{M}^n) $ for a.e.~$ t \in \mathbb{R}^+ $, given the uniform boundedness of $ \{ \rho_\varepsilon \}_{\varepsilon>0} $. In order to show that such convergence occurs at \emph{every} $ t $, note that by \eqref{eest-3-k-last} the family $ \{ \Upsilon_\varepsilon(\rho_\varepsilon) \}_{\varepsilon>0} $ is equicontinuous with values in $ L^2(\mathbb{M}^n) $, at least for times bounded away from zero:
$$
\left\|\Upsilon_\varepsilon(\rho_\varepsilon(t)) - \Upsilon_\varepsilon(\rho_\varepsilon(s)) \right\|_{L^2(\mathbb{M}^n)} \le \sqrt{t-s} \left\| \partial_t \Upsilon_\varepsilon(\rho_\varepsilon) \right\|_{L^2(\mathbb{M}^n\times(s,t))} \qquad \forall t>s > 0 \, ;
$$
by the Ascoli-Arzel\`a theorem we then deduce that $ \{\Upsilon_\varepsilon(\rho_\varepsilon(t))\}_{\varepsilon>0} $ converges locally in $ L^2(\mathbb{M}^n) $ to $ \Upsilon(\rho(t)) $ for every $t>0$, whence the convergence of $ \{ \rho_\varepsilon(t) \}_{\varepsilon>0} $ in $ L^1_{\mathrm{loc}}(\mathbb{M}^n) $, thanks to the just recalled uniform boundedness of $ \{ \rho_\varepsilon \}_{\varepsilon>0} $. Finally, the global convergence under \eqref{ricci-K} is again a consequence of mass conservation. 
\end{proof}

As mentioned above, a fundamental ingredient to the strategy of proof of Theorem \ref{main-result} (see  Section \ref{sec:str}) is the smoothing effect, namely a quantitative $ L^1(\mathbb{M}^n) $--$ L^\infty(\mathbb{M}^n) $ regularization property of the nonlinear evolution that depends only on the $ L^1 $ norm of the initial datum. To this end we need to ask some crucial extra assumptions: the validity of the Sobolev-type inequality \eqref{Sob} and a bound from below on the degeneracy of $ P $ given by the left-hand side of \eqref{below-above-prime}. The proof is largely inspired from \cite[Section 4]{FM}, where a Moser-type iteration is exploited (see also references quoted therein); nevertheless, here we are also interested in keeping track of the dependence of the multiplying constants on $ m $ as $ m \downarrow 1 $. 

\begin{proposition}[Smoothing effect]\label{smooth-approx}
Let $ \mathbb{M}^n $ ($ n \ge 3 $) comply with \eqref{Sob}. Let $ P $ comply with \eqref{increas} and the left-hand inequality in \eqref{below-above-prime}. Let $ \varepsilon>0 $ and 
 $ \rho_0 \in {L}^1(\mathbb{M}^n) \cap {L}^\infty(\mathbb{M}^n) $ be nonnegative. Then the weak energy solution $ \rho_\varepsilon $ of \eqref{pme-approx}, where $P_{\varepsilon}$ is defined by \eqref{P-app}, satisfies the \emph{smoothing estimate} 
\begin{equation}\label{smoothing-one}
\left\| \rho_\varepsilon(t) \right\|_{{L}^\infty\left(\mathbb{M}^n\right)} \le C \left( t^{-\frac{n}{2+n(m-1)}} \left\| \rho_0 \right\|_{{L}^1(\mathbb{M}^n)}^{\frac{2}{2+n(m-1)}} + \left\| \rho_0 \right\|_{{L}^1(\mathbb{M}^n)} \right) \qquad \forall t>0 
\end{equation}
provided
\begin{equation}\label{cond-eps-one}
\left\| \rho_0 \right\|_{{L}^\infty(\mathbb{M}^n)} \le \frac{1}{\varepsilon} \, ,
\end{equation}
where $ C\geq 1 $ is a constant depending only on $c_0$, $C_S$, $ n $ and independent of $ m $ ranging in a bounded subset of $ (1,+\infty) $. As a consequence, if $ \rho $ is the weak energy solution of \eqref{pme-reg} starting from the same initial datum, there holds
\begin{equation}\label{smoothing-limit}
\left\| \rho(t) \right\|_{{L}^\infty\left(\mathbb{M}^n\right)} \le C \left( t^{-\frac{n}{2+n(m-1)}} \left\| \rho_0 \right\|_{{L}^1(\mathbb{M}^n)}^{\frac{2}{2+n(m-1)}} + \left\| \rho_0 \right\|_{{L}^1(\mathbb{M}^n)} \right) \qquad \forall t>0 \, .
\end{equation}
\end{proposition}
\begin{proof}
Given $t>0$, we consider the sequence of time steps $t_j:=(1-2^{-j})t$, for all $ j \in \mathbb{N} $, so that $t_0=0$ and $t_{\infty}=t$. Associated with $ \{ t_j \}_{j \in \mathbb{N}}$, we take an increasing sequence of exponents $ \{ p_j \}_{j \in \mathbb{N}} $ to be defined later, such that $p_0 \ge 2$ and $p_{\infty}=\infty$. Throughout, we will work with the approximate solutions $ \{ \rho_{\varepsilon,k} \}_{\varepsilon>0,k \in \mathbb{N} } $ defined in the proof of Proposition \ref{exunimain}, so that the computations we will perform below are justified. The key starting point consists of multiplying the differential equation in \eqref{pme-approx} by the $ (p_j-1) $-th power of $\rho_{\varepsilon,k}$, integrating by parts in $ D_k\times[t_j,t_{j+1}]$, using \eqref{power-assumption} (only the bound from below) and \eqref{est-peps-3} along with \eqref{eq: non-exp-k} and \eqref{cond-eps-one}, so as to obtain
\begin{equation}\label{la3.2}
\begin{aligned}
\frac{4 \, c_0 \, m \, p_j \,(p_j-1)}{(m+p_j-1)^2} \, \int_{t_j}^{t_{j+1}} \int_{D_k} \left| \nabla \! \left(\rho_{\varepsilon,k}^{{(m+p_j-1)}/{2}}\right) \right|^2 \d\mathcal{V} \d t & \le \\
 p_j \, (p_j-1) \, \int_{t_j}^{t_{j+1}} \int_{D_k}\rho_{\varepsilon,k}^{p_j-2} \, P^\prime_\varepsilon(\rho_{\varepsilon,k}) \left| \nabla \rho_{\varepsilon,k} \right|^2 \d\mathcal{V} \d t & = \left\| \rho_{\varepsilon,k}(t_j) \right\|_{p_j}^{p_j} - \left\| \rho_{\varepsilon,k}(t_{j+1}) \right\|_{p_j}^{p_j} \le  \Vert \rho_{\varepsilon,k}(t_{j})\Vert_{p_j}^{p_j} \, .
\end{aligned}
\end{equation}
For readability's sake, we set $ \| \cdot \|_{L^p(D_k)} = \| \cdot \|_{p}  $. Before proceeding further, it is convenient to recall (see \cite[Theorem 3.1]{BCLS}) that the Sobolev-type inequality \eqref{Sob} can equivalently be rewritten in a ``Gagliardo-Nirenberg'' form as
\begin{equation}\label{NGN}
\begin{gathered}
\left\| f \right\|_{L^r(\mathbb{M}^n)} \le \widetilde{C}_S \left( \left\| \nabla{f} \right\|_{L^2(\mathbb{M}^n)} + \left\| f \right\|_{L^2(\mathbb{M}^n)}  \right)^{\vartheta(s,r,n)} \left\| f \right\|_{L^s(\mathbb{M}^n)} ^{1-\vartheta(s,r,n)} \qquad \forall f \in W^{1,2}(\mathbb{M}^n) \cap L^s(\mathbb{M}^n) \\[0.15cm]
\text{for every} \ 0 < s < r \le 2^\star \, , \qquad \text{where } \vartheta =  \vartheta(s,r,N) := \frac{2n\,(r-s)}{r\,[2n-s(n-2)]} \in (0,1) 
\end{gathered}
\end{equation}
and $ \widetilde{C}_S $ is another positive constant that can be taken independent of $ r,s $. Taking advantage of Young's inequality, it is not difficult to show that \eqref{NGN} implies
\begin{equation}\label{NGN-bis}
\begin{gathered}
\left\| f \right\|_{L^r(\mathbb{M}^n)} \le \widetilde{C}_S \left( \left\| \nabla{f} \right\|_{L^2(\mathbb{M}^n)} + \left\| f \right\|_{L^s(\mathbb{M}^n)}  \right)^{\vartheta(s,r,n)} \left\| f \right\|_{L^s(\mathbb{M}^n)} ^{1-\vartheta(s,r,n)} \qquad \forall f \in W^{1,2}(\mathbb{M}^n) \cap L^s(\mathbb{M}^n)
\\[0.15cm]
\text{for every} \ 0 < s < r \le 2^\star \ \text{with} \ s \le 2 \, ,
\end{gathered}
\end{equation}
for a possibly different positive constant $ \widetilde{C}_S $ as above that we do not relabel. We are now in position to handle the l.h.s.~of \eqref{la3.2} by applying \eqref{NGN-bis} to the function 
$$ f = \rho_{\varepsilon,k}^{{(m+p_j-1)}/{2}}(t) \, , $$
which yields (we can suppose that the solution is not identically zero)
\begin{equation}\label{la3.2-bis}
\begin{aligned}
& \, \frac{2 \, c_0 \, m \, p_j \,(p_j-1)}{\widetilde{C}_S^{\frac{2}{\vartheta}} \, (m+p_j-1)^2} \bigintsss_{t_j}^{t_{j+1}} \frac{\left\| \rho_{\varepsilon,k}(t) \right\|_{{r(m+p_j-1)}/{2}}^{{(m+p_j-1)}/{\vartheta}}}{\left\| \rho_{\varepsilon,k}(t) \right\|^{{(1-\vartheta)(m+p_j-1)}/{\vartheta}}_{{s(m+p_j-1)}/{2}} } \, \d t \\
\leq & \, \left\| \rho_{\varepsilon,k}(t_j) \right\|_{p_j}^{p_j} + \frac{4 \, c_0 \, m \, p_j \,(p_j-1)}{(m+p_j-1)^2} \, \int_{t_j}^{t_{j+1}} \left\| \rho_{\varepsilon,k}(t) \right\|^{m+p_j-1}_{{s(m+p_j-1)}/{2}} \, \d t \, .
\end{aligned}
\end{equation}
Upon making the (feasible) choices
$$ s=\frac{2p_j}{m+p_j-1} \, , \qquad r=2+\frac{2s}{n}=2\,\frac{(n+2)p_j+n(m-1)}{n(m+p_j-1)} \, , $$ 
recalling the recursive definition of $ \{ t_j \}_{j\in\mathbb{N}} $ and using \eqref{eq: non-exp-k}, from \eqref{la3.2-bis} we can infer that
\begin{equation}\label{la3.2-ter}
\frac{c_0 \, m \, p_j \,(p_j-1)\, t}{\widetilde{C}_S^{\frac{2}{\vartheta}} \,2^j (m+p_j-1)^2} \, \frac{\left\| \rho_{\varepsilon,k}(t_{j+1}) \right\|_{p_{j+1}}^{p_{j+1}}}{\left\| \rho_{\varepsilon,k}(t_j) \right\|^{{2p_j}/{n}}_{p_j} } \le \left\| \rho_{\varepsilon,k}(t_j)\right\|_{p_j}^{p_j} + \frac{2 \, c_0 \, m \, p_j \,(p_j-1)\,t}{2^j (m+p_j-1)^2} \, \left\| \rho_{\varepsilon,k}(t_j) \right\|^{m+p_j-1}_{p_j} ,
\end{equation}
where $ p_{j} $ is also defined recursively by
\begin{equation}\label{eq:pj}
p_{j+1}=\frac{n+2}{n} \, p_j+m-1 \qquad  \Longrightarrow
\qquad p_{j}= \left[ p_0 + \frac{n(m-1)}{2} \right] \left( \frac{n+2}{n} \right)^j - \frac{n(m-1)}{2} \quad \forall j \in \mathbb{N} \, .
\end{equation}	
From here on, we will denote by $H$ a generic positive constant that depends only on $ c_0,\widetilde{C}_S,n,p_0 $ and is independent of $ m $ ranging in a bounded subset of $ (1,+\infty) $, which may vary from line to line. Hence estimate \eqref{la3.2-ter} can be rewritten as 
\begin{equation}\label{la3.2-quater}
\left\| \rho_{\varepsilon,k}(t_{j+1}) \right\|_{p_{j+1}}^{p_{j+1}} \le H \left( \frac{2^j }{t} \left\| \rho_{\varepsilon,k}(t_j) \right\|_{p_j}^{\frac{n+2}{n}\,p_j} + \left\| \rho_{\varepsilon,k}(t_j)\right\|_{p_j}^{\frac{n+2}{n}\,p_j + m-1} \right) .
\end{equation} 
By combining \eqref{eq: non-exp-k}, the monotonicity of $ \{ p_j \}_{j \in \mathbb{N}} $, interpolation and Young's inequalities, we easily obtain:
$$ \left\| \rho_{\varepsilon,k}(t_j) \right\|_{p_j} \le \left\| \rho_0 \right\|_\infty + \left\| \rho_0 \right\|_{p_0} , $$
whence from \eqref{la3.2-quater} there follows
\begin{equation}\label{la3.2-young}
\left\| \rho_{\varepsilon,k}(t_{j+1}) \right\|_{p_{j+1}} \le H^{\frac{j+1}{p_{j+1}}} \left[ t^{-1} + \left( \left\| \rho_0 \right\|_\infty + \left\| \rho_0 \right\|_{p_0} \right)^{m-1} \right]^{\frac{1}{p_{j+1}}} \left\| \rho_{\varepsilon,k}(t_j)\right\|_{p_j}^{\frac{n+2}{n} \frac{p_j}{p_{j+1}}} .
\end{equation} 
Iterating \eqref{la3.2-young} and exploiting again \eqref{eq: non-exp-k} (in the l.h.s.~of \eqref{la3.2-young}) yields
\begin{equation*}
\left\| \rho_{\varepsilon,k}(t) \right\|_{p_{j+1}} 
\le  H^{\frac{\sum_{h=1}^{j+1} h \left( \frac{n+2}{n} \right)^{j+1-h} }{p_{j+1}}} \left[ t^{-1} + \left( \left\| \rho_0 \right\|_\infty + \left\| \rho_0 \right\|_{p_0} \right)^{m-1} \right]^{\frac{\sum_{h=0}^{j} \left( \frac{n+2}{n} \right)^h }{p_{j+1}}} \left\| \rho_0 \right\|_{p_0}^{\left(\frac{n+2}{n}\right)^{j+1} \frac{p_0}{p_{j+1}}} ;
\end{equation*} 
by letting $ j \to \infty $, recalling \eqref{eq:pj}, we thus end up with
\begin{equation*}
\left\| \rho_{\varepsilon,k}(t) \right\|_{\infty} \le H \left[ t^{-1} + \left( \left\| \rho_0 \right\|_\infty + \left\| \rho_0 \right\|_{p_0} \right)^{m-1} \right]^{\frac{n}{2p_0+n(m-1)}} \left\| \rho_0 \right\|_{p_0}^{\frac{2p_0}{2p_0+n(m-1)}} ,
\end{equation*}
whence
\begin{equation}\label{la3.2-young-limit-bis}
\left\| \rho_{\varepsilon,k}(t) \right\|_{\infty} \le H \left[ t^{-\frac{n}{2p_0+n(m-1)}} + \left( \left\| \rho_0 \right\|_\infty + \left\| \rho_0 \right\|_{p_0} \right)^{\frac{n(m-1)}{2p_0+n(m-1)}} \right] \left\| \rho_0 \right\|_{p_0}^{\frac{2p_0}{2p_0+n(m-1)}} .
\end{equation}
At this point we need to take advantage of the following version of Young's inequality:
\begin{equation*}\label{young-first}
A^{\theta} \, B^{1-\theta} \le \epsilon \, \theta \, A + \epsilon^{-\frac{\theta}{1-\theta}} \, (1-\theta) \, B  \qquad \forall A,B,\epsilon > 0 \, , \quad \forall \theta \in (0,1) \, .
\end{equation*} 
Upon choosing
$$
A= \left\| \rho_0 \right\|_\infty + \left\| \rho_0 \right\|_{p_0} , \quad B= \left\| \rho_0 \right\|_{p_0} , \quad \theta = \frac{n(m-1)}{2p_0+n(m-1)} \, , \quad \epsilon = \left( H \, \theta \, 2^{1+\frac{\theta}{m-1}} \right)^{-1} ,
$$
from \eqref{la3.2-young-limit-bis} we infer that
\begin{equation}\label{est-iter-1}
\left\| \rho_{\varepsilon,k}(t) \right\|_{\infty} \le \frac{\left\| \rho_0 \right\|_\infty}{2^{1+\frac{\theta}{m-1}}} + H \, t^{-\frac{\theta}{m-1}} \left\| \rho_0 \right\|_{p_0}^{1-\theta} + \left[ 2^{-1-\frac{\theta}{m-1}} + H \left(H \, \theta \, 2^{1+\frac{\theta}{m-1}} \right)^{\frac{\theta}{1-\theta}} \right] \left\| \rho_0 \right\|_{p_0} ;
\end{equation}
since $ \theta $ stays bounded away from $ 1 $ and $ {\theta}/{(m-1)} $ stays bounded as $ m $ ranges in a bounded subset of $ (1,+\infty) $, we can equivalently rewrite \eqref{est-iter-1} as  
\begin{equation}\label{est-iter-2}
\left\| \rho_{\varepsilon,k}(t) \right\|_{\infty} \le \frac{\left\| \rho_0 \right\|_\infty}{2^{1+\frac{\theta}{m-1}}} + H \, t^{-\frac{\theta}{m-1}} \left\| \rho_0 \right\|_{p_0}^{1-\theta} + H \left\| \rho_0 \right\|_{p_0} \qquad \forall t>0 \, .
\end{equation}
In order to remove the dependence of the r.h.s.~of \eqref{est-iter-2} on $ \left\| \rho_0 \right\|_\infty $, we can use a time-shift argument, namely for each $ j \in \mathbb{N} $ we consider \eqref{est-iter-2} evaluated at $ t \equiv t/2^j $ with time origin shifted from $ 0 $ to $ t/2^{j+1} $ (we implicitly rely on the uniqueness of energy solutions). This, along with \eqref{eq: non-exp-k}, ensures that
\begin{equation}\label{est-iter-3}
\left\| \rho_{\varepsilon,k} \big(t/2^j\big) \right\|_{\infty} \le \frac{\left\| \rho_{\varepsilon,k} \big (t/2^{j+1} \big) \right\|_{\infty}}{2^{1+\frac{\theta}{m-1}}} + 2^{\frac{\theta(j+1)}{m-1}} H \, t^{-\frac{\theta}{m-1}} \left\| \rho_0 \right\|_{p_0}^{1-\theta} + H \left\| \rho_0 \right\|_{p_0} \qquad \forall j \in \mathbb{N} \, .
\end{equation}
By iterating \eqref{est-iter-3} from $ j=0 $ to $ j=J \in \mathbb{N} $, we obtain:
\begin{equation*}
\left\| \rho_{\varepsilon,k}(t) \right\|_{\infty} \le \frac{\left\| \rho_0 \right\|_{\infty}}{2^{\left( 1+\frac{\theta}{m-1} \right)(J+1)}} + 2^{\frac{\theta}{m-1}} H \, t^{-\frac{\theta}{m-1}}  \left\| \rho_0 \right\|_{p_0}^{1-\theta} \, \sum_{j=0}^{J} 2^{-j} + H \left\| \rho_0 \right\|_{p_0} \, \sum_{j=0}^J 2^{-\left( 1+\frac{\theta}{m-1} \right)j} \, ,
\end{equation*}  
so that taking limits as $ J \to \infty $ yields 
\begin{equation}\label{est-iter-limit}
\left\| \rho_{\varepsilon,k}(t) \right\|_{\infty} \le H \left( t^{-\frac{\theta}{m-1}} \left\| \rho_0 \right\|_{p_0}^{1-\theta} + \left\| \rho_0 \right\|_{p_0} \right) \qquad \forall t>0 \, .
\end{equation} 
We finally need to extend the just proved estimate to the case $ p_0=1 $, the one we are primarily interested in. Given any $ p_0 \ge 2 $ as above (fixed), let us plug the interpolation inequality 
$$ \left\| \rho_0 \right\|_{p_0} \le \left\| \rho_0 \right\|_{\infty^{\phantom{a}}}^{1-\frac{1}{p_0}} \left\| \rho_0 \right\|_1^{\frac{1}{p_0}} $$
into \eqref{est-iter-limit}:
\begin{equation}\label{case-1-smooth-m2-ext-proof-1}
\begin{gathered}
\left\| \rho_{\varepsilon,k}(t) \right\|_{\infty} \leq C \left\| \rho_0 \right\|_{\infty^{\phantom{a}}}^{\frac{2(p_0-1)}{2p_0+n(m-1)}} \left( t^{-\frac{n}{2p_0+n(m-1)}} \left\| \rho_0 \right\|_{1}^{\frac{2}{2p_0+n(m-1)}} + \left\| \rho_0 \right\|_{\infty^{\phantom{a}}}^{\frac{n(m-1)(p_0-1)}{p_0[2p_0+n(m-1)]}} \left\| \rho_0 \right\|_{1}^{\frac{1}{p_0}} \right) \quad \forall t>0 \, , \\
\end{gathered}
\end{equation}
where $ C $ stands for a generic positive constant as in the statement. By exploiting again a time-shift argument, it is readily seen that \eqref{case-1-smooth-m2-ext-proof-1} entails, for all $ j \in \mathbb{N} $,
\begin{equation}\label{case-1-smooth-m2-ext-proof-2}
\begin{aligned}
 \left\| \rho_{\varepsilon,k}\big(t/2^j\big) \right\|_{\infty} \leq & \, 2^{\frac{n(j+1)}{2p_0+n(m-1)}} \, C \left\| \rho_{\varepsilon,k}\big(t/2^{j+1}\big) \right\|_{\infty}^{\frac{2(p_0-1)}{2p_0+n(m-1)}}  \\
& \times \left( t^{-\frac{n}{2p_0+n(m-1)}} \left\| \rho_0 \right\|_{1}^{\frac{2}{2p_0+n(m-1)}} + \left\| \rho_0 \right\|_{\infty^{\phantom{a}}}^{\frac{n(m-1)(p_0-1)}{p_0[2p_0+n(m-1)]}} \left\| \rho_0 \right\|_{1}^{\frac{1}{p_0}} \right) .
\end{aligned}
\end{equation}
Since
\[
 \frac{2(p_0-1)}{2p_0+n(m-1)} \le 1-\frac{1}{p_0} \, ,
\]
a straightforward iteration of \eqref{case-1-smooth-m2-ext-proof-2} ensures that
\begin{equation}\label{case-1-smooth-m2-ext-proof-4}
\begin{aligned}
\left\| \rho_{\varepsilon,k}(t) \right\|_{\infty} \le & \, C \left( t^{-\frac{n}{2p_0+n(m-1)}} \left\| \rho_0 \right\|_{1}^{\frac{2}{2p_0+n(m-1)}} + \left\| \rho_0 \right\|_{\infty^{\phantom{a}}}^{\frac{n(m-1)(p_0-1)}{p_0[2p_0+n(m-1)]}} \left\| \rho_0 \right\|_{1}^{\frac{1}{p_0}} \right)^{\frac{2p_0+n(m-1)}{2+n(m-1)}} , \\
\le  & \, C \left( t^{-\frac{n}{2+n(m-1)}} \left\| \rho_0 \right\|_{1}^{\frac{2}{2+n(m-1)}} + \left\| \rho_0 \right\|_{\infty^{\phantom{a}}}^{\frac{n(m-1)(p_0-1)}{p_0[2+n(m-1)]}} \left\| \rho_0 \right\|_{1}^{\frac{2p_0+n(m-1)}{p_0[2+n(m-1)]}} \right) . 
\end{aligned}
\end{equation}
By applying a Young-type inequality similar to the one that led us to \eqref{est-iter-1}, from \eqref{case-1-smooth-m2-ext-proof-4} we easily deduce that
\begin{equation}\label{case-1-smooth-m2-ext-proof-5}
\begin{aligned}
\left\| \rho_{\varepsilon,k}(t) \right\|_{\infty} \le & \, \frac{\left\| \rho_0 \right\|_\infty}{2^{1+\frac{n}{2+n(m-1)}}} + C \, t^{-\frac{n}{2+n(m-1)}} \left\| \rho_0 \right\|_{1}^{\frac{2}{2+n(m-1)}}  \\
& \, + C \left(C \, \frac{n(m-1)(p_0-1)}{p_0[2+n(m-1)]} \, 2^{1+\frac{n}{2+n(m-1)}} \right)^{\frac{n(m-1)(p_0-1)}{2p_0+n(m-1)}} \left\| \rho_0 \right\|_1  \\
\le & \, \frac{\left\| \rho_0 \right\|_\infty}{2^{1+\frac{n}{2+n(m-1)}}} + C \, t^{-\frac{n}{2+n(m-1)}} \left\| \rho_0 \right\|_{1}^{\frac{2}{2+n(m-1)}} + C \left\| \rho_0 \right\|_1 \qquad \forall t>0  \, .
\end{aligned}
\end{equation}
Estimate \eqref{case-1-smooth-m2-ext-proof-5} is completely analogous to \eqref{est-iter-2}, so that by reasoning in the same fashion we end up with 
\begin{equation}\label{est-iter-limit-L1}
\left\| \rho_{\varepsilon,k}(t) \right\|_{L^\infty(D_k)} \le C \left( t^{-\frac{n}{2+n(m-1)}} \left\| \rho_0 \right\|_{L^1(D_k)}^{\frac{2}{2+n(m-1)}} + \left\| \rho_0 \right\|_{L^1(D_k)} \right) \qquad \forall t>0 \, .
\end{equation}
Recalling the convergence results encompassed by Proposition \ref{exunimain}, the smoothing effect \eqref{smoothing-one} follows by letting $ k \to \infty $ in \eqref{est-iter-limit-L1}, whereas \eqref{smoothing-limit} follows by letting $ \varepsilon\downarrow 0 $ in \eqref{smoothing-one}.
\end{proof}

The next proposition establishes that solutions starting from bounded and compactly-supported data stay with compact support, at least {for short times}. It is a consequence of the power degeneracy of $ P $ induced by assumption \eqref{below-above-prime} (here we need both sides), hence it is a purely \emph{nonlinear} effect. We stress that this property will be crucial in order to show two essential facts: solutions starting from data in $\mathscr{M}_2^M(\mathbb{M}^n)$ belong to $\mathscr{M}_2^M(\mathbb{M}^n)$ for all times and they form a \emph{continuous} curve with values in $ (\mathscr{M}_2^M(\mathbb{M}^n),\mathcal{W}_2) $. 

\begin{proposition}[Compactness of the support]\label{compact-support}
Let $ P$ comply with \eqref{increas} and \eqref{below-above-prime}. Let $ \rho_0 \in L^1(\mathbb{M}^n) \cap L^\infty(\mathbb{M}^n) $ be nonnegative with \emph{compact support}. Then there exist $ t_1>0 $ and a compact set $ B \subset \mathbb{M}^n $, depending on $ \rho_0 , m , c_0 , c_1 $ and $ \mathbb{M}^n $, such that the weak energy solution $ \rho $ to \eqref{pme-reg} satisfies
\begin{equation}\label{eq:comp-supp}
\mathrm{supp} \, \rho(t) \subset  B \qquad \forall t \in \left[0,t_1\right] . 
\end{equation} 
\end{proposition}
\begin{proof}
Since $ \mathbb{M}^n $ is a smooth, complete, connected and noncompact Riemannian manifold, it is well known that it admits a regular exhaustion, namely a sequence of open sets $ D_k \subset \mathbb{M}^n $ such that $ \overline{D}_k $ is a smooth, compact manifold with boundary (for all $ k \in \mathbb{N} $) and there hold
$$
\overline{D}_{k} \Subset D_{k+1} \qquad \text{and} \qquad \bigcup_{k=1}^{\infty} D_k = \mathbb{M}^n \, .
$$
In particular, $ \partial D_k $ is a smooth $(n-1)$-dimensional, compact, orientable submanifold of $ \mathbb{M}^n $, with a natural orientation given by the outward-pointing normal field w.r.t.~$D_k$. For such a construction we refer e.g.~to \cite[Proposition 2.28, Theorem 6.10, Propositions 15.24 and 15.33]{Lee}. Given $ \epsilon>0 $, let us define the set of all points inside $ D_k $ whose distance from $ \partial D_k $ is smaller than $ \epsilon $, that is
$$
D_k^\epsilon := \left\{ x \in D_k: \ \mathsf{d}(x,\partial D_k) < \epsilon \right\} .
$$
Since $ \partial D_k $ enjoys the above recalled regularity properties, if $ \epsilon $ is sufficiently small then each $ x \in D_k^\epsilon $ admits a unique projection $ \pi(x) $ onto $ \partial D_k $. Hence every such point is uniquely identified by the pair $ \Pi(x) := (\pi(x),\delta(x)) $, where $ \delta(x) $ is the geodesic distance from $ x $ to $ \pi(x) $ (or equivalently to $ \partial D_k $). Moreover, the map $ \Pi $ is a diffeomorphism between $ D_k^\epsilon $ and $  \partial D_k \times (0,\epsilon) $, so that one can use $ \delta = \delta(x) $ and $ \pi = \pi(x) $ as coordinates that span the whole $ D_k^\epsilon $ (see e.g.~\cite{Foote}). It is not difficult to check that $ \delta $ being a geodesic coordinate, the Laplacian of a regular function $ \phi $ (defined on $ D_k^\epsilon $) that depends only on $ \delta $ reads 
\begin{equation}\label{Lap-Bel}
\Delta \phi (\pi,\delta) = \phi^{\prime\prime}(\delta) + \mathsf{m}(\pi,\delta) \, \phi^\prime(\delta) \qquad \forall (\pi,\delta) \in \partial D_k \times (0,\epsilon) \, ,
\end{equation}
where $ \mathsf{m}(\pi,\delta) $ is also regular (in fact it is the Laplacian of the distance function itself). 

Taking advantage of such framework, first of all we pick $k$ so large that $\mathrm{supp} \, \rho_0 \subset D_{k-1} $ and $ \epsilon>0 $ so small that, alongside with the unique-projection property, there holds $ D_{k-1} \cap D_k^\epsilon =  \emptyset $. Then we define 
$$
\Sigma_\epsilon := \Pi^{-1}\!\left( \partial D_k \times \{ \epsilon \} \right) ,
$$
namely the set of points inside $ D_k $ whose distance to $ \partial D_k $ is equal to $ \epsilon $, which describes a smooth submanifold having analogous properties to $ \partial D_k $ (note that, since one has the right to choose $\epsilon$ arbitrarily small, $ \Pi $ can smoothly be extended up to $ \partial D_k \times \{ \epsilon \}  $). We also define $ \Omega_\epsilon $ to be the regular domain enclosed by $ \Sigma_\epsilon $. Now let us consider the Cauchy-Dirichlet problem
\begin{equation}\label{eq-sup-pb}
\begin{cases}
\partial_t u = \Delta P(u)  & \text{in } \mathbb{M}^n \setminus \Omega_\epsilon \times (0,t_1) \, , \\ 
u = \| \rho_0 \|_\infty & \text{on } \Sigma_\epsilon \times (0,t_1) \, , \\
u = 0 & \text{on } \mathbb{M}^n \setminus \Omega_\epsilon \times \{ 0 \} \, ,
\end{cases}
\end{equation}
where $ t_1>0 $ is a small enough time to be chosen later. Since $ \rho \le \| \rho_0 \|_\infty $ in $ \mathbb{M}^n \times \mathbb{R}^+ $ and $ \mathrm{supp} \, \rho_0 \subset \Omega_\epsilon $, it is apparent that $ \rho $ is a \emph{subsolution} of \eqref{eq-sup-pb}. Our aim is to construct a \emph{supersolution} which depends spatially only on $\delta$ and has compact support for all $ t \in [0,t_1] $. The candidate profile is modeled after Euclidean planar \emph{traveling waves} for the porous medium equation, see \cite[Section 4.3]{V07}. That is, we consider the following function:
\begin{equation}\label{cs1}
\overline{u}(\delta,t) := P^{-1} \! \left( \left[ C_1 \left( C_2 t + \delta - \frac{\epsilon}{2} \right)_+ \right]^{\frac{m}{m-1}} \right) \qquad \forall (\delta,t) \in (0,\epsilon] \times [0,t_1] \, ,
\end{equation}
where $ C_1 $ and $ C_2 $ are positive constants to be selected. In view of the assumptions on $ P $, it is not difficult to deduce the following inequalities:
\begin{equation}\label{p-inv}
\left( \frac{v}{c_1} \right)^{\frac{1}{m}} \le P^{-1}(v) \le \left( \frac{v}{c_0} \right)^{\frac{1}{m}}  \qquad \forall v \ge 0 \, ,
\end{equation}
\begin{equation}\label{p-inv-prime}
\left[ P^{-1} \right]^\prime\!(v) \ge \frac{c_0^{1-\frac{1}{m}}}{c_1 \, m}  \, v^{\frac{1}{m}-1} \qquad \forall v > 0 \, .
\end{equation}
Clearly $ \overline{u}(\delta,0) \ge 0 $ and, thanks to \eqref{p-inv}, 
$$
\overline{u}(\epsilon,t) \ge c_1^{-\frac{1}{m}} \left[ C_1 \, \frac{\epsilon}{2} \right]^{\frac{1}{m-1}} \qquad \forall t \ge 0 \, ;
$$
hence a first requirement to make sure that $ \overline{u} $ complies with the boundary condition in \eqref{eq-sup-pb} is
\begin{equation}\label{r1}
C_1 \ge \frac{2}{\epsilon} \, c_1^{\frac{m-1}{m}} \, \| \rho_0 \|_\infty^{m-1} \, .
\end{equation}
Let us now compute the derivatives of $ \overline{u} $ and $ P(\overline{u})$ needed to construct a supersolution:
\begin{equation}\label{deriv-1}
\begin{aligned}
\partial_t \overline{u}(\delta,t)  = &  \, C_2 \, C_1^{\frac{m}{m-1}} \frac{m}{m-1} \left( C_2 t + \delta - \frac{\epsilon}{2} \right)_+^{\frac{1}{m-1}} \left[ P^{-1} \right]^\prime\!\left( \left[ C_1 \left( C_2 t + \delta - \frac{\epsilon}{2} \right)_+ \right]^{\frac{m}{m-1}}  \right) \\ 
\stackrel{\eqref{p-inv-prime}}{\ge} & \, C_2 \, C_1^{\frac{1}{m-1}} \frac{c_0^{\frac{m-1}{m}}}{(m-1)c_1} \left( C_2 t + \delta - \frac{\epsilon}{2} \right)_+^{\frac{2-m}{m-1}} \, ,
\end{aligned}
\end{equation}
\begin{equation}\label{deriv-2}
\partial_\delta \! \left( P(\overline{u}) \right) \! (\delta,t) =  C_1^{\frac{m}{m-1}} \, \frac{m}{m-1} \left( C_2 t + \delta - \frac{\epsilon}{2} \right)_+^{\frac{1}{m-1}} ,
\end{equation}
\begin{equation}\label{deriv-3}
\partial_{\delta \delta}\! \left( P(\overline{u}) \right) \! (\delta,t) =  C_1^{\frac{m}{m-1}} \, \frac{m}{(m-1)^2} \left( C_2 t + \delta - \frac{\epsilon}{2} \right)_+^{\frac{2-m}{m-1}} .
\end{equation}
We pick $t_1$ in such a way that the distance of the support of $ \overline{u} $ from $ \partial D_k $ is not smaller than $ \epsilon/4 $ for all $ t \in [0,t_1] $, namely
\begin{equation}\label{r2}
t_1 = \frac{\epsilon}{4C_2} \, .
\end{equation}
Let $ \sigma $ denote the maximum of $ \mathsf{m}(\pi,\delta) $ in the region $ E_\epsilon := \partial D_k \times [\epsilon/4,\epsilon] $. Because $ \overline{u} $ is nondecreasing in $ \delta $ and \eqref{r2} ensures that the support of $ \overline{u} $ lies in $  E_\epsilon $, in order to guarantee that the latter is a (weak) supersolution of the differential equation in \eqref{eq-sup-pb} it suffices to ask that (recalling \eqref{Lap-Bel}) 
\begin{equation}\label{cs2}
\partial_t \overline{u}(\delta,t) \ge \partial_{\delta \delta} P(\overline{u}) (\delta,t) + \sigma \, \partial_\delta P(\overline{u}) (\delta,t) \qquad \forall (\delta,t) \in [\epsilon/4,\epsilon] \times [0,t_1] \, .
\end{equation}
Thanks to \eqref{deriv-1}--\eqref{deriv-3}, after some simplifications we find that \eqref{cs2} holds if
\begin{equation*}\label{cs3}
C_2 \, \frac{c_0^{\frac{m-1}{m}}}{(m-1)c_1} \ge C_1 \, \frac{m}{(m-1)^2} \left[ 1 + (m-1) \sigma \left( C_2 t + \delta - \frac{\epsilon}{2} \right)_+  \right] \qquad \forall (\delta,t) \in [\epsilon/4,\epsilon] \times [0,t_1] \, ,
\end{equation*}
the latter inequality being in turn implied by 
\begin{equation}\label{cs4}
C_2 \ge C_1 \, \frac{c_1 \, m}{(m-1)c_0^{\frac{m-1}{m}}} \left[ 1 + \frac{3(m-1)\sigma \epsilon}{4} \right] .
\end{equation}
Hence by choosing $C_1$ as in \eqref{r1}, $C_2$ as in \eqref{cs4} and finally $t_1$ as in \eqref{r2}, we infer that \eqref{cs1} is indeed a supersolution of \eqref{eq-sup-pb} (obviously extended in $ \mathbb{M}^n \setminus D_k $). By comparison we can therefore assert that $ \rho \le \overline{u} $ in $ \mathbb{M}^n \setminus \Omega_\epsilon \times [0,t_1] $, which yields \eqref{eq:comp-supp} with $ B=\overline{D}_{k} $. 

As concerns the comparison principle we have just applied, let us point out that in order to justify it rigorously one should know a priori that $ \rho $ is also a strong solution, namely that it has an $ L^1(\mathbb{M}^n) $ time derivative: see \cite[Section 8.2]{V07}, we refer in particular to the analogue of \cite[Lemma 8.11]{V07} in our framework. On the other hand $ \overline{u} $ is a strong supersolution by construction. To circumvent this issue, it is enough (for instance) to exploit the fact that $ \rho $ can always be seen as the limit of solutions $ \rho_j $ to homogeneous Dirichlet problems set up on each $ D_{j} $ (recall the proof of Proposition \ref{exunimain}). Since every $ \rho_j $ is a strong solution in $ D_j $ (see e.g.~\cite[Corollary 8.3]{V07} in the Euclidean setting) and $ \overline{u} $ clearly satisfies homogeneous Dirichlet boundary conditions on $ \partial D_j $ for $ j $ large enough, one obtains $ \rho_j \le \overline{u} $ in $ D_j \setminus \Omega_\epsilon \times [0,t_1] $ for every $ j \in \mathbb{N} $ by proceeding as above, and then lets $ j \to \infty $. 
\end{proof}

\subsection{Variational solutions, {linearized and adjoint equation}}\label{sec:Variational solutions}

For the purposes of proving Theorem \ref{main-result}, we first introduce a suitable (variational) notion of solution of the approximate problem \eqref{pme-approx} and we show its equivalence with the notion of weak energy solution discussed in the previous subsection.
Hereafter we identify $\H$ with its dual $\H'$ and consider the following Hilbert triple:
$$
\V \hookrightarrow \H \equiv \H' \hookrightarrow \V'.  
$$
Problem \eqref{pme-approx} reads
\begin{equation}\label{pme-evol-approx}
\frac{\d}{\d t}\rho = \Delta(P_\varepsilon(\rho)) \, , \qquad \rho(0) = \rho_0 \, ,
\end{equation}
where $ \rho $ is seen as a curve with values in $ \H $ and, accordingly, $\Delta$ is the realization of the (self-adjoint) Laplace-Beltrami operator in $\H$. In agreement with the notations of Subsection \ref{sect:not}, for every $\varepsilon > 0$ and $ T>0 $ we recall the definition of the set $\mathcal{ND}(0,T)$ associated with $P_\varepsilon$: 
\begin{equation*}
\mathcal{ND}_{P_\varepsilon}(0,T):= \left\lbrace u \in W^{1,2}((0,T);\H) \cap C^1([0,T];\V'): \ u \ge 0 \, , \ P_\varepsilon(u) \in L^2((0,T);\D)  \right\rbrace.
\end{equation*}
Note that the nonlinearity $P_{\varepsilon}$ falls within the class of functions considered in \cite[Subsection 3.3]{AMS}, in the more general framework of Dirichlet forms.

\begin{definition}[Strong variational solutions]\label{def.strong-var-sol}
Let $P$ comply with \eqref{increas} and $P_\varepsilon$ ($ \varepsilon>0 $) be defined by \eqref{P-app}. Let $ \rho_0 \in \H $, with $ \rho_0 \ge 0 $, and $T>0$. We say that a curve $\rho \in W^{1,2}((0,T); \V,\V')$, with $ \rho \ge 0 $, is a \emph{strong variational solution} of \eqref{pme-evol-approx} in the time interval $ (0,T) $ if there holds
\begin{equation}\label{eee}
-\leftidx{_{\mathbb{V}^\prime}}{\!\left\langle \tfrac{\d}{\d t} \rho(t) ,  \eta \right\rangle}_{\mathbb{V}} = \int_{\mathbb{M}^n} \left\langle \nabla P_\varepsilon(\rho(t)) \, , \nabla\eta \right\rangle \d\mathcal{V}  \qquad \text{for a.e. } t \in (0,T) \, , \quad \forall \eta \in \mathbb{V} \, ,
\end{equation}
and $ \lim_{t \downarrow 0} \rho(t)=\rho_0 $ in $\mathbb{H}$.
\end{definition} 
We point out that Definition \ref{def.strong-var-sol} does make sense since $ \rho \in W^{1,2}((0,T); \V,\V') $ implies $ \rho \in C([0,T];\H) $, see \cite[formula (3.28)]{AMS} (this is indeed a rather general fact).

\smallskip 

The following well-posedness result is established by \cite[Theorem 3.4]{AMS}.

\begin{proposition}[Existence of strong variational solutions]\label{prop: sol AMS}
Let $P$ comply with \eqref{increas} and $P_\varepsilon $ ($ \varepsilon>0 $) be defined by \eqref{P-app}. Let $ \rho_0 \in \H $, with $ \rho_0 \ge 0 $, and $T>0$. Then there exists a \emph{unique} strong variational solution of \eqref{pme-evol-approx}, in the sense of Definition \ref{def.strong-var-sol}. If in addition $ \rho_0 \in \mathbb{V}$ then $\rho\in \mathcal{ND}_{P_\varepsilon}(0,T)$.
\end{proposition}

Weak energy solutions and strong variational solutions in fact coincide. 

\begin{proposition}[Equivalent notions of solution]\label{oleinik}
Let $P$ comply with \eqref{increas} and $P_\varepsilon $ ($ \varepsilon>0 $) be defined by \eqref{P-app}. Let $T>0$. Then for any nonnegative $\rho_0\in L^1(\mathbb{M}^n)\cap L^{\infty}(\mathbb{M}^n)$ the weak energy solution of \eqref{pme-approx} (provided by Proposition \ref{exunimain}) and the strong variational solution of \eqref{pme-evol-approx} (provided by Proposition \ref{prop: sol AMS}) are equal, up to $t=T>0$.
\end{proposition}
\begin{proof}
Let us denote by $ \hat{\rho} $ the solution constructed in Proposition \ref{prop: sol AMS}. Thanks to the integrability properties of $ \hat{\rho} $ and the $C^1$ regularity of the map $ \rho \mapsto P_\varepsilon(\rho) $, we know that $ \hat{\rho} \in L^2(\mathbb{M}^n \times (0,T)) $, which is equivalent to $ P_\varepsilon(\hat{\rho}) \in L^2(\mathbb{M}^n \times (0,T)) $, and $ \nabla \hat{\rho} \in L^2(\mathbb{M}^n \times (0,T)) $, which is equivalent to $ \nabla P_\varepsilon(\hat{\rho}) \in L^2(\mathbb{M}^n \times (0,T)) $. By \eqref{eee}, for any curve $ \eta \in W^{1,2}((0,T);W^{1,2}(\mathbb{M}^n)) $ with $ \eta(T)=0 $ there holds 
\begin{equation}\label{eee-3}
-\leftidx{_{\mathbb{V}^\prime}}{\!\left\langle \tfrac{\d}{\d t} \hat{\rho}(t) ,  \eta(t) \right\rangle}_{\mathbb{V}} = \int_{\mathbb{M}^n} \left\langle \nabla P_\varepsilon(\hat{\rho}(t)) \, , \nabla\eta(t) \right\rangle \d\mathcal{V}  \qquad \text{for a.e. } t \in (0,T) \, ;
\end{equation}
since both $ \hat{\rho} $ and $ \eta $ are continuous curves with values in $ L^2(\mathbb{M}^n) $, integrating \eqref{eee-3} between $ t=0 $ and $ t=T $ yields
\begin{equation*}\label{eee2}
\int_0^T \int_{\mathbb{M}^n} \hat{\rho} \,  \partial_t\eta \, \d\mathcal{V} \d t + \int_{\mathbb{M}^n} \rho_0(x) \, \eta(x,0) \,  \d\mathcal{V}(x)  = \int_0^T \int_{\mathbb{M}^n} \left\langle \nabla P_\varepsilon(\hat{\rho}) \, , \nabla\eta \right\rangle \d\mathcal{V} \d t \, ,
\end{equation*}
which shows that $ \hat{\rho} $ is also a weak energy solution of \eqref{pme-approx} starting from $ \rho_0 $ and therefore it coincides with the one provided by Proposition \ref{exunimain}, up to the observations made in the first part of the corresponding proof.
\end{proof}

\begin{remark}[On possibly different constructions of weak energy solutions]\label{R} \rm
In Subsection \ref{weak-sol} we used a well-established approach to prove existence of weak energy solutions of \eqref{pme-reg}, which consists in the first place of solving \emph{evolution} problems associated with nondegenerate nonlinearities on regular domains. As shown above, this technique is suitable to prove several key estimates, especially the smoothing effect of Proposition \ref{smooth-approx}. Nevertheless, we mention that there exists at least another fruitful method, which relies first on solving a \emph{discretized} version (in time) of problem \eqref{pme-reg} by means of the Crandall-Liggett Theorem (see \cite[Chapter 10]{V07} in the Euclidean context). This is precisely the technique employed in \cite[Section 3.3]{AMS} to construct solutions of \eqref{pme-reg} in the general setting considered therein; the advantage of such an approach is that it also works in nonsmooth frameworks (like metric-measure spaces). However, in that case the proof of the smoothing effect is less trivial and should be investigated further (one can no longer differentiate $ L^p $ norms along the flow), for instance by taking advantage of the abstract tools developed in \cite{CH}, which a priori work upon assuming the validity of the stronger \emph{Euclidean} Sobolev inequality \eqref{Sob-euc}.
\end{remark}

To implement the Hamiltonian approach described in the Introduction, it is necessary to study the linearization of \eqref{pme-evol-approx} along with its formal adjoint. More precisely, in the variational setting $\V \hookrightarrow \H \hookrightarrow \V'$ described above, we can consider the \emph{forward linearized} equation
\begin{equation}\label{linear}
\frac{\d}{\d t} w = \Delta\!\left[   P_{\varepsilon}^\prime \!\left(\rho \right) w \right], \qquad w(0) = w_0 \, ,
\end{equation} 
and the \emph{backward adjoint} equation
\begin{equation}\label{adjoint}
\frac{\d}{\d t} \phi =  - P_{\varepsilon}^\prime \!\left(\rho \right) \Delta \phi \, , \qquad \phi(T) = \phi_T \, .
\end{equation}

Following \cite[Theorem 4.5]{AMS}, we begin with rephrasing in our setting a well-posedness result for \eqref{linear}.
Hereafter we denote by $\D'$ the dual of $\D$, recalling that $\H \hookrightarrow \V' \hookrightarrow \D'$ with continuous and dense inclusions.

\begin{theorem}[Forward linearized equation]\label{pb:dual}
Let $P$ comply with \eqref{increas} and $P_\varepsilon $ ($ \varepsilon>0 $) be defined by \eqref{P-app}. Let $T>0$. For every nonnegative $\rho\in L^2((0,T);\mathbb{H})$ and for every $ w_0 \in \mathbb{V}'$, there exists a unique weak solution $w\in W^{1,2}((0,T);\mathbb{H}, \mathbb{D}')$ of \eqref{linear}, in the sense that it satisfies
\begin{equation}\label{pme-approx-s-weak-formulation}
\leftidx{_{\mathbb{V}^\prime}}{\!\left\langle w(r),\theta(r) \right\rangle}_{\mathbb{V}}-\int_{0}^r\int_{\mathbb{M}^n} \left[ \partial_t \theta(t)+P'_{\varepsilon}(\rho(t)) \, \Delta\theta(t) \right] w(t) \, \d\mathcal{V}\d t = \leftidx{_{\mathbb{V}^\prime}}{\!\left\langle w_0,\theta(0) \right\rangle}_{\mathbb{V}} \qquad \forall r \in [0,T]
\end{equation}
for every $\theta\in W^{1,2}((0,T);\mathbb{D},\mathbb{H})$. 
\end{theorem} 

As concerns \eqref{adjoint} we have the following result, whose proof can be found in \cite[Theorem 4.1]{AMS}.

\begin{theorem}[Backward adjoint equation]\label{th: dual-back}
Let $P$ comply with \eqref{increas} and $P_\varepsilon $ ($ \varepsilon>0 $) be defined by \eqref{P-app}. Let $T>0$. For every nonnegative $\rho\in L^2((0,T);\mathbb{H})$ and for every $ \phi_T \in\mathbb{V}$, there exists a unique strong solution $\phi\in W^{1,2}((0,T);\mathbb{D},\mathbb{H})$ of \eqref{adjoint}.
Moreover, if ${\phi_T}\in L^{\infty}(\mathbb{M}^n) \cap \V $ then $\left\|\phi(t)\right\|_{L^\infty(\M^n)} \leq \left\| \phi_T \right\|_{L^{\infty}(\mathbb{M}^n)}$ for every $t\in[0,T]$.
\end{theorem}

\section{Proof of the main results} \label{sec:str}

This section is entirely devoted to proving Theorems \ref{main-result} and \ref{optimal}. After a brief introduction to the strategy of proof of Theorem \ref{main-result} (Subsection \ref{subsec: out}), we will first treat the noncompact case (Subsection \ref{noncomp}) and then shortly address the compact case (Subsection \ref{compact}). Finally, in Subsection \ref{opt-small} we will show that our estimate is optimal for small times, namely Theorem \ref{optimal}.

\subsection{Outline of the strategy}\label{subsec: out}

The idea is to prove the stability estimate \eqref{wass-contr} for a suitable approximation of problem \eqref{pme}, passing to the limit in the approximation scheme only at the very end. Let us briefly sketch the main steps of the proof.
\begin{itemize}
\item[1.]  We firstly consider the ``elliptic'' nonlinearity $P_\varepsilon$ as in \eqref{P-app} and introduce a regular initial density $\rho_0$ belonging to $ L^\infty_c(\M^n) \cap \V $. We denote by $\rho$, $\phi$ and $w$ the solutions of the approximated problems \eqref{pme-approx}, \eqref{linear} and \eqref{adjoint}, respectively (for the moment for simplicity we omit the subscript $\varepsilon$).
\item[2.] We estimate the derivative $\frac{\d}{\d t} \mathcal{E}_{\rho(t)}[\phi(t)]$ of the Hamiltonian functional defined in \eqref{eq:weighted Dirichlet}. Here it is essential to exploit the lower bound on the Ricci curvature in the Bakry-\'Emery form \eqref{eq:BE-c}, which allows us to deduce that (Lemma \ref{deriv-ham-2})
 \begin{equation*}
\frac12 \frac{\d}{\d t} \, \mathcal{E}_{\rho(t)} \! \left[ \phi(t) \right] \ge -K \, \int_{\mathbb{M}^n} \Gamma(\phi(t)) \, P_\varepsilon\!\left( \rho(t) \right) \d\mathcal{V} \, .
\end{equation*}
We then use the smoothing effect provided by Proposition \ref{smooth-approx} to integrate the above  differential inequality; this yields the estimate
\begin{equation*}
\mathcal{E}_{\rho(t)}[\phi(t)] \geq \exp(-K\,C(t,m,n)) \, \mathcal{E}_{\rho_0}[\phi(0)] \, ,
\end{equation*}
where an explicit computation of $C(t,m,n)>0$ is given in Lemma \ref{deriv-ham-chiusa}.

\item[3.] We take a pair of initial data $\rho^0_0, \rho_0^1 \in L^\infty_c(\M^n) \cap \V $ and connect them by a \emph{regular} curve $ \{ \rho^s_0 \}_{s \in [0,1]} $ (in the sense of Definition \ref{regcurve}). For any $ \rho^s_0 $, hereafter $ t \mapsto \rho^s(t) $ will stand for the corresponding solution of \eqref{pme-approx} and $ \phi^s $ for a solution of \eqref{adjoint} with $ \rho \equiv \rho^s $. We then denote by $ (s,x) \mapsto Q_s \varphi(x) $ the (Lipschitz) solution of the Hopf-Lax problem \eqref{eq:hopf-lax} starting from an arbitrary $\varphi \in \mathrm{Lip}_c(\M^n)$ and by $w^s(t) \equiv t \mapsto \tfrac{\d }{\d s} \rho^s(t)$ the solution of the linearized equation \eqref{linear}. For every $t > 0$ we compute the Wasserstein distance $\cW(\rho^0(t),\rho^1(t))$ in the (Kantorovich) formulation recalled by Proposition  \ref{prop: hopf-lax duality} in terms of the Hamiltonian. The ``duality'' relation between $\phi^s$ and $w^s$ (Lemma \ref{formula-dualita}) guarantees that 
\begin{equation*}
\int_{\mathbb{M}^n} Q_1 \varphi \, \rho^1(t) \, \d\mathcal{V} - \int_{\mathbb{M}^n} \varphi \, \rho^0(t) \, \d\mathcal{V} = \int_0^1 \left(  -\frac12 \, \mathcal{E}_{\rho^s(t)} \! \left[ \phi^s(t) \right] + \leftidx{_{\mathbb{V}^\prime}}{\!\left\langle w^s(0),\phi^s(0) \right\rangle}_{\mathbb{V}}  \right) \d s \, ,
\end{equation*}
where the final datum of $ \phi^s $ is given at time $ T \equiv t $ by $\phi^s(t) = Q_s\varphi$. 
\item[4.] By exploiting the regularity of the curve $s \mapsto \rho_0^s \mathcal{V} =: \mu^s $, we can take advantage of the key identity
\begin{equation*}
\int_0^1 \left| \dot{\mu}^s \right|^2 \d s =\int_0^1 \mathcal{E}_{\rho_0^s}^\ast \! \left[ \tfrac{\d}{\d s}\rho_0^s \right]  \d s \, .
\end{equation*}
By combining the latter with the estimate obtained in Step 2 and recalling the definition \eqref{eq:dual weighted Dirichlet} of the (Fenchel) \emph{dual} Hamiltonian $ \mathcal{E}^\ast_\rho $, we can deduce that  
\begin{equation*}
\int_{\mathbb{M}^n} Q_1 \varphi \, \rho^1(t) \, \d\mathcal{V} - \int_{\mathbb{M}^n} \varphi \, \rho^0(t) \, \d\mathcal{V} \leq \frac12  \exp\{K\,C(t,m,n)\}\, \int_0^1 \left| \dot{\mu}^s \right|^2 \d s \, ;
\end{equation*}
this is the content of Lemma \ref{l1}.
\item[5.] We use Lemma \ref{lemma12-2ams}, which ensures that the right-hand side can be made arbitrarily close to the squared Wasserstein distance between $ \rho^0_0 $ and $ \rho^1_0 $ (this in fact implies a further approximation of the initial data). As a consequence, we end up with 
\begin{equation*}
\int_{\mathbb{M}^n} Q_1 \varphi \, \rho^1_\varepsilon(t) \, \d\mathcal{V} - \int_{\mathbb{M}^n} \varphi \, \rho^0_\varepsilon(t) \, \d\mathcal{V} \leq \frac12 \exp\{K\,C(t,m,n)\} \, \cW^2\!\left(\rho_0^0, \rho_0^1\right) ,
\end{equation*} 
where we have reintroduced the dependence on $ \varepsilon $ in view of the last passage to the limit.

\item[6.] By virtue of \eqref{conv-epsilon-bis}, we are allowed to first pass to the limit as $\varepsilon \downarrow 0$ and then take the supremum over all $\varphi \in \mathrm{Lip}_c(\M^n)$, which yields  
\begin{equation*}
\mathcal{W}_2 \! \left( \rho^0(t) , \rho^1(t) \right) \le \exp\{K\, C(t,m,n)\} \, \cW\!\left(\rho_0^0, \rho_0^1\right).
\end{equation*}

\item[7.] We exploit Proposition \ref{compact-support} in order to show that such solutions do belong to $\Mdue$ for all times; here we apply inductively the stability estimate itself in the form $\cW(\rho(t),\rho(t+ \tau))$, for small $ \tau>0 $, along with \eqref{elem}.
Then, upon approximating the initial data, we show that the stability estimate extends to the whole class $\Mdue.$
\item[8.] As a final step, we prove that the solutions constructed above are indeed weak Wasserstein solutions, in the sense of Definition \ref{defsol-w}. This basically follows from the smoothing effect \eqref{smoothing} and the energy inequality \eqref{eest}. Uniqueness of Wasserstein solutions is also a direct consequence of the uniqueness result for weak energy solutions, together with their regularity properties.
\end{itemize}

\subsection{The noncompact case}\label{noncomp}

Throughout this whole subsection we will assume again that $ \mathbb{M}^n $ is in addition noncompact and with infinite volume, hence we will carry out the proof of Theorem \ref{main-result} in this case only. We will then discuss in Subsection \ref{compact} the (simple) modifications required to deal with compact manifolds.

Let $ \rho$ be a weak energy solution of \eqref{pme-approx} and let $\phi$ be a strong variational solution of the associated \emph{backward adjoint} problem, according to Theorem \ref{th: dual-back}.
Upon recalling \eqref{eq:weighted Dirichlet}, we define the \emph{Hamiltonian} functional as
\begin{equation*}\label{def:ham}
\mathcal{E}_{\rho(t)} \! \left[ \phi(t) \right] := \int_{\mathbb{M}^n} \Gamma(\phi(t)) \, \rho(t) \, \d\mathcal{V} \,.
\end{equation*} 
Following \cite{AMS}, we firstly connect the time derivative of the Hamiltonian with the \emph{carr\'e du champ} operators defined in \eqref{gamma} and \eqref{gamma2} (see \cite[Theorem 11.1 and Lemma 11.2]{AMS} for a detailed proof). 

\begin{lemma}\label{deriv-ham}
Let $ P$ comply with \eqref{increas}, $P_\varepsilon $ ($ \varepsilon>0 $) be defined by \eqref{P-app} and $T>0$. Let $ \rho\in \mathcal{ND}_{P_\varepsilon}(0,T)$ be a bounded solution of \eqref{pme-approx}, provided by Proposition \ref{prop: sol AMS}. Let $ \phi\in W^{1,2}((0,T);\mathbb{D},\mathbb{H})$ be a bounded strong solution of \eqref{adjoint}, provided by Theorem \ref{th: dual-back}. Then the map $ t \mapsto \mathcal{E}_{\rho(t)} \! \left[ \phi(t) \right] $ is absolutely continuous in $ [0,T] $ and satisfies the identity
\begin{equation*}\label{deriv-ham-formula}
\frac12 \frac{\d}{\d t} \, \mathcal{E}_{\rho(t)} \! \left[ \phi(t) \right] = \boldsymbol{\Gamma}_2[\phi(t);P_\varepsilon(\rho(t))] + \int_{\mathbb{M}^n} R(\rho(t)) \left( \Delta \phi(t) \right)^2 \d\mathcal{V} \qquad \text{a.e\ in} \ (0,T) \, ,
\end{equation*}
where 
\begin{equation*}
R(\rho):= \rho \left(P_\varepsilon \right)^\prime\!(\rho) - P_\varepsilon(\rho) \qquad \forall \rho \ge 0 \, .
\end{equation*}
\end{lemma}
By requiring the additional assumption \eqref{fund-ineq} on the nonlinearity, we are able to exploit the curvature bound \eqref{ricci-K} in the Bakry-\'Emery form \eqref{eq:BE-c}.
\begin{lemma}\label{deriv-ham-2}
Let the hypotheses of Lemma \ref{deriv-ham} hold. Assume in addition that $\mathbb{M}^n$ ($ n \ge 3 $) complies with \eqref{ricci-K} and $ P $ complies with \eqref{fund-ineq}. Then
\begin{equation}\label{deriv-ham-below}
\frac12 \frac{\d}{\d t} \, \mathcal{E}_{\rho(t)} \! \left[ \phi(t) \right] \ge -K \, \int_{\mathbb{M}^n} \Gamma(\phi(t)) \, P_\varepsilon\!\left( \rho(t) \right) \d\mathcal{V} \qquad \text{a.e\ in} \ (0,T) \, .
\end{equation}
\end{lemma}
\begin{proof}
By combining Lemma \ref{deriv-ham} and the Bakry-\'Emery condition \eqref{eq:BE-c} with $ f \equiv \phi(t) $ and $ \rho \equiv P_\varepsilon\!\left( \rho(t) \right) $, we deduce that
\begin{equation*}\label{eq:dh1}
\begin{aligned}
 \frac12 \frac{\d}{\d t} \, \mathcal{E}_{\rho(t)} \! \left[ \phi(t) \right] \ge & \, -K \, \int_{\mathbb{M}^n} \Gamma(\phi(t)) \, P_\varepsilon\!\left( \rho(t) \right) \d\mathcal{V} \\
 &+ \int_{\mathbb{M}^n} \left[ \rho(t) \left(P_\varepsilon \right)^\prime\!(\rho(t)) - \left(1 - \tfrac{1}{n}\right)P_\varepsilon(\rho(t)) \right] \left( \Delta \phi(t) \right)^2 \d\mathcal{V} \, .
\end{aligned}
\end{equation*}
The conclusion follows upon taking advantage of \eqref{est-peps-4}.
\end{proof} 

If $ K > 0 $ in general it is not clear how to bound the r.h.s.~of \eqref{deriv-ham-below} in terms of the Hamiltonian itself. Nevertheless, if $ P$ complies with \eqref{below-above-prime} and the Sobolev-type inequality \eqref{Sob} holds, the smoothing effect provided by Proposition \ref{smooth-approx} allows us to do so.

\begin{lemma}\label{deriv-ham-chiusa} 
Let $\mathbb{M}^n$ ($ n \ge 3 $) comply with \eqref{ricci-K} and \eqref{Sob}. Let $ P$ comply with \eqref{increas}, \eqref{below-above-prime} and \eqref{fund-ineq}. Let $ T>0 $ and $\rho_\varepsilon \in \mathcal{ND}_{P_\varepsilon}(0,T)$ be the (weak energy) solution of \eqref{pme-approx} corresponding to some nonnegative $ \rho_0 \in L^1(\mathbb{M}^n) \cap L^\infty(\mathbb{M}^n) \cap W^{1,2}(\mathbb{M}^n) $ with $ \| \rho_0 \|_{L^1(\mathbb{M}^n) } = : M $ (recall Proposition \ref{oleinik}), where $P_\varepsilon $ ($ \varepsilon>0 $) is defined by \eqref{P-app}. Let $ \phi\in W^{1,2}((0,T);\mathbb{D},\mathbb{H})$ be a bounded solution of \eqref{adjoint} provided by Theorem \ref{th: dual-back}. Suppose that $ \varepsilon $ is so small that 
\begin{equation*}\label{cond-eps}
\left\| \rho_0 \right\|_{{L}^\infty(\mathbb{M}^n)} \le \frac{1}{\varepsilon} \, .
\end{equation*}
Then
\begin{equation}\label{int-ham-chiusa}
\mathcal{E}_{\rho_\varepsilon(t)} \! \left[ \phi(t) \right] \ge \exp\left\{ -2K \, c_1 \, \mathfrak{C}_{m} \left[ \left(tM^{m-1}\right)^{\frac{2}{2+n(m-1)}}  \vee \left(tM^{m-1}\right) + \frac{\varepsilon}{c_1 \mathfrak{C}_m} t \right]  \right\} \mathcal{E}_{\rho_0} \! \left[ \phi(0) \right] \qquad \forall t \ge 0 \, ,
\end{equation}
where $ C>0 $ is the same constant appearing in \eqref{smoothing-one} and 
\begin{equation}\label{def-cm}
\mathfrak{C}_{m} := C^{m-1} \, 2^{m-2} \left[ 2+n(m-1) \right] .
\end{equation}
\end{lemma}
\begin{proof}
By combining inequalities \eqref{est-peps-1} and \eqref{deriv-ham-below}, we obtain:
\begin{equation}\label{deriv-ham-below-1}
\frac12 \frac{\d}{\d t} \, \mathcal{E}_{ \rho_\varepsilon(t)} \! \left[\phi(t)\right] \ge -K \, \int_{\mathbb{M}^n} \Gamma(\phi(t)) \left[ P(\rho_\varepsilon(t)) + \varepsilon \rho_\varepsilon(t) \right] \d\mathcal{V} \, .
\end{equation}
Thanks to Proposition \ref{smooth-approx}, we know that 
\begin{equation}\label{smoothing-rho}
\begin{aligned}
\left\| \rho_\varepsilon(t) \right\|_{L^\infty\left(\mathbb{M}^n\right)}^{m-1} \le & \, C^{m-1} \left( t^{-\frac{n}{2+n(m-1)}} \left\| \rho_0 \right\|_{{L}^1(\mathbb{M}^n)}^{\frac{2}{2+n(m-1)}} + \left\| \rho_0 \right\|_{{L}^1(\mathbb{M}^n)} \right)^{m-1} \\
= & \, C^{m-1} M^{m-1} \, g_m\!\left( t \, M^{m-1} \right) \qquad \forall t>0 \, ,
\end{aligned}
\end{equation}
where 
\begin{equation*}\label{fm-1}
g_m(s) :=  \left( s^{-\frac{n}{2+n(m-1)}} + 1 \right)^{m-1} \qquad \forall s>0 \, .
\end{equation*}
It is apparent that 
\begin{equation}\label{fm-2}
g_m(s) \le
\begin{cases}
2^{m-1} \, s^{-\frac{n(m-1)}{2+n(m-1)}} & \text{if } s \in (0,1) \, , \\
 2^{m-1} & \text{if } s \ge 1 \, .
\end{cases}
\end{equation}
If we plug \eqref{smoothing-rho} in \eqref{deriv-ham-below-1} and recall that $ P(\rho)/\rho \le c_1 \, \rho^{m-1} $, we find:
\begin{equation}\label{deriv-ham-below-2}
\begin{aligned}
\frac12 \frac{\d}{\d t} \, \mathcal{E}_{\rho_\varepsilon(t)} \! \left[ \phi(t) \right] \ge & - K \, \int_{\mathbb{M}^n} \Gamma(\phi(t)) \, \rho_\varepsilon(t) \left[c_1 \, \rho_\varepsilon(t)^{m-1} + \varepsilon \right] \d\mathcal{V} \\
\ge & -K \left[c_1 \, C^{m-1} M^{m-1} \, g_m\!\left( t \, M^{m-1} \right) + \varepsilon  \right] \mathcal{E}_{\rho_\varepsilon(t)} \! \left[ \phi(t) \right] ;
\end{aligned}
\end{equation}
by integrating \eqref{deriv-ham-below-2} we therefore obtain 
\begin{equation}\label{int-ham-x}
\mathcal{E}_{\rho_\varepsilon(t)} \! \left[ \phi(t) \right] \ge \exp\left\{ -2K \left( c_1 \, C^{m-1} \int_0^{t \, M^{m-1}} g_m(s) \, \d s + \varepsilon t \right)  \right\} \mathcal{E}_{\rho_0} \! \left[ \phi(0) \right] \qquad \forall t \ge 0 \, .
\end{equation}
In order to suitably simplify \eqref{int-ham-x}, by exploiting \eqref{fm-2} we easily infer that 
\begin{equation*}\label{int-ham-xx}
\int_0^\tau g_m(s) \, \d s \le
\begin{cases}
2^{m-1} \, \frac{2+n(m-1)}{2} \, \tau^{\frac{2}{2+n(m-1)}} & \text{if } \tau \in (0,1) \, , \\
2^{m-1} \left[ \tau + \frac{n(m-1)}{2} \right] & \text{if } \tau \ge 1 \, ,
\end{cases}
\end{equation*}
which implies
\begin{equation*}\label{int-ham-xxx}
\int_0^\tau g_m(s) \, \d s \le 2^{m-2} \left[ 2+n(m-1) \right] \left( \tau^{\frac{2}{2+n(m-1)}}  \vee \tau \right) \qquad \forall \tau>0 \, ,
\end{equation*}
whence \eqref{int-ham-chiusa}.
\end{proof} 
 
In the following, we will connect any two (sufficiently regular) initial data $ \rho^0_0 $ and $ \rho^1_0 $ with a \emph{regular} curve $ \{ \rho^s_0 \}_{s \in [0,1]} $ (in the sense of Definition \ref{regcurve}) and consider the corresponding solution $t \mapsto  \rho_\varepsilon^s(t) $ of \eqref{pme-approx} with initial datum $ \rho^s_0 $, that is
\begin{equation}\label{pme-approx-s}
\begin{cases}
\partial_t \rho_\varepsilon^s = \Delta P_\varepsilon(\rho_\varepsilon^s) & \text{in } \mathbb{M}^n \times \mathbb{R}^+ \, , \\
\rho_{\varepsilon}^s(0)= \rho^s_0 & \text{on } \mathbb{M}^n \times \{ 0 \} \,.
\end{cases}
\end{equation} 
Reasoning as in \cite{AMS}, we will exploit the lower bound on the Hamiltonian ensured by Lemma \ref{deriv-ham-chiusa} in order to prove the stability estimate \eqref{wass-contr}. We start by studying the quantity
$$
s \mapsto \int_{\mathbb{M}^n} Q_s \varphi \, \rho_\varepsilon^s(t) \, \d \mathcal{V}  \, ,
$$
where $ \varphi \in \mathrm{Lip}_c(\mathbb{M}^n) $ is arbitrary but fixed and $ [0,1] \times \M^n \ni (s,x) \mapsto Q_s \varphi(x) $ is the (Lipschitz and compactly-supported) solution of the Hopf-Lax problem \eqref{eq:hopf-lax}. To this aim, for (almost) every $s \in (0,1) $  we also introduce the solution $ w^s $ of the linearized equation \eqref{linear} starting from $ \tfrac{\d}{\d s} \rho_0^s$: 
\begin{equation}\label{pme-approx-s-deriv}
\begin{cases}
 \partial_t w^s = \Delta \!\left[ P_\varepsilon^\prime \!\left(\rho_\varepsilon^s \right) w^s \right] & \text{in } \mathbb{M}^n \times \mathbb{R}^+ \, , \\
w^s(0)= \tfrac{\d}{\d s} \rho^s_0 & \text{on } \mathbb{M}^n \times \{ 0 \} \, .
\end{cases}
\end{equation} 
Thanks to Theorem \ref{pb:dual} and Remark \ref{rr}, if $ \{ \rho^s_0 \}_{s \in [0,1]} $ is a regular curve we can guarantee that \eqref{pme-approx-s-deriv} admits a weak solution, at least for almost every $ s \in (0,1) $. Moreover, \cite[Theorem 4.6]{AMS} ensures that $ w^s(t) = \tfrac{\d}{\d s} \rho^s_\varepsilon(t) $.
 \eqref{pme-approx-s-weak-formulation} with initial datum $w_0=\Delta P(\rho_0)$.

\begin{lemma}\label{formula-dualita}
Let $P$ comply with \eqref{increas} and $P_\varepsilon$ ($ \varepsilon>0 $) be defined by \eqref{P-app}. Given a regular curve $ \{ \rho^s_0 \}_{s \in [0,1]} $ and $ T>0 $, let $ \rho_{\varepsilon}^s\in \mathcal{ND}_{P_\varepsilon}(0,T)$ be the corresponding (weak energy) solution of \eqref{pme-approx-s}. Then, for every $  \varphi \in \mathrm{Lip}_c(\mathbb{M}^n) $ and every $ t \in (0,T) $, the map $ s  \mapsto \int_{\mathbb{M}^n} Q_s \varphi \, \rho_\varepsilon^s(t) \, \d\mathcal{V} $ is Lipschitz continuous in $ [0,1] $ and satisfies 
\begin{equation}\label{eq:deriv-Qs}
\frac{\d}{\d s} \int_{\mathbb{M}^n} Q_s \varphi \, \rho_\varepsilon^s(t) \, \d\mathcal{V} = -\frac{1}{2} \, \int_{\mathbb{M}^n} \Gamma\!\left( Q_s \varphi \right) \rho_\varepsilon^s(t) \, \d\mathcal{V} + \leftidx{_{\mathbb{V}^\prime}}{\!\left\langle w^s(t),Q_s \varphi \right\rangle}_{\mathbb{V}} \qquad \text{for a.e. }s \in (0,1) \, ,
\end{equation}
where $ (s,x) \mapsto Q_s \varphi(x) $ is the (Lipschitz and compactly-supported) solution of the Hopf-Lax problem \eqref{eq:hopf-lax} and $ w^s(t) = \tfrac{\d}{\d s} \rho^s_\varepsilon(t) $ is the weak solution of \eqref{pme-approx-s-deriv} provided by Theorem \ref{pb:dual}. 

Moreover, if we denote by $ r : (0,t) \mapsto \phi^s(r) $ the solution of the backward adjoint problem \eqref{adjoint} corresponding to $ \rho \equiv \rho_\varepsilon^s $ with final condition $ \phi^s(t) = Q_s \varphi $, provided by Theorem \ref{th: dual-back}, the following identities hold:
\begin{equation}\label{eq:dual-id}
\leftidx{_{\mathbb{V}^\prime}}{\!\left\langle w^s(t),Q_s \varphi \right\rangle}_{\mathbb{V}} = \leftidx{_{\mathbb{V}^\prime}}{\!\left\langle w^s(t) , \phi^s(t) \right\rangle}_{\mathbb{V}} = \leftidx{_{\mathbb{V}^\prime}}{\!\left\langle w^s(0) , \phi^s(0) \right\rangle}_{\mathbb{V}} = \leftidx{_{\mathbb{V}^\prime}}{\!\left\langle \tfrac{\d}{\d s} \rho^s_0 , \phi^s(0) \right\rangle}_{\mathbb{V}} \quad \text{for a.e. } s \in (0,1) \,  .
\end{equation}
\end{lemma}
\begin{proof}
The map $ s  \mapsto \int_{\mathbb{M}^n} Q_s \varphi \, \rho_\varepsilon^s(t) \, \d\mathcal{V} $ is Lipschitz continuous by virtue of the Lipschitz-continuity of $ (s,x) \mapsto Q_s \varphi (x) $ (plus the boundedness of its support) and the Lipschitz-continuity of the curve $ s \mapsto \rho^s_0 $ with values in $ \V' $ (recall Remark \ref{rr}) along with the fact that the semigroup generated by \eqref{pme-approx-s} turns out to be also a contraction with respect to $ \| \cdot \|_{\V'} $. For more details we refer the reader to \cite[Proof of Theorem 12.5]{AMS}. Once we have observed this, identity \eqref{eq:deriv-Qs} is a direct consequence of \eqref{eq:hopf-lax} and the equality $ w^s(t) = \tfrac{\d}{\d s} \rho^s_\varepsilon(t) $ (for a.e.~$ s \in (0,1) $ independently of $t$), which can rigorously be proved by proceeding as in \cite[Theorem 4.6]{AMS}.

As concerns \eqref{eq:dual-id}, it is enough to observe that it is nothing but formula \eqref{pme-approx-s-weak-formulation} with $ \rho \equiv \rho^s_\varepsilon $, $ w \equiv w^s $ and $ \theta \equiv \phi^s $ (actually with $ r $ and $t$ interchanged).
\end{proof}

\begin{lemma}\label{l1}
Let $\mathbb{M}^n$ ($ n \ge 3 $) comply with assumptions \eqref{ricci-K} and \eqref{Sob}. Let moreover $ P $ comply with assumptions \eqref{increas}, \eqref{below-above-prime}, \eqref{fund-ineq} and $P_\varepsilon $ ($ \varepsilon>0 $) be defined by \eqref{P-app}. Let $ \rho^{0}_{\varepsilon} $ and $ \rho_{\varepsilon}^{1} $ be any two (weak energy) solutions of \eqref{pme-approx} corresponding to the initial data $ \rho^0_0 $ and $ \rho^1_0$, respectively, both nonnegative, belonging to $ {L}^\infty_c(\mathbb{M}^n) \cap W^{1,2}(\mathbb{M}^n)  $ and having the same mass $ M>0 $. Suppose that $ \{ \rho^s_0 \}_{s \in [0,1]} $ is any regular curve (in the sense of Definition \ref{regcurve}) connecting $ \rho^0_0 $ with $ \rho^1_0$, which satisfies
\begin{equation}\label{hp-eps}
\left\| \rho^s_0 \right\|_{L^\infty(\mathbb{M}^n)} \le \frac 1 \varepsilon \qquad \forall s \in [0,1] \, .
\end{equation}
Then for every $  \varphi \in \mathrm{Lip}_c(\mathbb{M}^n) $ there holds
\begin{equation}\label{p4}
\begin{aligned}
& \, \int_{\mathbb{M}^n} Q_1 \varphi \, \rho_\varepsilon^1(t) \, \d\mathcal{V} - \int_{\mathbb{M}^n} \varphi \, \rho_\varepsilon^0(t) \, \d\mathcal{V} \\
\le & \, \frac12 \exp\left\{ 2K \, c_1 \, \mathfrak{C}_{m} \left[ \left(tM^{m-1}\right)^{\frac{2}{2+n(m-1)}}  \vee \left(tM^{m-1}\right) + \frac{\varepsilon}{c_1 \mathfrak{C}_m} t \right]  \right\} \int_0^1 \left| \dot{\mu}^s \right|^2 \d s  \, ,
\end{aligned}
\end{equation}
where $ \mu^s := \rho_0^s \, \mathcal{V}  $ and $ \{ Q_s \varphi \}_{s \in [0,1]} $ is the (Lipschitz and compactly-supported) solution of the Hopf-Lax problem \eqref{eq:hopf-lax} and the constant $ \mathfrak{C}_m $ is defined in \eqref{def-cm}.
\end{lemma}
\begin{proof}
We follow the line of proof of \cite[Theorem 12.5]{AMS}, keeping the same notations as in Lemma \ref{formula-dualita}. By combining \eqref{eq:deriv-Qs} and \eqref{eq:dual-id}, we obtain:
\begin{equation*}\label{p1}
\begin{aligned} 
\int_{\mathbb{M}^n} Q_1 \varphi \, \rho_\varepsilon^1(t) \, \d\mathcal{V} - \int_{\mathbb{M}^n} \varphi \, \rho_\varepsilon^0(t) \, \d\mathcal{V} = & \int_0^1  \left( -\frac12 \int_{\mathbb{M}^n} \Gamma\!\left( \phi^s(t) \right) \rho_\varepsilon^s(t) \, \d\mathcal{V} + \leftidx{_{\mathbb{V}^\prime}}{\!\left\langle \tfrac{\d}{\d s} \rho^s_0 , \phi^s(0) \right\rangle}_{\mathbb{V}} \right) \d s \\
= & \int_0^1 \left(  -\frac12 \, \mathcal{E}_{\rho^s_\varepsilon(t)} \! \left[ \phi^s(t) \right] + \leftidx{_{\mathbb{V}^\prime}}{\!\left\langle \tfrac{\d}{\d s} \rho^s_0 , \phi^s(0) \right\rangle}_{\mathbb{V}}\right) \d s \, .
\end{aligned}
\end{equation*}
Now we can apply, at every $s \in [0,1]$, estimate \eqref{int-ham-chiusa} from Lemma \ref{deriv-ham-chiusa} with $ \rho_\varepsilon(t) \equiv \rho_\varepsilon^s(t) $ and $ \phi(t) \equiv \phi^s(t) $, under assumption \eqref{hp-eps}. This yields, upon recalling \eqref{eq:dual weighted Dirichlet}, 
\begin{equation}\label{p2}
\begin{aligned}
& \int_{\mathbb{M}^n} Q_1 \varphi \, \rho_\varepsilon^1(t) \, \d\mathcal{V} - \int_{\mathbb{M}^n} \varphi \, \rho_\varepsilon^0(t) \, \d\mathcal{V} \\
\le & \bigintsss_0^1 \left( - \frac12 \, e^{-2K \, c_1 \, \mathfrak{C}_{m} \left[ \left(tM^{m-1}\right)^{\frac{2}{2+n(m-1)}}  \vee \left(tM^{m-1}\right) + \frac{\varepsilon}{c_1 \mathfrak{C}_m} t \right]} \mathcal{E}_{\rho^s_0} \! \left[ \phi^s(0) \right] + \leftidx{_{\mathbb{V}^\prime}}{\!\left\langle \tfrac{\d}{\d s} \rho^s_0 , \phi^s(0) \right\rangle}_{\mathbb{V}} \right) \d s \\
= & \, e^{2K \, c_1 \, \mathfrak{C}_{m} \left[ \left(tM^{m-1}\right)^{\frac{2}{2+n(m-1)}}  \vee \left(tM^{m-1}\right) + \frac{\varepsilon}{c_1 \mathfrak{C}_m} t \right]} \int_0^1 \left( -\frac12 \, \mathcal{E}_{\rho^s_0} \! \left[ \psi^{s,t} \right] +\leftidx{_{\mathbb{V}^\prime}}{\!\left\langle \tfrac{\d}{\d s} \rho^s_0 , \psi^{s,t} \right\rangle}_{\mathbb{V}} \right) \d s \\
\le & \, e^{2K \, c_1 \, \mathfrak{C}_{m} \left[ \left(tM^{m-1}\right)^{\frac{2}{2+n(m-1)}}  \vee \left(tM^{m-1}\right) + \frac{\varepsilon}{c_1 \mathfrak{C}_m} t \right]} \int_0^1 \frac12 \, \mathcal{E}_{\rho_0^s}^\ast \! \left[ \tfrac{\d}{\d s} \rho^s_0 \right] \d s \, ,
\end{aligned}
\end{equation}  
where we have set
$$
\psi^{s,t} := e^{-2K \, c_1 \, \mathfrak{C}_{m} \left[ \left(tM^{m-1}\right)^{\frac{2}{2+n(m-1)}}  \vee \left(tM^{m-1}\right) + \frac{\varepsilon}{c_1 \mathfrak{C}_m} t \right]} \, \phi^{s}(0) \, . 
$$ 
Estimate \eqref{p4} thus follows from \eqref{p2} in view of \eqref{eq:key-id-E}.
\end{proof} 

In order to prove Theorem \ref{main-result}, we need first to approximate the geodesic connecting $ \mu_0 $ and $ \hat{\mu}_0 $ in $ (\mathscr{M}_2^M(\M^n),\mathcal{W}_2) $ by regular curves, let $ \varepsilon \to 0 $ in \eqref{pme-approx} and finally pass to the limit in the approximation of the  measures $ \mu_0 $ and $ \hat{\mu}_0 $ by bounded and compactly supported densities as in Lemma \ref{l1}.

\begin{proof}[{Proof of Theorem \ref{main-result} (noncompact case)}]
To begin with, we suppose that $ \mu_0 = \rho_0 \mathcal{V} $ and $ \hat{\mu}_0 = \hat{\rho}_0 \mathcal{V} $, where $ \rho_0 $ and $ \hat{\rho}_0 $ are initial data complying with the assumptions of Lemma \ref{l1}: we will remove this hypothesis only at the very end of the proof. By virtue of Lemma \ref{lemma12-2ams}, we know that there exists a sequence of regular curves $ \{ \rho_j^s \}_{j\in\mathbb{N},s\in[0,1]} $ satisfying \eqref{eq:est-lemma-app-A}--\eqref{eq:improve2} (let $ \rho^1 = \rho_0 $ and $ \rho^0 = \hat{\rho}_0 $ according to the corresponding notations). Given $ \varepsilon>0 $, if we denote by $ t \mapsto (\rho_j^s)_\varepsilon(t) $ each weak energy solution of \eqref{pme-approx} starting from $ \rho_0 \equiv \rho_j^s $, then by Lemma \ref{l1} we know that 
 \begin{equation}\label{p4-1}
 \begin{aligned}
& \, \int_{\mathbb{M}^n} Q_1 \varphi \, (\rho^1_j)_\varepsilon(t) \, \d\mathcal{V} - \int_{\mathbb{M}^n} \varphi \, (\rho^0_j)_\varepsilon(t) \, \d\mathcal{V} \\
\le & \, \frac12 \exp\left\{ 2K \, c_1 \, \mathfrak{C}_{m} \left[ \left(tM^{m-1}\right)^{\frac{2}{2+n(m-1)}}  \vee \left(tM^{m-1}\right) + \frac{\varepsilon}{c_1 \mathfrak{C}_m} t \right]  \right\} \int_0^1 \left| \dot{\mu}^s_j \right|^2 \d s
\end{aligned}
\end{equation}
for every $ \varphi \in \mathrm{Lip}_c(\mathbb{M}^n) $, provided 
\begin{equation}\label{hp-ep-n}
\left\| \rho^s_j \right\|_{L^\infty(\mathbb{M}^n)} \le \frac 1 \varepsilon \qquad \forall s \in [0,1] \, .
\end{equation}
Let us pass to the limit in \eqref{p4-1} as $ j\to\infty $. In the sequel, we denote by $ \rho_\varepsilon $ and $ \hat{\rho}_\varepsilon $ the weak energy solutions of \eqref{pme-approx} starting from $ \rho_0 $ and $ \hat{\rho}_0 $, respectively. Thanks to \eqref{eq:est-lemma-app-B}, \eqref{eq:improve1} (with $ p=1 $) and the $ L^1 $-contraction property \eqref{L1-est} of weak energy solutions, which guarantees that $ (\rho_j^0)_\varepsilon(t) \to \rho_\varepsilon(t) $ and $ (\rho_j^1)_\varepsilon(t) \to \hat{\rho}_\varepsilon(t) $ in $ L^1(\M^n) $, we deduce that
\begin{equation}\label{p4-2}
\begin{aligned}
& \, \int_{\mathbb{M}^n} Q_1 \varphi \, \rho_\varepsilon(t) \, \d\mathcal{V} - \int_{\mathbb{M}^n} \varphi \, \hat{\rho}_\varepsilon(t) \, \d\mathcal{V} \\
\le & \, \frac12 \exp\left\{ 2K  \, c_1 \, \mathfrak{C}_{m} \left[ \left(tM^{m-1}\right)^{\frac{2}{2+n(m-1)}}  \vee \left(tM^{m-1}\right) + \frac{\varepsilon}{c_1 \mathfrak{C}_m} t \right]  \right\} \mathcal{W}_2^2 \! \left( \rho_0 , \hat{\rho}_0 \right) 
\end{aligned}
\end{equation}
upon requiring 
 \begin{equation*}\label{p4-3}
\limsup_{j \to \infty} \sup_{s \in [0,1]} \left\| \rho^s_j \right\|_{{L}^\infty(\mathbb{M}^n)} \leq \frac{1}{2\varepsilon}
\end{equation*}
in view of \eqref{hp-ep-n}, which holds for $ \varepsilon $ small enough thanks to \eqref{eq:improve2}. We are now in position to let $ \varepsilon \downarrow 0 $. The r.h.s.~of \eqref{p4-2} is clearly stable as $ \varepsilon \downarrow 0 $. In order to pass to the limit in the l.h.s.~we need to exploit Proposition \ref{exunimain}: in particular, formula \eqref{conv-epsilon-bis} ensures that $\{ \rho_\varepsilon(t) \}_{\varepsilon>0}$ and $ \{ \hat{\rho}_\varepsilon(t) \}_{\varepsilon>0} $ converge in $ L^1(\mathbb{M}^n) $ to $ \rho(t) $ and $ \hat{\rho}(t) $, respectively, so that \eqref{p4-2} yields
\begin{equation}\label{F1}
\int_{\mathbb{M}^n} Q_1 \varphi \, \rho(t) \, \d\mathcal{V} - \int_{\mathbb{M}^n} \varphi \, \hat{\rho}(t) \, \d\mathcal{V}
\le \frac12 \exp\left\{ 2K \, c_1 \, \mathfrak{C}_{m} \left[ \left(tM^{m-1}\right)^{\frac{2}{2+n(m-1)}}  \vee \left(tM^{m-1}\right) \right]  \right\} \mathcal{W}_2^2 \! \left( \rho_0 , \hat{\rho}_0 \right) .
\end{equation} 
If we take the supremum of the l.h.s.~of \eqref{F1} over all $ \varphi \in \mathrm{Lip}_c(\mathbb{M}^n) $, then by virtue of Proposition \ref{prop: hopf-lax duality} we obtain
\begin{equation}\label{p4-2-ter}
\mathcal{W}_2 \! \left( \rho(t) , \hat{\rho}(t) \right) \le \exp\left\{K \, c_1 \, \mathfrak{C}_{m} \left[ \left(tM^{m-1}\right)^{\frac{2}{2+n(m-1)}}  \vee \left(tM^{m-1}\right) \right]  \right\} \mathcal{W}_2 \! \left( \rho_0 , \hat{\rho}_0 \right) \qquad \forall t>0 \, ,
\end{equation}
namely \eqref{wass-contr} restricted to initial data $ \rho_0 , \hat{\rho}_0 \in L^\infty_c(\mathbb{M}^n) \cap W^{1,2}(\mathbb{M}^n)$. 
It is apparent that estimate \eqref{p4-2-ter} remains true in the wider class $ \rho_0,\hat{\rho}_0 \in L^1(\mathbb{M}^n) \cap L^\infty(\mathbb{M}^n) \cap \mathscr{M}_2^M $: indeed by local regularization and a standard truncation argument, one can pick sequences of nonnegative initial data of mass $ M $ belonging to $ {L}^\infty_c(\mathbb{M}^n) \cap W^{1,2}(\mathbb{M}^n)  $ which converge to $ \rho_0 $ and $\hat{\rho}_0$, respectively, both in $ L^1(\mathbb{M}^n) $ and in $ (\mathscr{M}_2^M(\M^n) ,\mathcal{W}_2) $ (recall Proposition \ref{wass-conv}). Thanks again to \eqref{L1-est}, i.e.~the stability of solutions in $ L^1(\M^n) $, this suffices to pass to the limit in \eqref{F1} and hence in \eqref{p4-2-ter}. 

We still have to prove that $ \rho(t) \in \mathscr{M}_2^M(\M^n) $ for all $ t>0 $, since the mass-conservation property \eqref{cons-mass-th} only ensures that $ \rho(t) \in \mathscr{M}^M(\M^n) $. To this aim, we take advantage of Proposition \ref{compact-support}: from the latter we know that if $ \rho_0 \in L^\infty_c(\mathbb{M}^n) $ then the weak energy solution $ \rho(t) $ of \eqref{pme-reg} stays (uniformly) bounded with (uniform) compact support in a suitable time interval $ [0,t_1] $, so that in particular $ \rho(t) \in \mathscr{M}_2^M(\M^n) $ for all $ t \in [0,t_1] $. Let $ \tau \in (0,t_1] $. Since $ \{ \rho(t+\tau) \}_{t \ge 0} $ is the weak energy solution of \eqref{pme-reg} starting from $ \rho(\tau) \in L^1(\mathbb{M}^n) \cap L^\infty(\mathbb{M}^n) \cap \mathscr{M}_2^M(\M^n) $, estimate \eqref{p4-2-ter} applied to $ \hat{\rho}(t)=\rho(t+\tau) $ guarantees that 
\begin{equation}\label{wass-finite}
\mathcal{W}_2 \! \left( \rho(t), \rho(t+\tau) \right) \le \exp\left\{ K \, c_1 \, \mathfrak{C}_{m} \left[ \left(tM^{m-1}\right)^{\frac{2}{2+n(m-1)}}  \vee \left(tM^{m-1}\right) \right]  \right\} \mathcal{W}_2 \! \left( \rho_0 , \rho(\tau) \right) < \infty  \quad \forall t >0 \, ,
\end{equation} 
whence $ \rho(t) \in \mathscr{M}_2^M(\M^n) $ also for all $ t \in (t_1,2t_1] $ upon recalling \eqref{elem}. It is then clear how one can set up an induction procedure to establish that in fact $ \rho(t) \in \mathscr{M}_2^M(\M^n) $ for \emph{all} $ t>0 $. Furthermore, $ \rho \in C([0,+\infty);\mathscr{M}_2^M(\mathbb{M}^n)) $. Indeed, the just mentioned property of compactness of the support for short times and the $ L^1$-continuity ensured by Proposition \ref{exunimain} easily imply, along with Proposition \ref{wass-conv}, that 
\begin{equation}\label{c-0}
\lim_{t \downarrow 0} \mathcal{W}_2(\rho(t),\rho_0) = 0 \, .
\end{equation} 
Hence by combining \eqref{wass-finite} (understood for all $ t,\tau>0 $) and \eqref{c-0}, we deduce that for every $ t_0 > 0 $ there holds
\begin{equation*}\label{c-t0}
\begin{aligned}
& \lim_{t \to t_0} \mathcal{W}_2(\rho(t),\rho(t_0)) \\
\le & \, \exp\left\{ K \, c_1 \, \mathfrak{C}_{m} \left[ \left(t_0M^{m-1}\right)^{\frac{2}{2+n(m-1)}}  \vee \left(t_0M^{m-1}\right) \right]  \right\} \lim_{t \to t_0} \mathcal{W}_2 \! \left(\rho(|t-t_0|) , \rho_0 \right)  = 0 \, .
\end{aligned}
\end{equation*} 
We have therefore shown the validity of Theorem \ref{main-result} under the additional assumptions $ \mu_0 = \rho_0 \mathcal{V} $ and $ \hat{\mu}_0 = \hat{\rho}_0 \mathcal{V} $ with $ \rho_0 , \hat{\rho}_0 \in L^\infty_c(\mathbb{M}^n) \cap W^{1,2}(\mathbb{M}^n) $. In order to be able to deal with general initial data as in the statement, first of all we take a sequence of nonnegative functions $ (\rho_{j,0},\hat{\rho}_{j,0}) \in [ {L}^\infty_c(\mathbb{M}^n) \cap W^{1,2}(\mathbb{M}^n) ]^2 $ of mass $ M $ such that 
\begin{equation}\label{approx-j}
\lim_{j \to \infty} \rho_{j,0} = \mu_0 \qquad \text{and} \qquad   \lim_{j \to \infty} \hat{\rho}_{j,0} = \hat{\mu}_0  \qquad \text{in } \left( \mathscr{M}_2^M(\M^n) , \mathcal{W}_2 \right) ,
\end{equation}
which exists as a consequence of Definition \ref{regcurve}, Remark \ref{rr} and Lemma \ref{lemma12-2ams} (only applied at the endpoints $ s=0 , 1$): the additional property of the compactness of the support can be obtained again by a straightforward truncation argument. Estimate \eqref{p4-2-ter} applied to the corresponding sequences of solutions, which we denote by $ \{ (\rho_j , \hat{\rho}_j) \}_{j \in \mathbb{N}} $, yields
\begin{equation}\label{wass-contr-j}
\begin{gathered}
\mathcal{W}_2 \!\left(\rho_j(t),\rho_i(t) \right) \leq \exp\!\left\{ K \, c_1 \, \mathfrak{C}_{m} \left[ \left(tM^{m-1}\right)^{\frac{2}{2+n(m-1)}}  \vee \left(tM^{m-1}\right) \right] \right\} \mathcal{W}_2\! \left( \rho_{j,0} , \rho_{i,0} \right) ,  \\
\mathcal{W}_2 \!\left(\hat{\rho}_j(t),\hat{\rho}_i(t) \right) \leq \exp\!\left\{ K \, c_1 \, \mathfrak{C}_{m} \left[ \left(tM^{m-1}\right)^{\frac{2}{2+n(m-1)}}  \vee \left(tM^{m-1}\right) \right] \right\} \mathcal{W}_2\! \left( \hat{\rho}_{j,0} , \hat{\rho}_{i,0} \right) ,
\end{gathered}
\end{equation}
for every $ t>0 $ and $ i,j \in \mathbb{N} $, whereas the smoothing effect \eqref{smoothing-limit} ensures that 
\begin{equation}\label{smoothing-j}
\left\| \rho_{j}(t) \right\|_{L^\infty\left(\mathbb{M}^n\right)} \vee \left\| \hat{\rho}_{j}(t) \right\|_{L^\infty\left(\mathbb{M}^n\right)} \le C \left( t^{-\frac{n}{2+n(m-1)}} M^{\frac{2}{2+n(m-1)}} +M \right) \qquad \forall t > 0 \, , \ \forall j \in \mathbb{N} \, .
\end{equation}
From \eqref{approx-j} and \eqref{wass-contr-j} we infer that $ \{ \rho_j \}_{j \in \mathbb{N}} $ and $ \{ \hat{\rho}_j \}_{j \in \mathbb{N}} $ are Cauchy sequences in the space $ C([0,T);(\mathscr{M}_2^M(\mathbb{M}^n),\mathcal{W}_2)) $ for every $T>0$, hence they converge to two corresponding curves $ \rho $ and $ \hat{\rho} $, respectively, both in $ C([0,T);(\mathscr{M}_2^M(\mathbb{M}^n),\mathcal{W}_2)) $ for all $ T>0 $. By construction estimates \eqref{smoothing-j} and \eqref{p4-2-ter} (applied to $ \rho \equiv \rho_j$ and $ \hat{\rho} \equiv \hat{\rho}_j $) are preserved at the limit, ensuring the validity of \eqref{smoothing}--\eqref{wass-contr}. We are thus left with proving that $ \rho$ and $\hat{\rho} $ are indeed Wasserstein solutions of \eqref{pme} in the sense of Definition \ref{defsol-w}, i.e.~they comply with \eqref{sol-w1} and \eqref{sol-w2}. Of course it is enough to show it for $ \rho $ only. Since the latter satisfies \eqref{smoothing} and $ \| \rho(t) \|_{L^1(\mathbb{M}^n)} = M $ for all $ t>0 $, the first property in \eqref{sol-w1} is trivially fulfilled. In order to establish the second one and \eqref{sol-w2}, we take advantage of the energy estimate \eqref{eest} applied to each $ \rho \equiv \rho_j $ (with time origin shifted from $ 0 $ to $ \tau \in (0,T) $) combined with \eqref{below-above-prime} and \eqref{smoothing-j}, which yield
\begin{equation}\label{eest-j}
\begin{aligned}
\int_\tau^T \int_{\mathbb{M}^n} \left| \nabla{P(\rho_j)} \right|^2 \d\mathcal{V} \d t + \int_{\mathbb{M}^n} \Psi\!\left( \rho_j(x,T) \right) \d\mathcal{V}(x) \le \, & \frac{c_1}{m+1} \int_{\mathbb{M}^n} \rho_{j}(x,\tau)^{m+1} \, \d\mathcal{V}(x) \\ \le \, & \frac{c_1 \, C^m M}{m+1} \left( \tau^{-\frac{n}{2+n(m-1)}} M^{\frac{2}{2+n(m-1)}} +M \right)^m .
\end{aligned}
\end{equation}
Starting from \eqref{eest-j}, using in a similar way the analogues of \eqref{eest-3-k-last}--\eqref{eest-4-k-last} with $ \rho \equiv \rho_j $ and the time origin shifted from $0$ to $ \tau $, one can reason as in the proof of Proposition \ref{exunimain} to deduce that $ \{ \rho_j \}_{j \in \mathbb{N}} $ converges to $ \rho $ and $ \{ \nabla P(\rho_j) \}_{j \in \mathbb{N}} $ converges to $ \nabla P(\rho) $ weakly in $ L^2(\mathbb{M}^n \times (\tau,T)) $ as $ j \to \infty $, whence the validity of \eqref{sol-w2} upon passing to the limit in the weak formulation satisfied by every $ \rho_j $.

Finally, the uniqueness of Wasserstein solutions is a simple consequence of the uniqueness of weak energy solutions (Proposition \ref{exunimain}) and the continuity in $ (\mathscr{M}_2^M(\mathbb{M}^n),\mathcal{W}_2) $ down to $ t=0 $. Indeed, if $ \rho $ and $ \hat{\rho} $ are two Wasserstein solutions starting from the same initial datum, they can be seen as weak energy solutions starting from the initial data $ \rho(\tau) $ and $ \hat{\rho}(\tau) $, respectively, for every $ \tau>0 $. In particular, there holds 
\begin{equation}\label{wass-last}
\mathcal{W}_2 \!\left(\rho(t),\hat{\rho}(t) \right) \leq \exp\!\left\{ K \, c_1 \, \mathfrak{C}_{m} \left[ \left(tM^{m-1}\right)^{\frac{2}{2+n(m-1)}}  \vee \left(tM^{m-1}\right) \right] \right\} \mathcal{W}_2\! \left( \rho(\tau) , \hat{\rho}(\tau) \right) \qquad \forall t>\tau > 0 \, ,
\end{equation}    
whence $  \mathcal{W}_2 \!\left(\rho(t),\hat{\rho}(t) \right) = 0 $ upon letting $ \tau \downarrow 0 $ in \eqref{wass-last}.
\end{proof}

\subsection{The compact case} \label{compact}

If $ \mathbb{M}^n $ is a \emph{compact} manifold, the construction of the Wasserstein solutions of \eqref{pme} performed in Subsection \ref{weak-sol}  is in fact easier with respect to the one performed in the noncompact case. Indeed, in the proofs of Propositions \ref{exunimain} and \ref{smooth-approx}, there is no need to fill $ \mathbb{M}^n $ with a regular exhaustion $ \{ D_k \}_{k \in \mathbb{N}} $: it is enough to solve the approximate problems (i.e.~the ones associated with the nonlinearity $ P_\varepsilon $) directly on the compact manifold, where integrations by parts are always justified. Moreover, mass conservation is plain because space-constant functions are admissible test functions in the weak formulation \eqref{sol-p2}. The compact-support property established in Proposition \ref{compact-support} is clearly for free.

As concerns the variational framework considered in Subsection \ref{sec:Variational solutions}, some less trivial modifications have to be implemented. That is, one defines the space
\begin{equation*}\label{def:V_E'}
\V_{\mathcal{E}}':= \left\{\ell\in \mathbb{V}': \, |_{\mathbb{V}'}\langle \ell,f \rangle_{\mathbb{V}}|\leq C\sqrt{\mathcal{E}(f)} \ \, \text{for every} \ f\in \mathbb{V} \, , \ \text{for some } C>0 \right\} 
\end{equation*}
endowed with the norm
$$
\left \| \ell \right\|_{\V_{\mathcal{E}}'} := \sup_{f \in \mathbb{V} \, : \ \mathcal{E}(f) \not \equiv 0} \frac{|_{\mathbb{V}'}\langle \ell,f \rangle_{\mathbb{V}}|}{\sqrt{\mathcal{E}(f)}} \, ,
$$
and the space
\begin{equation*}\label{def:H_E'}
\D_{\mathcal{E}}':=\left\{\ell\in \mathbb{D}': \, | _{\mathbb{D}'} \langle \ell,f \rangle_{\mathbb{D}}|\leq C\left\| \Delta f \right\|_{\H} \ \, \text{for every} \ f\in \mathbb{D} \, , \ \text{for some } C>0 \right\},
\end{equation*}
endowed with the norm
$$
\left \| \ell \right\|_{\D_{\mathcal{E}}'} := \sup_{f \in \mathbb{D} \, : \ \Delta f \not \equiv 0} \frac{|_{\mathbb{D}'}\langle \ell,f \rangle_{\mathbb{D}}|}{\left\| \Delta f \right\|_{\H}}  \, .
$$
Upon replacing $\V'$ with $\V_{\mathcal{E}}'$ and $\D'$ with $\D_{\mathcal{E}}'$, respectively, the results stated in Subsections \ref{sec:Variational solutions} and \ref{noncomp} continue to hold. Here we refer again to the machinery developed in \cite{AMS}.
 
We point out that, in view of the standard Dirichlet form we have dealt with, the only reason why $\V_{\mathcal{E}}'$ and $\D_{\mathcal{E}}'$ do not coincide with $\V'$ and $\D'$, respectively, is that in the compact case the kernel of the Dirichlet energy functional $\mathcal{E}:\H\rightarrow [0,+\infty]$ coincides with the set of constant functions, hence is nontrivial. In fact $ \V_{\mathcal{E}}' $ and $\D_{\mathcal{E}}'$ turn out to be identified as those elements of $ \V' $ and $ \D' $, respectively, that vanish on constant functions. On the contrary, in the noncompact case there holds
$$
\mathcal{E}(f) = 0 \ \ \text{and } \ f\in \H \quad \Longrightarrow \quad f = 0
$$
provided $ \mathcal{V}(\mathbb{M}^n) = \infty $, which is always true if \eqref{Sob} is satisfied.

\subsection{Optimality for small times}\label{opt-small}

In what follows, even if the discussion could in principle be made more general, we will restrict ourselves to $ \mathbb{M}^n = \mathbb{H}^n_K $, that is the $n$-dimensional hyperbolic space of constant {sectional} curvature $\mathrm{Sec} = - K$. The key starting point to show optimality is the next delicate result, inspired by \cite[Proposition 6]{Ol}. 

\begin{lemma}\label{lem:asintotico Ollivier}
Let $ K>0 $, $x\in \mathbb{H}^n_K$ and $v$ be a unit tangent vector of $ T_x \, \mathbb{H}^n_K $. Let $r,\delta>0$. Denote by $v^{\perp} \subset T_x \, \mathbb{H}^n_K $ the orthogonal subspace to $v$ and set $\mathrm{E}:=\exp_x v^{\perp}\subset \mathbb{H}^n_K$. Let $w \in v^{\perp} $ be another unit tangent vector. Consider the point $y:=\exp_x \delta v$ and set $ w' := I^x_y (w)  $, where $ I^x_y : T_x \, \mathbb{H}^n_K \to T_y \, \mathbb{H}^n_K $ stands for the parallel-transport map along the geodesic $t\mapsto \exp_x tv$. Then 
\begin{equation}\label{dist-2}
\mathsf{d}\!\left(\exp_yr w' , \mathrm{E} \right)=\delta\!\left(1 + \frac{K}{2} \, r^2 +O\!\left(r^3+\delta r^2\right)\right) \qquad \text{as } (r,\delta)\rightarrow 0 \, .
\end{equation}
More in general, if $ u' \in T_y \, \mathbb{H}^n_K $ is a unit tangent vector, then
\begin{equation}\label{dist-3}
\begin{aligned}
\mathsf{d}\!\left(\exp_yr u' , \mathrm{E} \right) = &  \, \delta\!\left(1 + \frac{K}{2} \,  r^2 \sin^2 \! \alpha(u',I^x_y (v)) +O(r^3+\delta r^2)\right) \\
& + r \cos \alpha(u',I^x_y (v)) + O(r^3) \qquad \text{as } (r,\delta)\rightarrow 0  \, ,
\end{aligned}
\end{equation}
where $ \alpha(\cdot,\cdot) \in[0,\pi] $ denotes the angle between unit vectors in $ T_y \, \mathbb{H}^n_K $. In all
the above identities, the remainder terms $ O(\cdot) $ can be considered independent of the chosen tangent vectors.
\end{lemma}
\begin{proof}
The expansion of formula \eqref{dist-2} is exactly what is proved in \cite[Section 8]{Ol}. Consider now a general unit tangent vector $u' \in T_y \, \mathbb{H}^n_K $. Let us denote by $P_{v}(ru')$ and $P_{v^{\perp}}(ru')$ the projections, in the tangent space $T_y \, \mathbb{H}^n_K $, of the vector $ru'$ on the subspace generated by $ I^x_y(v) $ and on its orthogonal subspace $ I^x_y(v^{\perp})$, respectively. Clearly, we have:
\begin{equation}\label{eq:proj}
\left|P_{v}(ru')\right|=\left|r\cos \alpha(u',I^x_y (v))\right| , \qquad \left|P_{v^{\perp}}(ru')\right| = \left|r\sin \alpha(u',I^x_y (v)) \right|
\end{equation}
and 
\begin{equation}\label{eq:proj-2}
P_{v}(ru')\perp P_{v^{\perp}}(ru') \, , \qquad P_{v}(ru')+ P_{v^{\perp}}(ru')=ru'\, .
\end{equation}
In agreement with \cite{GA}, we put
$$
\exp_y(P_{v}(ru')\,,P_{v^{\perp}}(ru')):=\exp_{\exp_y \! P_{v}(ru')}\left[I^y_{\exp_y\!P_{v}(ru')} \!\left( P_{v^{\perp}}(ru') \right) \right] .
$$
Thanks to \eqref{eq:proj} and \eqref{eq:proj-2}, we can apply \eqref{dist-2} with $ y $ replaced by $ \exp_y\!P_{v}(ru') $ and $ rw' $ replaced by the vector $ I^y_{\exp_y\!P_{v}(ru')}\!\left(P_{v^{\perp}}(ru') \right) $ (hence $ \delta $ replaced by $ \delta + r\cos \alpha(u',I^x_y (v)) $ and $ r $ replaced by $ | r\sin \alpha(u',I^x_y (v)) | $), which yields 
\begin{equation}\label{eq:proj-3}
\begin{aligned}
& \mathsf{d}\!\left(\exp_y(P_{v}(ru'),P_{v^{\perp}}(ru')), \mathrm{E} \right) \\
= & \, \delta\!\left(1 + \frac{K}{2} \,  r^2 \sin^2 \! \alpha(u',I^x_y (v)) +O(r^3+\delta r^2)\right) + r \cos \alpha(u',I^x_y (v)) + O(r^3) \, .
\end{aligned}
\end{equation} 
In order to establish \eqref{dist-3}, first of all we take advantage of the triangle inequality, so as to obtain
\begin{equation}\label{eq:proj-4}
\left| \mathsf{d}(\exp_yru',\mathrm{E}) - \mathsf{d}\big(\exp_y(P_{v}(ru')\,,P_{v^{\perp}}(ru')) , \mathrm{E} \big) \right| \le \mathsf{d} \big(\exp_y ru',\exp_y(P_{v}(ru')\,,P_{v^{\perp}}(ru'))\big) \, .
\end{equation}
Still in agreement with \cite{GA}, we denote by $ h_y(P_{v}(ru')\, , P_{v^{\perp}}(ru')) $ the unique vector of $ T_y \, \mathbb{H}^n_K $ such that 
$$
\exp_y\!\left( h_y(P_{v}(ru')\,, P_{v^{\perp}}(ru')) \right) = \exp_y(P_{v}(ru')\,,P_{v^{\perp}}(ru')) \, ; 
$$
on the other hand, by virtue of \cite[formula (3)]{GA} there holds
$$
\left| h_y(P_{v}(ru')\,, P_{v^{\perp}}(ru')) - r u' \right| = O(r^3) \, ,
$$
so that 
\begin{equation}\label{eq:last-formula}
\mathsf{d} \big(\exp_yru',\exp_y(P_{v}(ru')\,,P_{v^{\perp}}(ru'))\big)=O(r^3)
\end{equation}
upon recalling the well-known fact that the Riemannian distance locally can be replaced by the Euclidean distance up an error of order $ O(r^3) $ (see e.g.~\cite[formula (14.1)]{Vil}). Estimate \eqref{dist-3} then follows from \eqref{eq:proj-3}, \eqref{eq:proj-4} and \eqref{eq:last-formula}.
\end{proof}

Taking advantage of Lemma \ref{lem:asintotico Ollivier}, we are able to prove a lower bound for the Wasserstein distance between suitable radially-symmetric probability densities in $ \mathbb{H}^n_K $. 

\begin{lemma}\label{th:optimality}
Let $ K>0 $ and $ \{ \rho^\epsilon \}_{\epsilon \in (0,1)} $ be a family of (continuous) radially-symmetric probability densities in $ \mathbb{H}^n_K $, i.e.~each $ \rho^\epsilon : [0,+\infty) \mapsto [0,+\infty) $ satisfies
\begin{equation}\label{eq-sinh}
\frac{\left| \mathbb{S}^{n-1} \right|}{{K}^{\frac{n-1}{2}}} \int_{0}^{+\infty} \rho^\epsilon(r) \sinh\!\left(\sqrt{K} r\right)^{n-1} \mathrm{d}r = 1 \qquad \forall \epsilon \in (0,1) \,.
\end{equation}
Suppose in addition that there exist some $ \theta \in (0,1) $ and constants $ C_1,C_2>0 $ (independent of $ \epsilon $) such that
\begin{equation}\label{dis: condizione su densita}
\frac{C_1}{\epsilon^n} \, \chi_{[0,\theta \epsilon]}(r) \leq \rho^\epsilon(r) \leq \frac{C_2}{\epsilon^n} \, \chi_{[0,\epsilon]}(r) \qquad \forall \epsilon \in (0,1) \, , 	\quad \forall r \ge 0 \,.
\end{equation} 
Let $ x,y \in \mathbb{H}^n_K $ with $ \mathsf{d}(x,y)=:\delta>0 $ and consider the probability measures $ \mu_x^{\epsilon} $ and $\mu_y^{\epsilon}$ obtained by centering $ \rho^\epsilon $ at $x$ and $y$, respectively. That is, put $ \mu_x^{\epsilon} :=  \rho^\epsilon( \mathsf{d}(\cdot,x) ) \mathcal{V} \in \P(\mathbb{H}^n_K)$ and $  \mu_y^{\epsilon} :=  \rho^\epsilon( \mathsf{d}(\cdot,y) ) \mathcal{V} \in \P(\mathbb{H}^n_K)$. Then there exist constants $ \overline{\delta}=\overline{\delta}( n,K,C_1,C_2,\theta)>0 $ and $ \kappa=\kappa(n,C_1,C_2,\theta)>0 $ such that, if $ \delta \in (0,\overline{\delta}) $, 
\begin{equation}\label{dist-basso-W1}
\mathcal{W}_2(\mu_x^{\epsilon},\mu_y^{\epsilon})\geq \delta\left( 1 + \kappa \, K \, \epsilon^2 \right) \qquad \forall \epsilon \in (0,\overline{\epsilon}) \, ,
\end{equation}
where $ \overline{\epsilon}=\overline{\epsilon}(\delta,n,K,C_1,C_2,\theta) \in (0,1) $.
\end{lemma}
\begin{proof}
For simplicity we assume $ K=1 $ and set $ \mathbb{H}^n := \mathbb{H}^n_{1} $, since the modifications in order to deal with a general $ K>0 $ are inessential. So, let $v \in T_x \, \mathbb{H}^n$ be the unit vector such that $\exp_x \delta v =y$. Let $i:\mathbb{R}^n\rightarrow T_x\,\mathbb{H}^n$ be an isometric isomorphism that preserves orientation. As in Lemma \ref{lem:asintotico Ollivier}, we denote by $I^x_y$ the parallel-transport map between $T_x\,\mathbb{H}^n$ and $T_y\,\mathbb{H}^n$ along the geodesic $ t \mapsto \exp_x t v $. We then define the maps $\varphi_x:\mathbb{R}^n\rightarrow \mathbb{H}^n$ and $\varphi_y:\mathbb{R}^n\rightarrow \mathbb{H}^n$ as follows:
$$
\varphi_x := \exp_x \circ \, i \, , \qquad \varphi_y := \exp_y \circ \, I^x_y \circ i \, .
$$
First of all, we normalize $ \rho^\epsilon $ in such a way that it is a probability measure on $ \mathbb{R}^n $, namely we set
$$
{\rho}^\epsilon_E(r) := h(\epsilon) \, \rho^\epsilon(r) \qquad \forall r \ge 0
$$
with 
\begin{equation}\label{def-h}
h(\epsilon) := \frac{1}{1- \left| \mathbb{S}^{n-1} \right| \int_0^\epsilon \rho^\epsilon(r) \left( \sinh(r)^{n-1} - r^{n-1} \right) \mathrm{d}r} = 1+O(\epsilon^2) \, ,
\end{equation} 
where we used \eqref{eq-sinh} and \eqref{dis: condizione su densita}. Hence we put $ \mu^\epsilon_E := {\rho}^\epsilon_E(|\cdot|) \mathscr{L}^n $, the symbol $\mathscr{L}^n$ standing for the Lebesgue measure on $ \mathbb{R}^n $. Now we push forward the probability measure $ \mu^\epsilon_E $ on $ \mathbb{H}^n $ by means of the maps $ \varphi_x $ and $ \varphi_y $: 
\begin{equation}\label{pushforward}
\hat{\mu}_x^\epsilon := (\varphi_x)_{\sharp} \, \mu^\epsilon_E \, , \qquad \hat{\mu}_y^\epsilon := (\varphi_y)_{\sharp} \, \mu^\epsilon_E \, .
\end{equation} 
It is possible to show that $ \hat{\mu}_x^\epsilon $ and $\hat{\mu}_y^\epsilon$ are absolutely continuous w.r.t.~to $ {\mu}_x^\epsilon $ and ${\mu}_y^\epsilon$, respectively, in a quantitative way; more precisely, there exist bounded functions $ f_x^\epsilon : \mathbb{H}^n \to \mathbb{R} $ and $ f_y^\epsilon : \mathbb{H}^n \to \mathbb{R} $ such that
\begin{equation}\label{radon-nykodim} 
\mathrm{d} {\mu}_x^\epsilon = \left( 1 + \epsilon^2 f_x^\epsilon \right) \mathrm{d} \hat{\mu}_x^\epsilon \, , \qquad \mathrm{d} {\mu}_y^\epsilon = \left( 1 + \epsilon^2 f_y^\epsilon \right) \mathrm{d} \hat{\mu}_y^\epsilon
\end{equation}
and 
\begin{equation}\label{radon-nykodim-zero} 
\int_{\mathbb{H}^n} f_x^\epsilon \,  \mathrm{d} \hat{\mu}_x^\epsilon = \int_{\mathbb{H}^n} f_y^\epsilon \,  \mathrm{d} \hat{\mu}_y^\epsilon = 0 \, .
\end{equation}
Indeed, by construction $ \varphi_x $ and $ \varphi_y $ preserve radial lengths and angles. As a consequence, both $ \hat{\mu}_x^\epsilon $ and $ \hat{\mu}_y^\epsilon $ are represented on $ \mathbb{H}^n $ by the same radial density $ \hat{\rho}^\epsilon $ via the relation 
\begin{equation*}\label{ch-var}
\hat{\rho}^\epsilon(r) \, \sinh(r)^{n-1} = \rho_E^\epsilon(r) \, r^{n-1} = h(\epsilon) \, \rho^\epsilon(r) \, r^{n-1} \qquad \forall r \in (0,\epsilon) \, ,
\end{equation*}
whence 
\begin{equation*}\label{ch-var-bis}
\begin{gathered}
\rho^\epsilon(r) =  \frac{\sinh(r)^{n-1}}{h(\epsilon)\,r^{n-1}} \, \hat{\rho}^\epsilon(r) = \left( 1 + \epsilon^2 \, \frac{\sinh(r)^{n-1} - h(\epsilon) \, r^{n-1}}{\epsilon^2 \, h(\epsilon) \, r^{n-1}} \right) \hat{\rho}^\epsilon(r) =: \left( 1 + \epsilon^2 f^\epsilon(r) \right) \hat{\rho}^\epsilon(r)  \\ \forall r \in (0,\epsilon)
\end{gathered}
\end{equation*}
and therefore \eqref{radon-nykodim} holds with $ f_x^\epsilon(\cdot) = f^\epsilon(\mathsf{d}(\cdot,x)) $ and $ f_y^\epsilon(\cdot) = f^\epsilon(\mathsf{d}(\cdot,y)) $. Note that, in view of \eqref{def-h} and a standard Taylor expansion of $ \sinh(r) $, the function $ f^\epsilon $ is uniformly bounded by a constant that depends only on $ n $ and $C_2$. On the other hand, identity \eqref{radon-nykodim-zero} just follows by the fact that $ {\mu}_x^\epsilon $, $\hat{\mu}_x^\epsilon$, $ {\mu}_y^\epsilon $, $ \hat{\mu}_y^\epsilon $ are all probability measures.
  
Let $ \mathrm{E}_0 $ and $ \mathrm{E}_1 $ be the two disjoint, open, connected components in $ \mathbb{H}^n $ separated by $ \mathrm{E} $, the latter set being defined as in Lemma \ref{lem:asintotico Ollivier}. Assume for convenience that $ \mathrm{E}_1 $ contains the point $y$. In order to prove \eqref{dist-basso-W1}, as in \cite[Section 8]{Ol} we choose the following $1$-Lipschitz function ${g}:\mathbb{H}^n\rightarrow \mathbb{R}$: 
$$
{g}(z):=
\begin{cases} & \mathsf{d}(z,\mathrm{E}) \qquad \text{if } z \in \mathrm{E}_1 \, , \\
-\hspace{-3mm}& \mathsf{d}(z,\mathrm{E}) \qquad \text{otherwise} \, .
\end{cases}
$$
Upon recalling the duality formula \eqref{dualita W1} along with \eqref{order-wass} and \eqref{radon-nykodim}, we obtain: 
\begin{equation}\label{W1a}
\begin{aligned}
\mathcal{W}_2(\mu_x^\epsilon,\mu_y^\epsilon) \geq & \, \mathcal{W}_1(\mu_x^\epsilon,\mu_y^\epsilon) \\
\geq & \int_{\mathbb{H}^n} g(z) \left( 1 + \epsilon^2 f_y^\epsilon(z) \right) \mathrm{d} \hat{\mu}_y^\epsilon(z)-\int_{\mathbb{H}^n} g(z) \left( 1 + \epsilon^2 f_x^\epsilon(z) \right) \mathrm{d} \hat{\mu}_x^\epsilon(z) \, .
\end{aligned}
\end{equation}
Since $ {\mu}_x^\epsilon $ is represented by a radially-symmetric density about $x$ and $ \mathbb{H}^n $ also has a radially-symmetric structure (about any point), by the definition of $g$ it is not difficult to check that in fact 
\begin{equation}\label{W1b}
\int_{\mathbb{H}^n} g(z) \, \mathrm{d}\mu_x^\epsilon(z) = \int_{\mathbb{H}^n} g(z) \left( 1 + \epsilon^2 f_x^\epsilon(z) \right) \mathrm{d} \hat{\mu}_x^\epsilon(z)  = 0 \, ,
\end{equation}
therefore we can focus on the first integral. By virtue of \eqref{pushforward}, we have: 
\begin{equation}\label{W1c}
\int_{\mathbb{H}^n} g(z) \left( 1 + \epsilon^2 f_y^\epsilon(z) \right) \mathrm{d} \hat{\mu}_y^\epsilon(z) = \int_{\mathbb{R}^n} {g}\!\left(\varphi_y(q)\right)  \left( 1 + \epsilon^2 f_y^\epsilon(\varphi_y(q)) \right) \mathrm{d}\mu^\epsilon_E(q) \, ;
\end{equation}
on the other hand, thanks to \eqref{dist-3} and the fact that $ \mu^\epsilon_E $ is supported in the Euclidean ball $ B_\epsilon $ centered at the origin, we can write 
\begin{equation}\label{W1d}
\begin{aligned}
& \int_{\mathbb{R}^n} {g}\!\left(\varphi_y(q)\right)  \left( 1 + \epsilon^2 f_y^\epsilon(\varphi_y(q)) \right) \mathrm{d}\mu^\epsilon_E(q) \\
= &  \int_{B_\epsilon } \left[ \delta\!\left(1+\frac{|q|^2 - \left( q \cdot p_v \right)^2 }{2}+O(|q|^3+\delta |q|^2)\right) + q \cdot p_v + O(|q|^3) \right] \left( 1 + \epsilon^2 f^\epsilon(|q|) \right) \rho^\epsilon_E(|q|) \, \mathrm{d}q \, , \\
\end{aligned}
\end{equation}
where $ p_v := i^{-1}\, (v) $. Clearly, by symmetry, the middle term involving $ q \cdot p_v $ vanishes when integrated against any radial density. Hence, thanks to \eqref{dis: condizione su densita} (still the right-hand inequality) and \eqref{radon-nykodim-zero}, from \eqref{W1d} we can infer that 
\begin{equation*}\label{W1e}
\int_{\mathbb{R}^n} {g}\!\left(\varphi_y(q)\right)  \left( 1 + \epsilon^2 f_y^\epsilon(\varphi_y(q)) \right) \mathrm{d}\mu^\epsilon_E(q) = \delta \left[ 1 + \frac{n-1}{2n} \int_{B_\epsilon} |q|^2 \rho^\epsilon_E(|q|) \, \mathrm{d}q + O\!\left(\epsilon^3+\delta \epsilon^2\right) \right] + O(\epsilon^3) \, .
\end{equation*}
In view of the left-hand inequality in \eqref{dis: condizione su densita}, there exists a constant $ \kappa > 0 $ as in the statement such that 
\begin{equation}\label{W1f}
\int_{\mathbb{R}^n} {g}\!\left(\varphi_y(q)\right)  \left( 1 + \epsilon^2 f_y^\epsilon(\varphi_y(q)) \right) \mathrm{d}\mu^\epsilon_E(q) \ge \delta \left[ 1 + 3 \, \kappa \, \epsilon^2 + O\!\left(\epsilon^3+\delta \epsilon^2\right) \right] + O(\epsilon^3) \, .
\end{equation}
Upon collecting \eqref{W1a}, \eqref{W1b}, \eqref{W1c} and \eqref{W1f}, the thesis follows by choosing $ \overline{\delta} $ so small that $ \left| O(\delta \epsilon^2) \right| \le \kappa \, \epsilon^2 $ for all $ \delta \in (0,\overline{\delta}) $ and $ \overline{\epsilon} $ so small that $ \left| \delta O(\epsilon^3) \right| + \left| O(\epsilon^3) \right| \le \kappa \, \delta \, \epsilon^2 $ for all $ \epsilon \in (0,\overline{\epsilon}) $ and all $ \delta \in (0,\overline{\delta}) $. 
\end{proof}

\begin{proof}[Proof of Theorem \ref{optimal}]
Let $ M=1 $. Thanks to \cite[Theorem 1.1]{VazHyp}, we know that $ \rho(\cdot,t) $ and $ \hat{\rho}(\cdot,t) $ are represented by the same radial density centered at $x$ and $y$, respectively. That is, $ \rho(\cdot,t) = \tilde{\rho}(\mathsf{d}(\cdot,x),t) $ and $ \hat{\rho}(\cdot,t) = \tilde{\rho}(\mathsf{d}(\cdot,y),t) $ for a suitable continuous, bounded, radially-nonincreasing  family of densities $ (r,t): \mathbb{R}^+ \times \mathbb{R}^+ \mapsto \tilde{\rho}(r,t) $. First of all we observe that, since $ \mathbb{H}^n_K $ is a Cartan-Hadamard manifold, $ \tilde{\rho}(r,t) $ lies below the \emph{Euclidean} Barenblatt solution $ \tilde{\rho}_E(r,t) $, see \cite[Remark 2.12]{GMPrm} and \cite[Introduction]{VazHyp}. This means that there exist constants $ D=D(n,m)>0 $ and $ k=k(n,m)>0 $ such that 
\begin{equation}\label{eq-barenblatt-euclidee}
\tilde{\rho}(r,t) \le t^{-\frac{n}{2+n(m-1)}} \left( D - k \, r^2 \, t^{-\frac{2}{2+n(m-1)}}  \right)_+^{m-1} =: \tilde{\rho}_E(r,t) \qquad \forall (r,t) \in \mathbb{R}^+ \times \mathbb{R}^+ \, .
\end{equation}
In particular, 
\begin{equation}\label{eq-barenblatt-euclidee-bis}
\tilde{\rho}(r,t) \le \frac{D^{m-1}}{t^{\frac{n}{2+n(m-1)}}} \, \chi_{\left[0, A(t)\right]}(r) \quad \forall (r,t) \in \mathbb{R}^+ \times \mathbb{R}^+ \, , \qquad A(t):= \sqrt{\tfrac{D}{k}} \, t^{\frac{1}{2+n(m-1)}} \, .
\end{equation} 
Now let 
$$ I(t) := \inf_{ r \in \left[0 , \frac{A(t)}{2}  \right] } \tilde{\rho}(r,t) \qquad \forall t>0 \,. $$ 
By mass conservation, \eqref{eq-barenblatt-euclidee} and the fact that $ \tilde{\rho}(\cdot,t)  $ is nonincreasing, we can deduce the following:
\begin{equation}\label{ineq-mass}
\begin{aligned}
 \frac{1}{\left| \mathbb{S}^{n-1} \right|} = & \, K^{-\frac{n-1}{2}} \int_0^{\frac{A(t)}{2}} \tilde{\rho}(r,t) \sinh\!\left( \sqrt{K} r \right)^{n-1} \mathrm{d}r + K^{-\frac{n-1}{2}} \int_{\frac{A(t)}{2}}^{A(t)} \tilde{\rho}(r,t) \sinh\!\left( \sqrt{K} r \right)^{n-1} \mathrm{d}r \\
 \le & \, K^{-\frac{n-1}{2}} \int_0^{\frac{A(t)}{2}} \tilde{\rho}_E(r,t) \sinh\!\left( \sqrt{K} r \right)^{n-1} \mathrm{d}r  + K^{-\frac{n-1}{2}} \, I(t)  \, \int_{\frac{A(t)}{2}}^{A(t)} \sinh\!\left( \sqrt{K} r \right)^{n-1} \mathrm{d}r \\
  = & \left[ \frac{\lambda}{\left| \mathbb{S}^{n-1} \right|}  + I(t) \, C \, t^{\frac{n}{2+n(m-1)}} \right]  \left[ 1+ O\!\left( t^{\frac{2}{2+n(m-1)}} \right) \right] ,
\end{aligned}
\end{equation}
where
$$
\lambda := \left| \mathbb{S}^{n-1} \right| \int_0^{\frac{1}{2}\sqrt{\frac{D}{k}}} \tilde{\rho}_E(r,1) \, r^{n-1} \, \mathrm{d}r < 1 \, , \qquad C :=  \int_{\frac{1}{2}\sqrt{\frac{D}{k}}}^{\sqrt{\frac{D}{k}}} r^{n-1} \, \mathrm{d}r > 0 \, .
$$
Note that in the last passage we have exploited the scaling properties of $ \tilde{\rho}_E $. From \eqref{ineq-mass} and the definition of $ I(t) $, it is therefore apparent that there exist constants $ D_1 = D_1(n,m)>0 $ and $ t_1 = t_1 (n,K,m)> 0 $ such that 
\begin{equation}\label{eq-barenblatt-euclidee-below}
\tilde{\rho}(r,t) \ge \frac{D_1}{t^{\frac{n}{2+n(m-1)}}} \, \chi_{\left[0, \frac{A(t)}{2} \right]}(r) \qquad \forall (r,t) \in \mathbb{R}^+ \times (0,t_1) \, .
\end{equation} 
Hence, in order to estimate $ \mathcal{W}_2 \!\left(\rho(t),\hat{\rho}(t) \right) $ from below, we are in position to apply Lemma \ref{th:optimality}. Indeed, if we set $ \epsilon \equiv A(t) $ and $ \rho^\epsilon \equiv \tilde{\rho}(\cdot,t) $, then by virtue of \eqref{eq-barenblatt-euclidee-bis} and \eqref{eq-barenblatt-euclidee-below} we can claim that \eqref{dis: condizione su densita} is satisfied with $ \theta = 1/2  $ and suitable positive constants $ C_1,C_2 $ depending only on $ n $ and $m$, provided $ \epsilon < A(t_1) $ (condition \eqref{dis: condizione su densita} is required to hold for $ \epsilon \in (0,1) $ only for convenience). Estimate \eqref{optimal-delta} for $ M=1 $ is just \eqref{dist-basso-W1}, upon exploiting the above relation between $t$ and $\epsilon$, along with the trivial identity $ \mathcal{W}_2 \!\left(\delta_x,\delta_y\right) =  \mathsf{d}(x,y) $. 

In order to deal with a general mass $ M>0 $, it is enough to notice that $ M \rho(tM^{m-1}) $ and $ M \hat{\rho}(tM^{m-1}) $ are still solutions of \eqref{pme} starting from $ M \delta_x $ and $ M \delta_y $, respectively (recall that $ \cW^2 $ is proportional to the mass).
\end{proof}

\bigskip 

\noindent\textbf{Acknowledgment.} The second author is partially supported by the GNAMPA Project 2018 (Italy) ``Problemi Analitici e Geometrici Associati a EDP Non-Lineari Ellittiche e Paraboliche''. 
The first and third author are supported by the GNAMPA Project 2019 (Italy) ``Trasporto Ottimo per Dinamiche con Interazione''. The authors are grateful to Prof.~Giuseppe Savar\'e for fruitful discussions.
They would also like to thank the anonymous referee for his/her very careful reading of the paper and for the valuable suggestions.


\medskip

\begin{thebibliography}{100}

\bibitem{AGS} L. Ambrosio, N. Gigli, G. Savar\'e, ``Gradient Flows in Metric Spaces and in the Space of Probability Measures''. Second Edition. Lectures in Mathematics ETH Z\"urich. Birkh\"auser Verlag, Basel, 2008.

\bibitem{AGSheat} L. Ambrosio, N. Gigli, G. Savar\'e, \emph{Calculus and heat flow in metric measure spaces and applications to spaces with {R}icci bounds from below}, Invent. Math. 195 (2014), 289--391.

\bibitem{AGS2} L. Ambrosio,  N. Gigli, G. Savar\'e, \emph{Bakry-\'Emery curvature-dimension condition and Riemannian Ricci curvature bounds}, Ann. Probab. 43 (2015), 339--404.

\bibitem{AMS} L. Ambrosio, A. Mondino, G. Savar\'e, \emph{Nonlinear diffusion equations and curvature conditions in metric measure spaces}, {Mem.~Amer.~Math.~Soc, Vol. 262, 1270 (2019).} 

\bibitem{BCLS} D. Bakry, T. Coulhon, M. Ledoux, L. Saloff-Coste, \emph{Sobolev inequalities in disguise}, Indiana Univ. Math. J. 44 (1995), 1033--1074.

\bibitem{BE} D. Bakry, M. \'Emery, \emph{Diffusions hypercontractives (French) [Hypercontractive diffusions]}, S\'eminaire de probabilit\'es, XIX, 1983/84, 177--206, Lecture Notes in Math., 1123, Springer, Berlin, 1985. 

\bibitem{BGL} D. Bakry, I. Gentil, M. Ledoux, ``Analysis and Geometry of Markov Diffusion Operators''. Grundlehren der Mathematischen Wissenschaften [Fundamental Principles of Mathematical Sciences], 348. Springer, Cham, 2014. 

\bibitem{BS} D. Bianchi, A.G. Setti, \emph{Laplacian cut-offs, porous and fast diffusion on manifolds and other applications}, Calc. Var. Partial Differential Equations 57 (2018), Art. 4, 33 pp.

\bibitem{BC} F. Bolley, J.A. Carrillo, \emph{Nonlinear diffusion: geodesic convexity is equivalent to Wasserstein contraction}, Comm. Partial Differential Equations 39 (2014), 1860--1869.

\bibitem{BGV} M. Bonforte, G. Grillo, J.L. V\'azquez, \emph{Fast diffusion flow on manifolds of nonpositive curvature}, J. Evol. Equ. 8 (2008), 99--128.

\bibitem{CMV} J.A.  Carrillo, R.J. McCann, C. Villani, \emph{Contractions in the 2-Wasserstein length space and thermalization of granular media}, Arch. Ration. Mech. Anal. 179 (2006), 217--263.

\bibitem{CH} T. Coulhon, D. Hauer, \emph{Regularisation effects of nonlinear semigroups}, to appear in BCAM Springer Briefs, preprint arXiv: \url{https://arxiv.org/abs/1604.08737}.

\bibitem{DS}  S. Daneri, G. Savar\'e, \emph{Eulerian calculus for the displacement convexity in the Wasserstein distance}, SIAM J. Math. Anal. 40 (2008), 1104--1122.

\bibitem{E} M.  Erbar, \emph{The heat equation on manifolds as a gradient flow in the Wasserstein space}, Ann. Inst. Henri Poincar\'e Probab. Stat. 46 (2010), 1--23.

\bibitem{Foote} R.L. Foote, \emph{Regularity of the distance function}, Proc. Amer. Math. Soc. 92 (1984), 153--155. 

\bibitem{FM} A.R. Fotache, M. Muratori, \emph{Smoothing effects for the filtration equation with different powers}, J. Differential Equations 263 (2017), 3291--3326. 

\bibitem{FP} N. Fournier, B. Perthame,  \emph{Monge-Kantorovich distance for PDEs: the coupling method}, {to appear in EMS Surveys in Mathematical Sciences}, preprint arXiv: \url{https://arxiv.org/abs/1903.11349}.

\bibitem{GA} A.V. Gavrilov, \emph{The double exponential map and covariant derivation}, Siberian Math. J. 48 (2007), 56--61.

\bibitem{GMP13} G. Grillo, M. Muratori, M.M. Porzio, \emph{Porous media equations with two weights: smoothing and decay properties of energy solutions via Poincar\'e inequalities}, Discrete Contin. Dyn. Syst. 33 (2013), 3599--3640. 

\bibitem{GMPpures} G. Grillo, M. Muratori, F. Punzo, \emph{The porous medium equation with large initial data on negatively curved Riemannian manifolds}, J. Math. Pures Appl. 113 (2018), 195--226.

\bibitem{GMPrm} G. Grillo, M. Muratori, F. Punzo, \emph{The porous medium equation with measure data on negatively curved Riemannian manifolds}, J.~Eur.~Math.~Soc.~(JEMS) 20 (2018), 2769--2812. 

\bibitem{GMV1} G. Grillo, M. Muratori, J.L. V\'azquez, \emph{The porous medium equation on Riemannian manifolds with negative curvature. The large-time behaviour}, Adv. Math. 314 (2017), 328--377.

\bibitem{GMV2} G. Grillo, M. Muratori, J.L. V\'azquez, \emph{The porous medium equation on Riemannian manifolds with negative curvature: the superquadratic case}, Math. Ann. 373 (2019), 119--153. 

\bibitem{Heb} E. Hebey, ``Nonlinear Analysis on Manifolds: Sobolev Spaces and Inequalities'', Courant Lecture Notes in Mathematics, 5. New York University, Courant Institute of Mathematical Sciences, New York; American Mathematical Society, Providence, RI, 1999.

\bibitem{JKO} R. Jordan, D. Kinderlehrer, F. Otto, \emph{The variational formulation of the Fokker-Planck equation}, SIAM J. Math. Anal. 29 (1998), 1--17.

\bibitem{LT}H. Li, G. Toscani, \emph{Long-time asymptotics of kinetic models of granular flows}, Arch. Ration. Mech. Anal. 172 (2004), 407--428.

\bibitem{LiMS}  S. Lisini, E. Mainini, A. Segatti, \emph{A gradient flow approach to the porous medium equation with fractional pressure}, Arch. Ration. Mech. Anal. 227 (2018), 567--606.

\bibitem{Lee} J.M.~Lee, ``Introduction to Smooth Manifolds''. Second Edition. Graduate Texts in Mathematics, 218. Springer, New York, 2013. 

\bibitem{O} F. Otto, \emph{The geometry of dissipative evolution equations: the porous medium equation}, Comm. Partial Differential Equations 26 (2001), 101--174.

\bibitem{Ol} Y. Ollivier, \emph{Ricci curvature of Markov chains on metric spaces}, J. Funct. Anal. 256 (2009), 810--864. 

\bibitem{OT}  S. Ohta, A. Takatsu,  \emph{Displacement convexity of generalized relative entropies}, Adv. Math. 228 (2011), 1742--1787.

\bibitem{OW} F. Otto, M. Westdickenberg, \emph{Eulerian calculus for the contraction in the Wasserstein distance}, SIAM J. Math. Anal. 37 (2005), 1227--1255.

\bibitem{VRS} M.-K. von Renesse, K.-T. Sturm, \emph{Transport inequalities, gradient estimates, entropy, and Ricci curvature}, Comm. Pure Appl. Math. 58 (2005), 923--940. 

\bibitem{Str} R.S. Strichartz, \emph{Analysis of the Laplacian on the complete Riemannian manifold}, J. Funct. Anal. 52 (1983), 48--79.

\bibitem{Stu} K.-T. Sturm, \emph{Convex functionals of probability measures and nonlinear diffusions on manifolds}, 
J. Math. Pures Appl. 84 (2005), 149--168. 

\bibitem{Var} N.Th. Varopoulos, \emph{Small time Gaussian estimates of heat diffusion kernels. I. The semigroup technique}, Bull. Sci. Math. 113 (1989), 253--277. 

\bibitem{VazHeat} J.L. V\'azquez, \emph{Asymptotic behaviour for the heat equation in hyperbolic space}, {to appear in Comm. Analysis and Geometry}, preprint arXiv: \url{https://arxiv.org/abs/1811.09034}.

\bibitem{VazHyp} J.L. V\'azquez, \emph{Fundamental solution and long time behavior of the porous medium equation in hyperbolic space}, J. Math. Pures Appl. 104 (2015), 454--484.

\bibitem{V06} J.L. V\'azquez, ``Smoothing and Decay Estimates for Nonlinear Diffusion Equations. Equations of Porous Medium Type'', Oxford Lecture Series in Mathematics and its Applications, 33. Oxford University Press, Oxford, 2006.

\bibitem{V07} J.L. V\'azquez, ``The Porous Medium Equation. Mathematical Theory'', Oxford Mathematical Monographs, The Clarendon Press, Oxford University Press, Oxford, 2007. 

\bibitem{Vil} C. Villani, ``Optimal Transport, Old and New'', Springer Verlag, Grundlehren der mathematischen Wissenschaften, 2008.


\end{thebibliography}
\end{document}